\documentclass[10pt]{article}
\usepackage{amssymb,amsmath,textcomp}
\usepackage{enumerate}
\usepackage{hyperref}
\usepackage{mathrsfs}
\usepackage{amsthm}
\usepackage{xcolor}
\newtheorem{satz}{Theorem}[section]
\newtheorem{lemma}{Lemma}[section]

\newtheorem{bemerk1}{Remark}[section]

\newtheorem{korollar}{Corollary}[section]
\newtheorem{prop}{Proposition}[section]
\newcommand{\bvr}{\overline{\Phi}(r)}
\newcommand{\iR}{\mathbb{R}}
\newcommand{\iN}{\mathbb{N}}

\newcommand{\oH}{\hspace*{0.39em}\raisebox{0.6ex}{\textdegree}\hspace{-0.72em}H}
\DeclareMathOperator*{\eosc}{ess\,osc}
\DeclareMathOperator*{\esup}{ess\,sup}
\DeclareMathOperator*{\einf}{ess\,inf}

\newcommand{\dd}{d\mu(\alpha)}

\newcommand{\ki}{k_{1}}
\newcommand{\kjr}{k_{1}(\Phi(r))}
\newcommand{\tl}{\tilde{l}}

\newcommand{\epz}{\ep_{0}}

\newcommand{\supp}{\mathop{\mathrm{supp}}\limits}
\newcommand{\vs}{\varsigma}
\newcommand{\ka}{\kappa}

\newcommand{\muad}{\mu(\al) d \al}

\newcommand*{\norm}[1]{\left\Vert{#1}\right\Vert}
\newcommand*{\abs}[1]{\left\vert{#1}\right\vert}
\newcommand*{\om}{\omega}
\newcommand*{\Om}{\Omega}
\newcommand*{\mr}{\mathbb{R}}
\newcommand*{\izi}{\int_{0}^{\infty}}
\newcommand*{\izt}{\int_{0}^{t}}
\newcommand*{\izj}{\int_{0}^{1}}

\newcommand*{\p}{\partial}

\newcommand*{\dm}{D^{(\mu)}}
\newcommand*{\da}{D^{\alpha}}
\newcommand*{\im}{I^{(\mu)}}
\newcommand*{\al}{\alpha}

\newcommand*{\ma}{\mu(\alpha)}
\newcommand*{\mg}{\frac{\mu(\alpha)}{\Gamma(1-\alpha)}}

\newcommand*{\vf}{\varphi}

\newcommand*{\ve}{\varepsilon}
\newcommand{\podst}[2]{\left\{ \begin{array}{ll} #1 \\  #2 \\  \end{array} \right\}}

\def\Re{\operatorname{Re}}
\def\Im{\operatorname {Im}}
\def\arg{\operatorname{arg}}
\def\Arg{\operatorname{Arg}}
\def\divv{\operatorname {div}}

\newcommand{\eqq}[2]{\begin{equation}  #1  \label{#2}\end{equation}    }
\newcommand{\hd}{\hspace{0.2cm}}

\newcommand{\no}{\noindent}
\newcommand{\m}[1]{\mbox{#1}}

\newcommand{\R}{\mathbb{R}}

\newcommand{\jd}{\frac{1}{2}}

\newcommand{\ep}{\varepsilon}

\newcommand{\vp}{\varphi}

\newcommand{\ra}{\partial^{\alpha}}

\newcommand*{\esssup}{\mathop{\mathrm{ess\hspace{0.05cm}sup}}\limits}
\newcommand*{\essinf}{\mathop{\mathrm{ess\hspace{0.05cm}inf}}\limits}

\newcommand*{\gm}{\bar{\gamma}_{-}}
\newcommand*{\gp}{\bar{\gamma}_{+}}
\newcommand*{\gpb}{\gamma_{+}}
\newcommand*{\gmb}{\gamma_{-}}
\newcommand*{\imp}{\int_{\gm}^{\gp}}

\newcommand{\vr}{\Phi(r)}
\newcommand{\vdr}{\Phi(2r)}

\newcommand{\as}{\alpha_{*}}

\newcommand{\eqns}[1]{
\begin{eqnarray*}
\begin{split}
#1
\end{split}
\end{eqnarray*}
}

\newcommand{\eqnsl}[2]{
\begin{equation}
\label{#2}
\begin{split}
#1
\end{split}
\end{equation}
}

\newcommand{\vsi}{\varsigma}
\newcommand{\kaf}{\widetilde{\kappa}}
\newcommand{\muf}{\tilde{\nu}}

\newcommand{\inag}{\int_{0}^{\as}}

\newcommand{\sinpa}{\sin(\pi \al)}
\newcommand{\cospa}{\cos(\pi \al )}
\newcommand{\pal}{p^{\al}}
\newcommand{\ps}{p^{*}}

\newcommand{\naw}[1]{\left(  #1  \right)}

\newcommand{\ti}[1]{\tilde{#1}}
\newcommand{\tss}{t_{**}}
\newcommand{\tit}{\ti{t}}

\newcommand{\vkrl}{\overline{\Phi}(2^{-l}r)}
\newcommand{\vkrll}{\overline{\Phi}(2^{-(l+1)}r)}
\newcommand{\blrj}{B_{2^{-l}r}(x_{1})}
\newcommand{\bllrj}{B_{2^{-(l+1)}r}(x_{1})}
\newcommand{\bllrt}{B_{2^{-(l+1)}\cdot \frac{3}{2}r}(x_{1})}

\newcommand{\tw}{\tilde{w}}
\newcommand{\tv}{\tilde{v}}
\newcommand{\Qdl}{Q(2^{-l})}
\newcommand{\Qdll}{Q(2^{-(l+1)})}
\newcommand{\dml}{2^{-l}r}

\newcommand{\nic}[1]{ }

\newcommand{\Gf}{\tilde{G}}
\newcommand{\lf}{\tilde{l}}
\newcommand{\Pk}{\overline{\Phi}}
\newcommand{\Pkr}{\overline{\Phi}(r)}
\newcommand{\Pkro}{\overline{\Phi}(\rok)}

\newcommand{\tQ}{\tilde{Q}}
\newcommand{\rok}{\tilde{\rho}}

\newcommand{\nuNj}{\nu_{N+1}}
\newcommand{\nuN}{\nu_{N}}

\begin{document}
\begin{center}
{\bf\Large H\"older continuity  of weak solutions to evolution equations with distributed order fractional time derivative}
\renewcommand{\thefootnote}{\fnsymbol{footnote}}
\end{center}
\vspace{1.7em}
\begin{center}
Adam Kubica*, Katarzyna Ryszewska*, Rico Zacher${}^{\diamond}$
\end{center}

\vspace{1cm}
{\footnotesize \noindent \nic{{\bf address:} }*Department of Mathematics and Information Sciences\\
Warsaw University of Technology\\
pl. Politechniki 1, 00-661 Warsaw, Poland \\
Adam.Kubica@pw.edu.pl \\
Katarzyna.Ryszewska@pw.edu.pl \\
\\
${}^{\diamond}$Institute of Applied Analysis \\
University of Ulm \\
89069 Ulm, Germany \\
rico.zacher@uni-ulm.de \\
} \vspace{0.7em}
\begin{abstract}
We study the regularity of weak solutions to evolution equations with distributed order fractional time derivative. We prove a weak Harnack inequality for nonnegative weak supersolutions and H\"older continuity of weak solutions to this problem. Our results substantially generalise analogous known results for the problem with single order fractional time derivative.
\end{abstract}
\vspace{0.7em}
\begin{center}
\end{center}

\section{Introduction and main result}
In this paper we study the local H\"older continuity of weak solutions to the following problem
\begin{equation} \label{MProb}
\partial_{t} [k * (u-u_{0})] -\mbox{div}\,\big(A(t,x)Du\big)=f,\quad t\in (0,T),\,x\in
\Omega,
\end{equation}
where $T>0$, $\Omega$ is a bounded domain in $\iR^N$, $N \geq 1$, $u_0$ is a given initial data, $f$ is a bounded function, $A=(a_{ij})$ is an $\mathbb{R}^{N\times N}$-valued given function and $k$ is the kernel related to the distributed order fractional derivative.  Here $Du$ stands for the gradient w.r.t.\ to the spatial variables and $f_1\ast f_2$ is the
convolution on the positive half-line w.r.t.\ time, that is
$(f_1\ast f_2)(t)=\int_0^t f_1(t-\tau)f_2(\tau)\,d\tau$, $t\ge 0$. In the paper, it is merely assumed that the coefficients $A(t,x)$ are measurable, bounded and uniformly elliptic.

In the classical case, which formally corresponds to $k$ being the Dirac delta distribution so that the integro-differential operator w.r.t.\ time in (\ref{MProb}) becomes the classical time derivative, there is a well known De Giorgi-Nash-Moser regularity theory, see for example \cite{LSU}, \cite{Lm}, \cite{Trud}. Here, we study the problem with memory, 
in particular, $\partial_{t} [k * (u-u_{0})] $ covers such cases as the Caputo fractional derivative, multi-term fractional derivatives or a 'purely  distributed' fractional derivative.

Before describing the structure condition on $k$ we briefly give (only formally) some motivation for equation (\ref{MProb}) from physics. Assume that there exists a function $l\in L_{1,loc}(\iR_+)$ such that $k*l\equiv 1$. Let us further suppose that the (scalar) quantity $u$ is conserved in any (smooth) subdomain $V$ of the domain $\Omega$, i.e.
\eqq{\frac{d}{dt}  \int_{V} u(t,x)dx = - \int_{\partial V}  q(t,x) dS(x) ,}{wwa}
$q$ is the corresponding flux. It is well known that in case $q(t,x)= - A(t,x)D u(t,x)$, i.e.\  Fick's first law, (\ref{MProb}) becomes a classical diffusion equation. However, in many applications, especially in non-homogeneous materials this constitutive law is not appropriate and other forms of the flux are proposed (see for example \cite{Metz}, \cite{Voller}). Let us consider $q(t,x)= - \partial_{t} [ l* A(\cdot,x)D u(\cdot,x)](t) $,
 where $l\not \equiv 1$ (if $l\equiv 1 $, then it is Fick'a law), then  the flux clearly depends on all the values of $D u(\tau,x)$ for $\tau \in (0,t)$. In this case (\ref{wwa}) takes the form
\eqq{\frac{d}{dt}  \int_{V} u\,dx =  \int_{\partial V}  \partial_{t} [ l* (A D u)]\, dS(x) .}{wwb}
Assuming that the functions under consideration are smooth enough we arrive at
\[
\frac{d}{dt}  \int_{V} u\,dx =  \frac{d}{dt} \big[ l*    \int_{\partial V}   AD u \,dS(x)\big].
\]
We next apply the divergence theorem and convolve both sides with $k$ to obtain
\[
k*\frac{d}{dt}  \int_{V} u\, dx =  k*\frac{d}{dt} [ l*    \int_{ V} \divv\left(  AD u  \right)dx],
\]
i.e.
\[
 \int_{ V}\partial_{t} [k * (u-u|_{t=0})]\, dx  = \frac{d}{dt}\big[k*l*\int_{ V} \divv (ADu)\, dx\big].
\]
Since $k*l=1$ and $V $ is arbitrary, we arrive, at least formally, at (\ref{MProb}).


Let us now describe the class of kernels $k$ to be studied in this paper. Let $\{\alpha_{n}\}_{n=1}^{M}$ satisfy
\[
0<\al_{1}<\al_{2}<\dots <\al_{M}<1,
\]
$q_{n}$, $n=1, \dots, M$, be nonnegative numbers, and $w\in L_{1}((0,1))$ be nonnegative. We define the measure $\mu$ on the Borel sets in $\iR$ by
\eqq{d\mu=\sum_{n=1}^{M}q_{n}d\delta(\cdot - \al_{n})+w d \nu_{1},  }{bo1}
where $\delta(\cdot - \al_{n})$ is the Dirac measure at $\alpha_n$ and $\nu_{N}$ denotes the $N$-dimensional Lebesgue measure. Here we allow the first or the second component in the above
representation to vanish, but we always assume that $\mu \not \equiv 0$. Then we define
\begin{equation} \label{kintro}
 k(t):=\izj \frac{t^{-\al} }{\Gamma(1-\al)} \dd,\quad t>0,
\end{equation}
where $\Gamma$ is the Gamma function. Having introduced the kernel $k$, the associated distributed order fractional derivative (of Caputo type) of a sufficiently smooth function $v$ is given by
\[
\dm v:= \partial_{t} [k* (v-v(0))].
\]
We note that the notation $\dm$ coincides with the notation from \cite{decay} and \cite{nasza}.
We also point out that concerning integration w.r.t.\ $\mu$ we use the following convention. If $0\leq a<b\leq 1$, then $\int_{a}^{b} h(\al) \dd =\int_{(a,b]} h(\al) \dd$, but $\int_{a}^{a}h(\al) \dd = h(\al_{n}) q_{n}$ if $a=\al_{n}$ for some $n\in \{ 1,\dots, M\}$ and $0$ elsewhere.

There is a vast literature on diffusion equations with single order fractional time derivative, i.e. $\mu = \delta(\cdot-\al)$ with some $\al\in (0,1)$ and $k(t) = \frac{t^{-\al} }{\Gamma(1-\al)}$, which form an important class of {\em subdiffusion} equations. In fact, they can be used to model diffusive particles
with a mean squared displacement behaving as a multiple of $t^\alpha$ (\cite{Metz}). Diffusion equations with distributed order fractional derivatives
appear in the context of ultra-slow diffusion, see \cite{KochDO, SCK}. Here, the mean squared displacement might only have a logarithmic growth
in time.

\nic{We next introduce some notation which is frequently used in the paper. Observe that there exist $0<\gmb\leq \gpb<1$ such that
\eqq{\int_{\gmb}^{1}\dd >0.}{intmugk2}
and
\eqq{\int_{\gmb}^{\gpb}\dd >0.}{intmugkgmp}
Of course, the constants $\gmb$ and $\gpb$ are not uniquely defined, however in the arguments, we take fixed $\gmb$ and $\gpb$ satisfying property (\ref{intmugkgmp}). Usually we will take $\gmb$ and $\gpb$ close to the right endpoint of the support of $\mu$, for example
if $\izj w(\al) d\al  > 0$ then there exist some $0< \gm <\gp <1$ such that
\[
\imp  w(\al)d\al  > 0.
\]
For such $\gm$ and $\gp$ we set
\eqq{
\gpb= \max\{\gp,\al_{M}\}\m{ and }\gmb= \max\{\gm,\al_{M}\}.
}{fin1}
If $w \equiv 0$ we set just $\gpb=\gmb = \al_{M}$ and if $\mu = wd\nu_{1}$, then we set $\gpb=\gp$, $\gmb=\gm$.}

The regularity theory for weak solutions to the problem with single order fractional time derivative has been established in a series of papers by Zacher: \cite{Za2} (boundedness of weak solutions), \cite{base} (weak Harnack inequality for nonnegative weak supersolutions), and \cite{Zhol} (H\"older regularity of weak solutions).
Later, for the problem with single order fractional time derivative and fractional diffusion in space, H\"older continuity of weak solutions was proved in \cite{Caf}, and a weak Harnack inequality was derived in \cite{JPY}. In contrast to the classical parabolic De Giorgi-Nash-Moser
theory, the full Harnack inequality (in its usual form) fails to hold for nonnegative solutions to time-fractional diffusion equations (local or nonlocal in space) if the space dimension is at least two, see
\cite{DKSZ}.

It is worth emphasising that among these results, only the boundedness of weak solutions to (\ref{MProb}) have been proved with a more general kernel $k$ (\cite{Za2}). To describe the result, following \cite{ZWH}, a kernel $k\in L_{1,\,loc}(\iR_+)$ is called to be of type
$\mathscr{PC}$ if it is nonnegative and nonincreasing, and there
exists a kernel $l\in L_{1,\,loc}(\iR_+)$ such that $k\ast l=1$ in
$(0,\infty)$; $(k,l)$ is then a so-called $\mathscr{PC}$ pair. In this case, $l$ is completely
positive (cf.\ Thm. 2.2 in \cite{CN}) and thus nonnegative. The assumption on $k$ in the boundedness results from \cite{Za2} is now that $k$
is of type $\mathscr{PC}$ and that the corresponding kernel $l$ belongs to
$L_{p}((0,T))$ for some $p>1$. This assumption covers a wide class of kernels, including those considered in this paper.

So, regularity of weak solutions beyond boundedness has been established only in the case
of a single order fractional time derivative, i.e.\ $\mu = \delta(\cdot-\al)$. In this paper, we substantially generalise the above-mentioned results by developing a De Giorgi-Nash-Moser theory for evolution equations of the form \eqref{MProb} with a general distributed order fractional time derivative, i.e.\ the kernel $k$ is as described above in \eqref{bo1}, \eqref{kintro}. We establish a weak Harnack inequality for
nonnegative weak supersolutions and prove interior H\"older continuity of weak solutions to (\ref{MProb}). Our unified approach contains in particular such special cases as: single order fractional time derivatives ($\mu = \delta(\cdot-\al)$), multi-term fractional time derivatives ($\mu = \sum_{n=1}^{M}q_{n}\delta(\cdot - \al_{n})$) and 'purely distributed' fractional time derivatives ($d\ma = w(\al)d\nu_{1}(\al)$). One of the advantages of our approach is that we only assume nonnegativity and integrability of the weight function $w$. In the literature, the distributed order derivative is usually considered with a more regular weight $w$, and the authors frequently impose additional assumptions concerning the behaviour of $w$ near the endpoints of the interval $[0,1]$. In our treatment, we develop further some ideas introduced by the first two authors in \cite{decay} and \cite{nasza}, which allows us to work with very general measures $\mu$.

Scaling properties of the equations play an important role in De Giorgi-Nash-Moser regularity theory as the scaling indicates how the
local sets (time-space cylinders) are to be selected. Without suitable geometry, the iteration techniques of De Giorgi and Moser do not work.
We point out that, although the problem for a single fractional time derivative of order $\al\in (0,1)$ admits a natural scaling with similarity variable $s=|x|^2/t^\alpha$, this is no longer the case for the distributed order derivative. In fact, the {\em lack of natural scaling} is one of the biggest obstacles in establishing the regularity theory for (\ref{MProb}). We overcome this difficulty by introducing appropriate time-space cylinders whose shape depends on the kernel $k$.

In the proof of the weak Harnack inequality we mostly follow the approach of \cite{base}, i.e.\ we establish mean-value inequalities by means of Moser iteration schemes, prove weak $L_1$ estimates for the logarithm of the supersolution, and apply a lemma of Bombieri and Giusti. However, we point out that, although the
general idea of the proof and some basic estimates have been taken from \cite{base}, our case is much more involved. In particular, two crucial ingredients in the proof of the weak Harnack inequality are Lemma~\ref{scaling} and Lemma~\ref{rtheta}. It should be emphasised that these results are much easier to obtain if we assume that the support of the measure $\mu$ is cut-off from one.
The solution of the problem in the whole generality requires not only more careful and complicated calculations in the proofs of Lemma \ref{scaling} and Lemma \ref{rtheta}, but also a completely new argument in the essential part of the logarithmic estimates.
As a by-product, the latter leads to a significant improvement of Zacher's results on the single order case from \cite{base} by establishing
the {\em robustness of the estimates as $\alpha\to 1$}.

Having obtained the weak Harnack inequality, we use it to prove the H\"older continuity of weak solutions to (\ref{MProb}).
In addition to the substantial generalisation of the result in \cite{Zhol} on the single order case, another novelty is that,
in contrast to \cite{Zhol}, our argument is based on the weak Harnack inequality, which allows for a much less involved proof.
Even if the method of Harnack inequalities to prove regularity is well known, our argument seems to be new in the context of temporally non-local equations.

\nic{Let us introduce the Riemann-Liouville fractional derivative $\p^{\al}f(t) = \frac{1}{\Gamma(1-\al)}\frac{d}{dt}\izt (t-\tau)^{\al}f(\tau)d\tau$. By $\da f $ we denote the fractional Caputo derivative $\da f(t)= \ra [f(\cdot)-
f(0)](t)$, where $\alpha \in [0, 1]$.  Then we define the distributed order Caputo derivative
\eqq{\dm f(t)=\sum_{n=1}^{M} q_{n} D^{\al_n}f(t)+\int_{0}^{1}(\da f)(t) w(\al) d\al.
}{disCapdet}
We will use  the following shortcut
\eqq{\dm f (t) = \int_{0}^{1}(\da f)(t) \dd.}{disCap}
Similarly we denote
\[
\p^{(\mu)}f(t) = \int_{0}^{1}(\p^{\al} f)(t) \dd.
\]
Formula (\ref{disCap}) may be also written as a convolution
\[
\dm f(t) = \frac{d}{dt}(k*(f-f(0))), \m{ where } k(t):=\izj \frac{t^{-\al} }{\Gamma(1-\al)} \dd.
\]
}
Before we formulate the results, let us introduce the basic assumptions concerning $A$, $u_0$, and $f$. Letting $\Omega_T=(0,T)\times \Omega$ we will assume that
\begin{itemize}
\item [{\bf (H1)}] $A\in L_\infty(\Omega_T;\iR^{N\times
N})$, and
\[
\sum_{i,j=1}^N|a_{ij}(t,x)|^2\le \Lambda^2,\quad \mbox{for
a.a.}\;(t,x)\in \Omega_T.
\]
\item [{\bf (H2)}] There exists a $\nu>0$ such that
\[
\big(A(t,x)\xi|\xi\big)\ge \nu|\xi|^2,\quad\mbox{for a.a.}\;
(t,x)\in\Omega_T,\; \mbox{and all}\;\xi\in \iR^N.
\]
\item [{\bf (H3)}] $u_0\in L_2(\Omega)$ and $f\in L_2(\Omega_T)$.
\end{itemize}

We say that a function $u$ is a {\em weak solution (subsolution,
supersolution)} of (\ref{MProb}) in $\Omega_T$, if $u$ belongs to
the space
\[
Z:=\{v \in L_{2}((0,T);H^1_2(\Omega))\;\mbox{such that}\; k*v \in C([0,T];L_{2}(\Omega)), \mbox{and}\; (k*v)|_{t=0} = 0\}
\]
and for any nonnegative test function
\[
\eta\in \oH^{1,1}_2(\Omega_T):=H^1_2((0,T);L_2(\Omega))\cap
L_2((0,T);\oH^1_2(\Omega)) \quad\quad
\Big(\oH^1_2(\Omega):=\overline{C_0^\infty(\Omega)}\,{}^{H^1_2(\Omega)}\Big)
\]
with $\eta|_{t=T}=0$ we have
\begin{equation} \label{BWF}
\int_{0}^{T} \int_\Omega \Big(-\eta_t [k\ast (u-u_0)]+
(ADu|D \eta)\Big)\,dxdt= \,(\le,\,\ge )\, \int_{0}^{T} \int_\Omega f\eta \, dxdt.
\end{equation}

Weak solutions of (\ref{MProb}) in the class $Z$ have been
constructed in \cite{ZWH} under the assumptions (H1)--(H3). Note that the function $u_0$ plays
the role of the initial data for $u$, at least in a weak sense. In
case of sufficiently regular functions $u$ and
$k\ast(u-u_0)$ the condition $(k\ast
u)|_{t=0}=0$ implies $u|_{t=0}=u_0$, see \cite{ZWH} and Section 3.5 in \cite{nasza}.


Next we describe the geometry of the time-space cylinders appearing in our estimates.
As already mentioned, the choice of the right geometry is crucial for the De Giorgi-Nash-Moser techniques to work.
In our case, the dependence of the shape of the cylinders on the kernel $k$ is more complicated than in the single order case,
where only the order $\alpha\in (0,1)$ determines the geometry. By $B(x,r)$ we denote the open ball with
radius $r>0$ centered at $x\in \iR^N$ and recall that by  $\nuN$ we mean the
Lebesgue measure in $\iR^N$. We then set
\eqq{
\ki(t) = \izj t^{-\al} \dd,\quad t>0.
}{kidef}
We will show in Lemma~\ref{fi} that there is unique increasing function $\Phi \in C([0,\infty))\cap C^{1}((0,\infty))$ such that $\Phi(0) = 0$ and $k_1(\vr) = r^{-2}$ for all $r>0$. With this function, for $\delta\in(0,1)$, $t_0\ge 0$,
$\tau>0$, and a ball $B(x_0,r)$, we then consider the cylinders
\eqnsl{
Q_-(t_0,x_0,r, \delta)&=(t_0,t_0+\delta\tau \Phi(2r))\times B(x_0,\delta r),\\
Q_+(t_0,x_0,r,\delta)&=(t_0+(2-\delta)\tau \Phi(2r),t_0+2\tau
\Phi(2r))\times B(x_0,\delta r).
}{defQpm}
We note that in the single fractional order case, i.e.\  $\mu = \delta(\cdot-\beta)$ with some $\beta\in (0,1)$ we have $\Phi(r) = r^{\frac{2}{\beta}}$, which leads to the cylinders used in \cite{base}.

Let us now present the main results of the article. For this purpose we fix a number $\gmb\in (0,1)$ as large as possible for which
\eqq{\int_{\gmb}^{1}\dd >0.}{intmugk}
Note that here the supremum need not be assumed. The larger $\gmb$, the larger is the critical exponent in the weak Harnack inequality. We have the following theorem.
\begin{satz} \label{localweakHarnack}
Let $T>0, N \geq 1$, and $\Omega\subset \iR^N$ be a bounded
domain. Suppose the assumptions (H1)--(H3) are satisfied. For any $0<p<\frac{2+N\gmb}{2+N\gmb - 2\gmb}$ and any $\tau > 0$ there exists a number  $r^*=r^*(\mu,p)>0$ such that for every
 $\delta\in(0,1)$ and for
any $t_0\ge 0$, any $r\in (0,r^*]$ with $t_{0}+2\tau \Phi(2r) \leq T$, any
ball $B(x_0, r)\subset\Omega$, and any nonnegative weak
supersolution $u$ of (\ref{MProb}) in $(0,t_0+2\tau
\Phi(2r))\times B(x_0, r)$ with $u_0\ge 0$ in $B(x_0,
r)$ and $f\equiv 0$, there holds
\begin{equation} \label{localwHarnackF}
\Big(\frac{1}{\nuNj\big(Q_-(t_0,x_0,r, \delta)\big)}\,\int_{Q_-(t_0,x_0,r, \delta)}u^p\,d\nuNj\Big)^{1/p}
\le C \einf_{Q_+(t_0,x_0,r, \delta)} u,
\end{equation}
where $C=C(\nu,\Lambda,\delta,\tau,\mu,N,p)$.
\end{satz}

As a corollary from Theorem \ref{localweakHarnack}  we obtain the weak Harnack inequality for nonnegative supersolutions to the problem with bounded inhomogeneity. We note that one may repeat the proof of Theorem \ref{localweakHarnack} with inhomogeneity $f$ belonging to a Lebesgue space $L_{q_1}(L_{q_2})$ with sufficiently big $q_1,q_2$. However, to avoid technicalities we only consider bounded inhomogeneities, which is also sufficient to deduce H\"older regularity.
In order to simplify the notation we set $\bvr:=\Phi(2r)$, where $\Phi$ is the function introduced before \eqref{defQpm}.
For the function $f$ we write $f=f^{+}-f^{-}$, where $f^{+},f^{-}\geq 0$ denote the positive and negative part of $f$, respectively.

\begin{satz} \label{localweakHarnackinhomo}
Let $T>0$, $N\ge 1$, and $\Omega\subset \iR^N$ be a bounded
domain. Suppose the assumptions (H1)--(H3) are satisfied. Let
further $\delta\in(0,1)$ 
and $\tau>0$ be fixed. Then, for any $0<p<\frac{2+N\gmb}{2+N\gmb - 2\gmb}$,  there exists a number $r^{*} = r^{*}(\mu,p)>0$ such that for every $r \in (0, r_{*}]$ with $2\tau \bvr \leq T$,  for
any  ball $B(x_{0},
r) \subset \Omega$  and any nonnegative weak
supersolution $u$ of
\begin{equation} \label{MProbG}
\partial_t (k*(u-u_{0})) -\divv \big(A(t,x)Du\big)=f \m{ \hd in \hd }(0,2\tau
\bvr)\times B(x_0, r)
\end{equation}
  with $f \in L_{\infty}((0,2\tau
\bvr)\times B(x_0, r))$ and $u_0\ge 0$ in $B(x_0,
r)$, there holds
\begin{equation} \label{localwHarnackFG}
\Big(\frac{1}{\nuNj\big(Q_-(0,x_0,r,\delta)\big)}\,\int_{Q_-(0,x_0,r,\delta)}u^p\,d\nuNj\Big)^{1/p}
\le C \left(\einf_{Q_+(0,x_0,r,\delta)} u+r^{2}\|f^{-}\|_{L_{\infty}((0,2\tau
\bvr)\times B(x_0, r))}\right),
\end{equation}
where  $C=C(\nu,\Lambda,\delta,\tau,\mu,N,p)$.
\end{satz}

We apply this estimate to deduce H\"older regularity of weak solutions to (\ref{holdeq}), which is our final result.

\begin{satz}\label{holder}
Let $T>0, N \geq 1$ and $\Omega\subset \iR^N$ be a bounded
domain. Suppose the assumptions (H1)--(H3) are satisfied and let $u_{0} \in L_{\infty}(\Omega)$ and $f\in L_{\infty}(\Omega_{T})$. If $u$ is a bounded weak solution to
\eqq{
\partial_t (k*(u-u_{0}))-\divv\,\big(A(t,x)Du\big)=f,\quad t\in (0,T),\,x\in
\Omega,
}{holdeq}
then for any $V \subset \Omega_{T}$ separated from the parabolic  boundary of $\Omega_{T}$ by a positive distance, there exist
constants $C > 0$ and $\ve \in (0,1)$ depending only on $\mu$, $V$, $\Lambda$, $\nu$ and $N$ such that
\eqq{
\|u\|_{C^{0,\ve}(V)} \leq C(\|u\|_{L_{\infty}(\Omega_{T})} +\|u_{0}\|_{L_{\infty}(\Omega)} +\|f\|_{L_{\infty}(\Omega_{T})}).
}{holderkoniec}
\end{satz}
Recall that \cite[Theorem 3.1]{Za2} gives the boundedness of weak solutions to (\ref{holdeq}), provided it is bounded on the parabolic boundary of $\Om_{T}$. In this case  the assumption concerning the boundedness of $u$  in Theorem~\ref{holder}, may be skipped.
\nic{
\begin{korollar}
Let $T>0, N \geq 1$ and $\Omega\subset \iR^N$ be a bounded
domain. Suppose the assumptions (H1)--(H2) are satisfied and let $u_{0} \in L_{\infty}(\Omega)$, $f\in L_{\infty}(\Omega_{T})$. If $u$ is a  weak solution to
\eqq{
\partial_t (k*(u-u_{0}))-\divv\,\big(A(t,x)Du\big)=f,\quad t\in (0,T),\,x\in
\Omega
}{holdeq}
then for any $V \subset \Omega_{T}$ separated from the parabolic  boundary of $\Omega_{T}$ by a positive distance, there holds
\[
\|u\|_{C^{0,\ve}(V)} \leq C(\|u_{0}\|_{L_{\infty}(\Omega)} +\|f\|_{L_{\infty}(\Omega_{T})}),
\]
where $C > 0$ and $\ve \in (0,1)$ depend only on $\mu$, $V$, $\Lambda$, $\nu$ and $N$.
\end{korollar}
}

As already mentioned, the exponent in the weak Harnack inequality depends on $\gmb$. Let us provide some examples. In the case of a multi-term fractional time derivative ($\ma = \sum_{n=1}^{M}q_{n}\delta(\cdot - \al_{n})$ with $q_M>0$) we take $\gmb = \alpha_M$, which leads to the condition $p< \frac{2+N\al_M}{2+N\al_M - 2\al_M}$ in the weak Harnack inequality. We obtain the same result
in the case where a single order fractional derivative dominates the distributed order part associated with the weight function $w$
in the sense that
\[
d\mu = \sum_{n=1}^{M}q_{n}d\delta(\cdot - \al_{n}) + w d\nu_1,\quad q_M>0,\quad \hd \supp w \subseteq [0,\al_M].
\]
On the other hand if for every $\gamma \in (0,1)$ we have that $\int_{\gamma}^{1} w(\al) d\al > 0$, then we can take $\gmb$ arbitrarily close to one and the weak Harnack inequality holds for any positive $p< 1+\frac{2}{N }$, that is, the critical exponent coincides
with the one from the classical parabolic case.

The paper is organised as follows. In Section 2, we recall various preliminary results from distributed order calculus and establish estimates for the kernel $l$ and the resolvent kernel associated with $l$ which are crucial to make our approach work, but which are also of independent interest in the context of Volterra equations with completely positive kernels. Then we recall Moser iteration lemmas and the lemma of Bombieri and Giusti as well as other auxiliary results. Section 3 is devoted to the proof of the weak Harnack estimate (Theorem \ref{localweakHarnack}), while the final Section 4 contains the proof of Theorem~\ref{holder}.
\section{Preliminaries}
\subsection{Introduction from distributed order calculus}
We begin by showing that any kernel $k\in L_{1,\,loc}(\iR_+)$ from the class considered in this paper (see \eqref{kintro}) is of type
$\mathscr{PC}$. In \cite{nasza} it was already proven that there exists $l\in L_{1,loc}([0,\infty))$ such that $(k,l)$ is a $\mathscr{PC}$ pair if $d\mu\equiv wd\nu_{1}$, see also \cite{KochDO}. However for a $\mu$ given by (\ref{bo1}) the proof may be repeated without any changes. Thus, we arrive at the following result.
\begin{satz}\label{l}
Let $\mu$ be given by formula (\ref{bo1}), where $q_{n} > 0$ for $n=1,\dots, M$, $w \in L_{1}((0,1))$ is nonnegative a.e. on $(0,1)$.  Then, there exists a nonnegative $l\in L_{1,loc}([0,\infty))$
such that $k*l=1$ and  the operator of fractional integration $\im$, defined by the formula $ \im
u =l*u$  for $u \in AC([0,T])$ satisfies
\eqq{(\dm \im u)(t) = u(t) \hd \m{ and } \hd  (\im\dm u)(t) = u(t) - u(0).}{fi1}
Furthermore, $l$ is given by the formula
\eqq{l(t) = \frac{1}{\pi}\izi e^{-rt}H(r) dr,}{calka}
where
\eqq{
H(r) = \frac{ \izj r^{\al} \sin (\pi\al)\dd}{( \izj r^{\al} \sin (\pi\al)\dd)^{2}+ (\izj r^{\al} \cos
(\pi\al)\dd)^{2} }.
}{fin2}
\label{fint}
\end{satz}
\begin{bemerk1}
Since the pair $(k,l)$ is a $\mathscr{PC}$ pair it follows from \cite[Theorem 3.1]{ZWH} that the initial-boundary value problem (\ref{MProb}) with homogeneous Dirichlet boundary condition and initial condition $u|_{t=0}=u_0$ has exactly one weak solution belonging to $Z$.
\end{bemerk1}

We remark that in the whole paper, in the estimates, the constants, usually denoted by $c$, may change from line to line.

In the paper we will frequently use the fact that for $x\in (0,1]$ the expression $\int_{\gmb}^{1}x^{-\al}\dd$ dominates $\int_{0}^{\gmb}x^{-\al}\dd$. We will establish this auxiliary result in the next remark.
\begin{bemerk1}
There exists a positive constant $c=c(\mu)$ such that for every $x\in (0,1]$ there holds
\eqq{ \int_{0}^{1}x^{-\al}\dd \leq c \int_{\gmb}^{1}x^{-\al}\dd.}{estioddolu}
Indeed, since $x\in (0,1]$ we may write
\eqq{
\int_{\gmb}^{1}x^{-\al}\dd \geq x^{-\gmb}\int_{\gmb}^{1}\dd.
}{xmgam}
Using   (\ref{xmgam}) for $x \leq 1$ and (\ref{intmugk})  we get
\[
\int_{0}^{1}x^{-\al}\dd = \int_{\gmb}^{1}x^{-\al}\dd + \int_{0}^{\gmb}x^{-\al}\dd
\leq \int_{\gmb}^{1}x^{-\al}\dd + x^{-\gmb}\int_{0}^{\gmb}\dd
\]
\[
\leq \int_{\gmb}^{1}x^{-\al}\dd + \int_{\gmb}^{1}x^{-\al}\dd \frac{\int_{0}^{\gmb}\dd}{\int_{\gmb}^{1}\dd}\equiv c(\mu)\int_{\gmb}^{1}x^{-\al}\dd.
\]

\end{bemerk1}
Let us now establish the crucial estimates for the kernel $l$.
\begin{lemma}\label{kernels}
The kernel $l$ defined in Theorem \ref{l} satisfies $l \in C^{\infty}((0,\infty))$. Furthermore, for every $t > 0$ there holds
\eqq{
l(t) \leq  \frac{1}{\izj t^{1-\al}\dd} .
}{kerestp}
Moreover, for every $T>0$, there exists $c=c(\mu,T)$ such that for every $t \in (0,T]$
\eqq{l(t) \geq c \frac{1}{\izj t^{1-\al}\dd}.}{ldol}

\end{lemma}
\begin{proof}
At first, we note that from (\ref{calka}) we have
\[
l^{(n)}(t) =  \frac{1}{\pi}\izi e^{-rt}(-r)^{n}H(r) dr\in C((0,\infty)),\quad n\in \iN.
\]
Since $k*l=1$ and $l$ is nonincreasing, we infer that
\[
l(t) \izt k(\tau)d\tau \leq (k*l) (t)= 1,\quad t>0,
\]
hence
\[
l(t) \leq \frac{1}{\izt k(\tau)d\tau} = \frac{1}{\izj \frac{t^{1-\al}}{\Gamma(2-\al)}\dd} \leq  \frac{1}{\izj t^{1-\al}\dd} .
\]

To show (\ref{ldol}) we recall that from (\ref{calka}) we have
\[
l(t) \geq \frac{1}{\pi}\int_{1}^{\infty} e^{-pt}H(p)dp.
\]
We apply (\ref{estioddolu})  with $x=p^{-1}$ and we obtain
\eqq{\izj \pal \dd
\leq c(\mu) \int_{\gmb}^{1} \pal \dd \hd \m{ for } \hd p \geq 1.
}{estijedenszesc}
On the other hand we have
\eqnsl{\izj \pal \sinpa \dd \geq \int_{\gmb}^{1} \pal \sinpa \dd \geq c(\mu)\int_{\gmb}^{1} \pal (1-\al) \dd .
}{estijedenczter}
Therefore, using (\ref{estijedenszesc}) and (\ref{estijedenczter}) for $p\geq 1$ we get
\eqnsl{H(p)\geq c(\mu) \frac{\int_{\gmb}^{1} \pal (1-\al) \dd}{\naw{  \int_{\gmb}^{1} \pal \dd}^{2}}.  }{estijedenpie}
We note that
\[
\int_{0}^{1}x^{1-\al}\dd \rightarrow 0 \quad \m{as}\quad x\rightarrow 0,
\]
thus, we may fix $B=B(\mu, T)\geq \max\{1,T \}$ such that for $t\in (0,T]$
\[
\int_{0}^{1}\left(\frac{t}{B}\right)^{1-\al}\dd \leq \int_{0}^{1}\left(\frac{T}{B}\right)^{1-\al}\dd  \leq \jd \int_{\gmb}^{1}\dd.
\]
Therefore, for every $t \in (0,T]$ we have
\[
\jd \frac{1}{\int_{0}^{1}\left(\frac{t}{B}\right)^{1-\al}\dd} \geq \frac{1}{\int_{\gmb}^{1}\dd},
\]
and applying estimate (\ref{estijedenpie}) in (\ref{calka}) we obtain
\[
l(t) \geq c(\mu) \int_{1}^{\frac{B}{t}}e^{-pt} \frac{\int_{\gmb}^{1}  (1-\al) p^{\al}\dd}{\naw{  \int_{\gmb}^{1} p^{\al} \dd}^{2}}dp \geq \frac{c(\mu)}{e^{B}}\int_{1}^{\frac{B}{t}} \frac{\int_{\gmb}^{1}  (1-\al)p^{\al-2} \dd}{\naw{  \int_{\gmb}^{1} p^{\al-1} \dd}^{2}}dp
\]
\[
= \frac{c(\mu)}{e^{B}}\int_{1}^{\frac{B}{t}} \frac{d}{dp}\left(\frac{1}{  \int_{\gmb}^{1} p^{\al-1} \dd}\right)dp = \frac{c(\mu)}{e^{B}}\left(\frac{1}{ \int_{\gmb}^{1} (\frac{t}{B})^{1-\al} \dd} -\frac{1}{ \int_{\gmb}^{1} \dd}  \right)
\]
\[
\geq \frac{c(\mu)}{2e^{B}}\frac{1}{ \int_{0}^{1} (\frac{t}{B})^{1-\al} \dd}  \geq \frac{c(\mu)}{2e^{B}}\frac{1}{ \int_{0}^{1} t^{1-\al} \dd}.
\]

\end{proof}
\begin{bemerk1}\label{lbemerk}
There exists a positive constant $c=c(\mu)$ such that
\eqq{
l(t) \leq c t^{\gmb-1} \m{ for } t \in (0,1).
}{lesti}

Indeed, from Lemma \ref{kernels} and (\ref{intmugk}) we obtain that for $t \in (0,1)$
\[
l(t) \leq \frac{1}{\izj t^{1-\al}\dd} \leq \frac{1}{\int_{\gmb}^{1} t^{1-\al}\dd} \leq c(\mu)t^{\gmb-1}.
\]

\end{bemerk1}

Next, we present the formula for the solution to the resolvent equation associated with the kernel $l$, which will play an important role in the logarithmic estimates.

\begin{lemma}
Let $l$ be the kernel given by Theorem \ref{l}. The solution to the resolvent equation associated with the kernel $l$,
\eqq{
r_{\theta}(t) + \theta (r_{\theta}* l)(t) = l(t), \quad t>0,
}{reso}
with $\theta\ge 0$, is given by
\eqq{
r_{\theta}(t) = \frac{1}{\pi}\izi e^{-pt}H_{\theta}(p)dp, \m{ where } H_{\theta}(p):=\frac{ \izj p^{\al} \sin (\pi\al)\dd}{( \izj p^{\al} \sin (\pi\al)\dd)^{2}+ (\theta + \izj p^{\al} \cos
(\pi\al)\dd)^{2} }.
}{rformula}
\end{lemma}
\begin{proof}
We will prove this result applying the Laplace transform similarly as in  \cite[Theorem 2]{decay}. We recall Lemma~6 from \cite{nasza} (see also Lemma~2.1 in \cite{Laplace}), which serves as a tool for the inversion of the Laplace transform.

\begin{lemma}\label{odwrotna}
Let $F$ be a complex function satisfying the following assumptions:
\begin{enumerate}[1)]
\item
$F(p)$ is analytic in $\mathbb{C} \setminus (-\infty, 0].$
\item
The limit $F^{\pm}(t):=\lim\limits_{\vf\rightarrow \pi^{-}}F(te^{\pm i \vf})$ exists for a.a. $t>0$ and $F^{+} = \overline{F^{-}}$.
\item
For each   $0<\eta<\pi$
\begin{enumerate}[a)]
\item
 $|F(p)| = o(1)$, as $|p|\rightarrow \infty$ uniformly on $ |\arg(p)| < \pi - \eta$,
\item
$|F(p)|= o(\frac{1}{|p|})$, as $|p|\rightarrow 0$ uniformly on $|\arg(p)|< \pi - \eta$.
\end{enumerate}
\item
There exists $\ve_{0} \in (0,\frac{\pi}{2})$ and a function $a=a(r)$ such that \hd $\forall  \vf \in (\pi - \ve_{0}, \pi)$  the estimate  $\abs{F(re^{\pm i\vf})}   \leq a(r)$ holds, where $ \frac{a(r)}{1+r} \in L_{1}(\mr_{+})$.
\end{enumerate}
Then for $p\in \mathbb{C}$ such that $\Re{p}>0$ we have
\[
F(p) = \izi e^{-xp}f(x)dx,\quad \textrm{where}\quad f(x) = \frac{1}{\pi}\izi e^{-rx} \Im(F^{-}(r))dr.
\]
\end{lemma}
Applying the Laplace transform to (\ref{reso}) we have
\[
\hat{r_{\theta}}(p) + \theta \hat{r_{\theta}}(p)\cdot \hat{l}(p) = \hat{l}(p), \quad \mbox{and thus} \quad
\hat{r_{\theta}}(p) = \frac{\hat{l}(p)}{1+\theta \hat{l}(p)}.
\]
We recall that
\[
\hat{l}(p) = \frac{1}{p\hat{k}(p)} = \frac{1}{\izj p^{\al}\dd},
\]
where $p^{\al} = \exp(\al \log p)$ and $\log p$ denotes the principal branch of the logarithm.
Consequently,
\[
\hat{r_{\theta}}(p) = \frac{1}{\theta+\izj p^{\al}\dd}.
\]
We will show using Lemma \ref{odwrotna} that $r_{\theta}$ is given by (\ref{rformula}).

Let us define for $\theta > 0$
\[
F(p) = \frac{1}{\theta + \izj p^{\al}\dd}.
\]
Observe that $F$ is analytic in $\mathbb{C}\setminus (-\infty,0]$, because $\Im (p\hat{k}(p)) \neq 0$ if $\abs{\Arg(p)} \in (0,\pi)$ and $\Re (p\hat{k}(p))\neq 0$ if $\Arg(p) = 0$.
It is easy to see that $F^{+} = \overline{F^{-}}$ and
\[
|F(p)| \rightarrow 0 \m{ as } \abs{p} \rightarrow \infty \m{ and } |pF(p)| \rightarrow 0 \m{ as } \abs{p} \rightarrow 0,
\]
thus the assumptions 1)-3) of Lemma \ref{odwrotna} are satisfied.

We will now show that 4) holds, too. We fix $\epz \in (0, \frac{\pi}{2})$ and denote $p=re^{\pm i \vp}$, where $\vp \in (\pi-\epz, \pi)$. Then,
\[
\left| \izj p^{\al} \dd + \theta  \right|\geq \izj r^{\al} |\sin(\pm \vp \al ) | \dd \geq \int_{\gmb}^{\gpb} r^{\al} \sin(\vp \al ) \dd,
\]
with some fixed $\gamma_+\in [\gamma_-,1)$, such that $\int_{\gmb}^{\gpb}\dd > 0$. We note that $\sin(\vp \al ) \geq \min\{\sin{((\pi- \epz)\gmb)}, \sin{(\pi \gpb)}  \}$ for $\al \in (\gmb, \gpb)$, thus we obtain
\[
\left| \izj p^{\al} \dd + \lambda  \right|\geq c_{0} \min\{r^{\gmb}, r^{ \gpb} \},
\]
where the constant $c_{0}$ depends only on $\mu$. Thus, we may notice that
\eqq{|F(p)| \leq c_{0}^{-1} \max\{r^{-\gmb}, r^{-\gpb} \}. }{f2}
We define the majorant for $F(p)$ by the formula
\[
a(r) =  c_{0}^{-1} r^{-\gpb} \m{ for } 0<r \leq 1 \m{ and } a(r) =  c_{0}^{-1} r^{-\gmb} \m{ for } r > 1.
\]
Then $\frac{a(r)}{1+r} \in L_{1}(\mathbb{R_{+}})$. Since all the assumptions of Lemma \ref{odwrotna} are satisfied we obtain
\[
r_{\theta}(t) = \frac{1}{\pi}\izi e^{-pt} \Im(F^{-}(p))dp.
\]
After a brief calculation we arrive at (\ref{rformula}).
\end{proof}


\subsection{Properties of the cylinders}
In this subsection we introduce the function $\Phi$ which defines the shape of our time-space cylinders. Then we establish important estimates concerning the kernel $l$ and the resolvent kernel associated with~$l$.
\begin{lemma}\label{fi}
There exists  a unique strictly increasing function $\Phi \in C([0,\infty))\cap C^{1}((0,\infty))$ such that $\Phi(0) = 0$,
$\lim_{r\to \infty}\Phi(r)=\infty$ and
\eqq{
 \kjr = r^{-2}, \quad r>0.
}{zn1}
\end{lemma}
\begin{proof}
Note that $k_{1}(x) = \izj  x^{-\al}\dd$ is a decreasing and smooth function on $(0,\infty)$.
Furthermore, by monotone convergence, we deduce that  $\lim\limits_{x\rightarrow 0^{+}}k_{1}(x) = \infty$ and $\lim\limits_{x\rightarrow \infty}k_{1}(x) = 0$. Thus, by the Darboux property, for every $r >0$ there exists $\Phi(r) >0$ such that
\[
k_{1}(\Phi(r)) = r^{-2},
\]
and $\Phi(r)$ is uniquely determined because $k_{1}'<0$ on $(0,\infty)$. In particular, from the implicit function theorem we deduce that  $\Phi \in C^{1}((0,\infty))$ and $\Phi'(r)>0$. If $y_{0}:=\inf\limits_{r>0}\Phi(r)=\lim\limits_{r\rightarrow 0^{+}} \Phi(r)$ was positive, then we would have
\[
k_{1}(y_{0})= \lim_{r\rightarrow 0^{+}}k_{1}(\Phi(r))=\lim_{r\rightarrow 0^{+}} \frac{1}{r^{2}}=\infty,
\]
which gives a contradiction. Thus, we have $y_{0}=0$.
\end{proof}

It is easy to see that if $\mu = \delta_{\beta_{*}}$ for some $\beta_{*}\in (0,1)$, then $\Phi(r)= r^{\frac{2}{\beta_{*}}}$.
\begin{bemerk1}
We note that since $\frac{1}{\Gamma(1-\al)} \leq 1$ for $\al \in (0,1)$ we have
\eqq{
k(\vr) \leq \kjr = \frac{1}{r^2}.
}{znk}
\end{bemerk1}
\begin{prop}\label{philambda}
    For every $r > 0$ and for every $\lambda \in (0,1]$ there holds
    \[
    \Phi(\lambda r) \leq \lambda^2 \vr.
    \]
\end{prop}
\begin{proof}
    From (\ref{zn1}) we have
    \[
    k_1(\Phi(\lambda r)) = \lambda^{-2}r^{-2} = \lambda^{-2}\izj (\vr)^{-\alpha}\dd \geq \izj \lambda^{-2\alpha}(\vr)^{-\alpha}\dd = k_1(\lambda^2 \Phi( r)).
    \]
    Since $k_1$ is decreasing we obtain the claim.
\end{proof}
The subsequent lemma plays an essential role in the derivation of the mean value inequalities.
\begin{lemma}\label{scaling}
Let $\Phi$ be the function from Lemma \ref{fi} and $\gmb\in (0,1)$ be such that (\ref{intmugk}) is satisfied. Then for every  $1 \leq  p < \frac{1}{1-\gmb}$ there exist $r^{*}=r^{*}(\mu, p) > 0$ and $C=C(\mu,p) >0$ such that $\Phi(2r^{*})\leq 1$ and for every $r \in [0,r^{*}]$ there holds
\eqq{
 \norm{l}_{L_{p}(0,\vdr)}^{p} (\vdr)^{p-1} \leq C r^{2p}.
}{zn2}
\end{lemma}
\begin{proof}
We recall that thanks to (\ref{zn1})
\[
r^{2p} = 2^{2p}(k_1(\vdr))^{-p}.
\]
Therefore, (\ref{zn2}) is equivalent to the estimate
\[
\norm{l}_{L_{p}(0,\vdr)}^{p} (\vdr)^{p-1} \leq C (k_1(\vdr))^{-p}.
\]
In view of (\ref{kerestp}), it is enough to show that for $r>0$ small enough
\[
\int_{0}^{\vdr}\frac{1}{\left(\izj  \tau^{1-\al}\dd\right)^{p}}d\tau (\vdr)^{p-1} \leq c \frac{1}{\left(\izj  (\vdr)^{-\al} \dd \right)^{p}}
\]
with some constant $c=c(\mu,p)$. Let us denote $x:=\Phi(2r)$. Then multiplying by $x^{1-p}$ the preceding inequality can be reformulated as
\eqq{
\int_{0}^{x}\frac{1}{\left(\izj  \tau^{1-\al} \dd \right)^{p}}d\tau \leq c \frac{1}{\left(\int_{0}^{1} x^{1-\al-\frac{1}{p}} \dd \right)^{p}}.
}{finn1}
Let us denote
\[
g(x):=\int_{0}^{x}\frac{1}{\left(\izj  \tau^{1-\al} \dd \right)^{p}}d\tau, \quad h(x):=\frac{x}{\left(\int_{0}^{1} x^{1-\al} \dd \right)^{p}}.
\]
For $\tau\in (0,1]$ we have
\[
\izj  \tau^{1-\al}\dd \geq \int_{\gmb}^{1}  \tau^{1-\al} \dd \geq \tau^{1-\gmb} \int_{\gmb}^{1} \dd  \equiv \tau^{1-\gmb} c(\mu)
\]
where $c(\mu)>0$, thanks to (\ref{intmugk}). Therefore, we have
\[
\left(  \izj  \tau^{1-\al} \dd \right)^{-p}\leq c(\mu)^{-p}\tau^{p(\gmb-1)}\in L_{1}((0,1)),
\]
because $p< \frac{1}{1-\gmb}$. Hence, $g(x)$ is well defined and $g(0)=0$. Using the same reasoning for the function $h(x)$ we get
\[
h(x)\leq c(\mu)^{-p} x^{1+p(\gmb-1)},\quad x\in (0,1],
\]
and then we also see that $h(0)=0$.

We will now show that there exists $c=c(\mu, p) > 0$ such that
\[
g'(x) \leq c h'(x) \m{ \hd  for \hd } x \in (0,\Phi(2r^{*})],
\]
for some positive $r^{*}=r^{*}(\mu, p)$, which will finish the proof, since $g',h' \in L_{1}((0,\Phi(2r^{*})))$.

Let us firstly discuss the case when $1<p < \frac{1}{1-\gmb}$. We have
\[
g'(x) = \frac{1}{\left(\int_{0}^{1}x^{1-\al}\dd \right)^{p}}, \hd \hd
h'(x)=   p \frac{\int_{0}^{1}(\al+\frac{1}{p}-1)x^{-\al-\frac{1}{p}}\dd}{\left(\int_{0}^{1}x^{1-\al-\frac{1}{p}}\dd \right)^{p+1}},
\]
and consequently
\[
h'(x) = p g'(x) \frac{\int_{0}^{1}(\al+\frac{1}{p}-1)x^{-\al}\dd}{\int_{0}^{1}x^{-\al}\dd},
\]
and we have to show that there exists $c=c(\mu,p) >0 $ such that
\eqq{
\int_{0}^{1}(\al+\frac{1}{p}-1)x^{-\al}\dd \geq c \int_{0}^{1}x^{-\al}\dd
}{finn4}
for $x$ as before.  To this end we write
\eqq{
\int_{0}^{1}(\al+\frac{1}{p}-1)x^{-\al}\dd = \int_{1-\frac{1}{p}}^{1}(\al+\frac{1}{p}-1)x^{-\al}\dd - \int_{0}^{1-\frac{1}{p}}(1-\al-\frac{1}{p})x^{-\al}\dd.
}{finn2}
If $1<p<\frac{1}{1-\gmb}$, then $\gmb > 1-\frac{1}{p}$, and thus
\eqq{
\int_{1-\frac{1}{p}}^{1}(\al+\frac{1}{p}-1)x^{-\al}\dd \geq \int_{\gmb}^{1}(\al+\frac{1}{p}-1)x^{-\al}\dd \geq (\gmb+\frac{1}{p}-1)\int_{\gmb}^{1}x^{-\al}\dd.
}{zn2p}
On the other hand, for $x\in (0,1]$ we have
\[
\int_{0}^{1-\frac{1}{p}}(1-\al-\frac{1}{p})x^{-\al}\dd \leq (1-\frac{1}{p})x^{\frac{1}{p}-1}\int_{0}^{1-\frac{1}{p}} \dd = (1-\frac{1}{p})x^{-\gmb}\cdot x^{\gmb-1+\frac{1}{p}}\int_{0}^{1-\frac{1}{p}} \dd.
\]
Since $x \leq 1$ we may write
\[
\int_{\gmb}^{1}x^{-\al}\dd \geq x^{-\gmb}\int_{\gmb}^{1}\dd.
\]
Inserting this result into the line above we obtain
\eqq{
\int_{0}^{1-\frac{1}{p}}(1-\al-\frac{1}{p})x^{-\al}\dd \leq (1-\frac{1}{p})\frac{\int_{0}^{1-\frac{1}{p}} \dd}{\int_{\gmb}^{1}\dd} \int_{\gmb}^{1}x^{-\al}\dd \cdot x^{\gmb-1+\frac{1}{p}}.
}{zn2r}
Combining (\ref{finn2}), (\ref{zn2p}) and (\ref{zn2r}) yields
\eqns{
\int_{0}^{1}(\al+  \frac{1}{p}-1)  x^{-\al}\dd
\geq  \left[(\gmb+\frac{1}{p}-1) - (1-\frac{1}{p})\frac{\int_{0}^{1-\frac{1}{p}} \dd}{\int_{\gmb}^{1}\dd}\cdot x^{\gmb-1+\frac{1}{p}}\right] \int_{\gmb}^{1}x^{-\al}\dd.
}
Note that $\gmb+\frac{1}{p}-1$ is positive, hence,  there exists $r^{*} \in (0,1)$ depending only on $\mu$ and $p$, such that $\Phi(2r^{*}) \leq 1$ and for every $x \in (0,\Phi(2r^{*}))$ there holds
\[
\int_{0}^{1}(\al+\frac{1}{p}-1)x^{-\al}\dd \geq \frac{1}{2}\left(\gmb+\frac{1}{p}-1\right)\int_{\gmb}^{1}x^{-\al}\dd.
\]
Applying (\ref{estioddolu}) we obtain (\ref{finn4}).

The case $p=1$ follows from the previous case, by H\"older's inequality. But it is also instructive to give a direct argument,
to also provide some further estimates required later in the paper.

Analogously to the previous case, we have to show that there exists $c > 0$ depending only on $\mu$ such that
\eqq{
g(x) \leq c h(x),\hd \textrm{ where } \hd g(x) = \int_{0}^{x}\frac{1}{\izj  \tau^{1-\al} \dd}d\tau \m{ and } h(x) = \left(\izj x^{-\al}\dd\right)^{-1}.
}{gest}
Note that $g(0)=h(0)=0$. Thus it is enough to show $g'(x) \leq c h'(x)$ for $x \in (0,1]$.
Observe that
\[
h'(x) = g'(x) \frac{\izj \al x^{-\al} \dd
}{\izj x^{-\al}\dd}.
\]
Applying (\ref{estioddolu}) we arrive at
\eqq{
\izj \al x^{-\al} \dd \geq \gmb \int_{\gmb}^{1}  x^{-\al} \dd \geq c(\mu)\izj x^{-\al}\dd
}{estioddolualfa}
and hence $h'(x) \geq c(\mu) g'(x)$ which finishes the argument.

\end{proof}

\nic{
\begin{prop}
For $\tau \geq 1$ there holds
\[
\tau \vr \leq \Phi(\tau^{\as/2} r), \hd r>0.
\]
\label{tauvr}
\end{prop}
\begin{proof}
The inequality is equivalent to
\[
k(\tau \vr) \geq k(\Phi(\tau^{\as/2} r))= \tau^{-\as} r^{-2},
\]
where we applied (\ref{zn1}). To prove this we write
\[
k(\tau \vr) = \inag \mg \tau^{-\al} \vr^{-\al} d\al \geq \tau^{-\as}\inag \mg  \vr^{-\al} d\al = \tau^{-\as} r^{-2}.
\]

\end{proof}
}


The following estimate from below of the resolvent kernel associated with $l$ will play a crucial role in the logarithmic estimates.

\begin{lemma}\label{rtheta}
Let $\bar{C},C_{1} > 0$. There exist a positive $r^{*}=r^{*}(\mu)$ and positive $c_1,c_2,c_3$ depending only on $C_{1},\bar{C}$ and $\mu$ such that for every $r \in (0,r^{*}]$, $\theta = \frac{C_{1}}{r^{2}}$, $t \in (0,\bar{C}\vr)$, the following estimates hold:
\eqq{
r_{\theta}(t)\geq c_1 \frac{1}{t}(1*r_{\theta})(t) \geq c_2 l(t) \geq c_3 \frac{1}{(1*k)(t)}.
}{rtheta2}
Thus, in particular, there exist a positive $r^{*}=r^{*}(\mu)$ and a positive $c=c(C_{1},\bar{C},\mu)$ such that for every $r \in (0,r^{*}]$, $\theta = \frac{C_{1}}{r^{2}}$, and $t \in (0,\bar{C}\vr)$, we have
\eqq{
r_{\theta}(t)\geq   c \frac{1}{\int_{0}^{1}t^{1-\al}\dd }.
}{rtheta1}
\end{lemma}
\begin{proof}
We will distinguish two cases. Let us firstly assume that the support of $\mu$ is cut-off from one, i.e.
\eqq{
\exists \as \in (0,1)\quad\mbox{such that}\quad \supp \mu \subseteq [0,\as].
}{am1}
In this case the proof is easier. We note that for any $ c_{2} > 0$ and  for any $p \geq p^{*}:=c_{2}(\vr)^{-1}$ there holds
 \[
 \theta \leq \frac{C_{1}}{\min\{c_{2},1\}} \izj p^{\al} \dd.
 \]
 Indeed, applying (\ref{zn1}), for $p \geq p^{*}$ we have
 \eqq{
 \izj p^{\al} \dd \geq \izj c_{2}^{\al}(\vr)^{-\al} \dd \geq \min\{c_{2},1\}\izj (\vr)^{-\al} \dd = \frac{\min\{c_{2},1\}}{C_{1}}\theta.
 }{estteta}
Let us take here $c_2 = 1$. Further, for $p > p^{*}$ we may write
\[ \left( \theta + \izj \pal \cospa \dd   \right)^{2}  + \left( \izj \pal \sinpa \dd \right)^{2}
\]
\eqq{
\leq 2 \theta^{2} + 3\left( \izj p^{\al}\dd\right)^{2} \leq \left(2C^{2}_{1} + 3\right) \left( \izj p^{\al}\dd\right)^{2}.
}{estijedentrzy}

On the other hand, there holds
\eqnsl{\izj \pal \sinpa \dd = \inag \pal \sinpa \dd \geq c(\as) \inag \pal \al \dd .
}{estijedencztery}
Therefore, using (\ref{estijedentrzy}), (\ref{estijedencztery}) and (\ref{am1}) for $p\geq \ps$ we obtain
\eqnsl{H_{\theta}(p) \geq c(\mu,C_{1}) \frac{\inag \pal \al \dd}{\naw{  \int_{0}^{\as} \pal \dd}^{2}}.  }{estijedenpiec}
Applying (\ref{estioddolualfa}) with $t = \frac{1}{p}$ we arrive at

\[
H_{\theta}(p) \geq c(\mu,C_{1}) \frac{1}{  \int_{0}^{\as} \pal \dd}
\]
for  $p\geq \ps$, provided $\vr \leq 1$.
Using this estimate in (\ref{rformula}) we get
\[
r_{\theta}(t) \geq c(\mu,C_{1}) \int_{p^{*}}^{\infty}e^{-pt} \frac{1}{\izj p^{\al} \dd}dp = \podst{pt = w}{tdp = dw}
=c(\mu,C_{1})\int_{tp^{*}}^{\infty}e^{-w} \frac{1}{\int_{0}^{\as} w^{\al} t^{1-\al}\dd}dw.
\]
Letting $t \in (0,\bar{C} \vr)$ we thus have
\[
r_{\theta}(t) \geq  c(\mu,C_{1}) \int_{\bar{C}}^{\infty}e^{-w} \frac{1}{\int_{0}^{\as}w^{\al}t^{1-\al}\dd} dw \geq c(\mu,C_{1}) \int_{\max\{ \bar{C} ,1\} }^{\infty}e^{-w} \frac{1}{\int_{0}^{\as}w^{\al}t^{1-\al}\dd} dw
\]
\[
\geq c(\mu,C_{1}) \int_{\max\{ \bar{C},1\}}^{\infty}e^{-w}w^{-1}dw \frac{1}{\int_{0}^{\as}t^{1-\al}\dd} = c(\mu,C_{1},\bar{C})\frac{1}{\int_{0}^{\as}t^{1-\al}\dd}
\]
and we obtain (\ref{rtheta1}) in the case when (\ref{am1}) holds.

To show (\ref{rtheta2}) we first note that from (\ref{ldol}) we get
\[
\frac{1}{(1*k)(t)} = \frac{1}{\izj \frac{t^{1-\al}}{\Gamma(2-\al)} \dd} \leq \frac{1}{\izj t^{1-\al} \dd} \leq c(\mu, T) l(t).
\]
Next, from (\ref{kerestp}) we get
\[
l(t)\leq \frac{1}{\izj t^{1-\al} \dd} \leq \frac{1}{t} \izt \frac{1}{\izj \tau^{1-\al} \dd} d\tau  \leq \frac{c}{t} \izt r_{\theta}(\tau ) d\tau,
\]
where in the second inequality we use the fact that $t\mapsto \frac{1}{\izj t^{1-\al} \dd}$ is decreasing and in the last one we apply (\ref{rtheta1}). Finally, we note that  (\ref{rformula}) gives $r_{\theta}>0$, hence from (\ref{reso}) we get $r_{\theta}< l$ and then
\[
\frac{1}{t}(1*r_{\theta})(t)\leq \frac{1}{t}(1*l)(t)\leq  \frac{1}{t} \izt \frac{1}{\izj \tau^{1-\al} \dd} d\tau \leq \frac{c}{t}  \frac{1}{\izj t^{-\al} \dd} \leq c r_{\theta}(t),
\]
where we applied (\ref{kerestp}), (\ref{gest}) and (\ref{rtheta1}).

In order to show the result when (\ref{am1}) does not hold we firstly show that in the general case for $t \in (0,\bar{C}\vr)$  there exists $c=c(C_{1},\bar{C},\mu)>0$ such that
\eqq{
\frac{1}{t} (1* r_{\theta})(t) \geq c l(t).
}{intrt}
We note that since $r_\theta$ and $l$ are nonincreasing
\[
\theta \izt r_\theta(t-\tau)l(\tau)d\tau \leq \theta \left[r_{\theta}(\frac{t}{2})\int_{0}^{\frac{t}{2}}l(\tau)d\tau +  l(\frac{t}{2})\int_{\frac{t}{2}}^{t}r_{\theta}(t-\tau)d\tau \right].
\]
Applying (\ref{kerestp}), (\ref{gest}) and (\ref{zn1}) we have
\[
\theta \izt r_\theta(t-\tau)l(\tau)d\tau \leq C_{1}\izj (\Phi(r))^{-\al}\dd \left[ \frac{c(\mu)}{\izj t^{-\al}\dd}  r_{\theta}(\frac{t}{2}) +  \frac{1}{\izj 2^{\al-1} t^{1-\al}\dd} \int_{\frac{t}{2}}^{t}r_{\theta}(t-\tau)d\tau\right].
\]
Since $t \leq \bar{C} \Phi(r)$ we have $(\Phi(r))^{-\al} \leq \bar{C}^{\al} t^{-\al} $ and
\[
\theta \izt r_\theta(t-\tau)l(\tau)d\tau \leq c(C_{1},\bar{C},\mu) \left[r_\theta(\frac{t}{2}) + \frac{1}{t}\int_{0}^{\frac{t}{2}}r_{\theta}(\tau)d\tau\right]\leq c(C_{1},\bar{C},\mu) \frac{1}{t}\izt r_{\theta}(\tau)d\tau,
\]
where in the last estimate we used the fact that $r_{\theta}$ is nonincreasing. Recalling (\ref{reso}) we arrive at
\[
r_\theta(t) + c(C_{1},\bar{C},\mu)\frac{1}{t} \izt r_{\theta}(\tau)d\tau \geq l(t).
\]
Using again monotonicity of $r_\theta$ we obtain (\ref{intrt}).

Now, we will show that under the assumption that the support of $\mu$ is not cut-off from one there exists $c = c(C_{1},\bar{C},\mu)>0$ such that for any $t \in (0,\bar{C}\Phi(r))$, $0<r\leq r^{*} = r^{*}(\mu)$
\eqq{
(1* r_{\theta})(t) \leq c t r_\theta(t).
}{intrts}
Recalling (\ref{rformula}), for any $A > 0$ we have
\[
\pi \izt r_{\theta}(\tau)d\tau =\izi \frac{1 - e^{-pt}}{p} H_{\theta}(p) dp = t \int_{0}^{At^{-1}} \frac{1 - e^{-pt}}{pt} H_{\theta}(p) dp + \int_{At^{-1}}^{\infty} \frac{1 - e^{-pt}}{p} H_{\theta}(p) dp
\]
\eqq{
\leq e^{A} t \int_{0}^{At^{-1}}e^{-pt} H_{\theta}(p) dp + \int_{At^{-1}}^{\infty} \frac{1 - e^{-pt}}{p} H_{\theta}(p) dp.
}{hth3}
We will show that for appropriately chosen $A= A(C_{1},\bar{C},\mu)$, for every $0<r \leq r^{*} = r^{*}(\mu)$ and every $t \in (0,\bar{C}\vr)$ there holds
\eqq{
\int_{At^{-1}}^{\infty} \frac{1 - e^{-pt}}{p} H_{\theta}(p) dp \leq c(\bar{C},C_{1},\mu) t \int_{0}^{At^{-1}}e^{-pt} H_{\theta}(p) dp.
}{hth}
At first we note that
\eqq{
\int_{At^{-1}}^{\infty} \frac{1 - e^{-pt}}{p} H_{\theta}(p) dp \leq \int_{At^{-1}}^{\infty} \frac{1}{p} H_{\theta}(p) dp.
}{ht1}
Let us take for simplicity $A > \bar{C}$. We note that since (\ref{am1}) does not hold, $\int_{\frac{3}{4}}^{1} \dd$ is positive. Thus, similarly to (\ref{estteta})   for $p >At^{-1} > \frac{A}{\bar{C}}(\Phi(r))^{-1}$ there holds
 \[
 \int_{\frac{3}{4}}^{1} p^{\al} \dd \geq \int_{\frac{3}{4}}^{1} \left(\frac{A}{\bar{C}}\right)^{\al}(\vr)^{-\al} \dd \geq \left(\frac{A}{\bar{C}}\right)^{\frac{3}{4}}\int_{\frac{3}{4}}^{1} (\vr)^{-\al} \dd
 \]
 \[
 \geq \left(\frac{A}{\bar{C}}\right)^{\frac{3}{4}} c(\mu) \izj  (\vr)^{-\al} \dd = \left(\frac{A}{\bar{C}}\right)^{\frac{3}{4}} \frac{c(\mu)}{C_{1}}\theta,
\]
provided $\Phi(r) \leq 1$, where in the last estimate we used (\ref{estioddolu}).
Let us choose $A = A(\bar{C},C_{1},\mu)> \max\{1,\bar{C}\}$ such that
\[
\left(\frac{A}{\bar{C}}\right)^{\frac{3}{4}} \frac{c(\mu)}{C_{1}} > 4
\]
and fix it from now. Then,
\eqq{
\theta < \frac{1}{4} \int_{\frac{3}{4}}^{1} p^{\al} \dd \m { for } p > \frac{A}{t}.
}{tht4}
On the other hand, since $p\geq 1$ and the support of $\mu$ is not cut-off from one,
\[
\int_{0}^{\frac{1}{2}}p^{\al}\cos(\pi \al)\dd \leq p^{\frac{1}{2}} \int_{0}^{\frac{1}{2}}\cos(\pi \al)\dd = p^{-\frac{1}{4}}p^{\frac{3}{4}} \int_{0}^{\frac{1}{2}}\cos(\pi \al)\dd
\]
\[
\leq \frac{\int_{0}^{\frac{1}{2}}\cos(\pi \al)\dd}{p^{\frac{1}{4}} \int_{\frac{3}{4}}^{1}(-\cos(\pi\al))\dd} \int_{\frac{3}{4}}^{1}p^{\al}(-\cos(\pi\al))\dd.
\]
Observe that for $p >At^{-1}$ and $t < \bar{C}\Phi(r)$ there holds
\[
p^{-\frac{1}{4}} \leq \frac{(\bar{C}\vr)^{\frac{1}{4}}}{A^{\frac{1}{4}}} \leq \vr^{\frac{1}{4}}.
\]
Since $\vr$ is increasing and $\Phi(0)=0$, then there exists a positive $r^{*}$ depending only on $\mu$ such that for every $r\in (0,r^*]$ and $t < \bar{C}\Phi(r)$, and $p >At^{-1}$ we have
\[
\int_{0}^{\frac{1}{2}}p^{\al}\cos(\pi \al)\dd
\leq \frac{1}{2}\int_{\frac{3}{4}}^{1}p^{\al}(-\cos(\pi\al))\dd \leq \frac{1}{2}\int_{\frac{1}{2}}^{1}p^{\al}(-\cos(\pi\al))\dd.
\]
Thus, for such $p$ there holds
\[
\izj p^{\al} \cos(\pi \al)\dd \leq \frac{1}{2}\int_{\frac{1}{2}}^{1}p^{\al}(\cos(\pi\al))\dd \leq -\frac{\sqrt{2}}{4}\int_{\frac{3}{4}}^{1}p^{\al}\dd.
\]

Henceforth, we discuss only $0<r<r^{*}$. Using the estimate above together with (\ref{tht4}) and  (\ref{estioddolu}) we arrive at
\[
\abs{\theta + \izj p^{\al}\cos(\pi\al)\dd} \geq \frac{\sqrt{2}-1}{4}\int_{\frac{3}{4}}^{1}p^{\al}\dd \geq c(\mu)\int_{0}^{1}p^{\al}\dd \m{ for } p > \frac{A}{t}.
\]
Applying this estimate in (\ref{ht1}) we obtain
\[
\int_{At^{-1}}^{\infty} \frac{1 - e^{-pt}}{p} H_{\theta}(p) dp
\leq c(\mu)\int_{At^{-1}}^{\infty} \frac{\izj p^{\al}\sin(\pi\al)\dd}{\left(\izj p^{\al}\dd \right)^{2}}\frac{1}{p} dp = \podst{w = pt}{dw = t dp} =
\]
\[
=c(\mu)\int_{A}^{\infty} \frac{\izj w^{\al}t^{-\al}\sin(\pi\al)\dd}{\left(\izj w^{\al}t^{-\al}\dd \right)^{2} w} dw \leq c(\mu)\int_{A}^{\infty} \frac{\izj t^{-\al}\sin(\pi\al)\dd}{\left(\int_{\gmb}^{1} w^{\al}t^{-\al}\dd \right)^{2}} dw.
\]
Since the support of $\mu$ is not cut-off from one, we may take $\gmb > \frac{1}{2}$ and we have further
\eqq{
\int_{At^{-1}}^{\infty} \frac{1 - e^{-pt}}{p} H_{\theta}(p) dp \leq c(\mu) \int_{A}^{\infty} w^{-2\gmb} dw \frac{\izj t^{-\al}\sin(\pi\al)\dd}{\left(\int_{\gmb}^{1} t^{-\al}\dd \right)^{2}} \leq c(C_{1},\bar{C},\mu)  \frac{\izj t^{-\al}\sin(\pi\al)\dd}{\left(\int_{0}^{1} t^{-\al}\dd \right)^{2}},
}{ht2}
where in the last estimate we used (\ref{estioddolu}).
Let us now estimate
\[
t \int_{0}^{At^{-1}}e^{-pt} H_{\theta}(p) dp
\]
from below. We have
\[
t \int_{0}^{At^{-1}}e^{-pt} H_{\theta}(p) dp \geq e^{-A} t \int_{0}^{At^{-1}} H_{\theta}(p) dp.
\]
We note that for $t \in (0,\bar{C}\Phi(r))$ there holds  $At^{-1} > A \bar{C}^{-1}(\Phi(r))^{-1}$. Let us choose $B = \frac{A}{2\bar{C}}$. Then, $At^{-1} > B(\Phi(r))^{-1}$ for every $t \in (0,\bar{C}\Phi(r))$ and
\[
t \int_{0}^{At^{-1}}e^{-pt} H_{\theta}(p) dp  \geq e^{-A} t \int_{\frac{B}{\Phi(r)}}^{At^{-1}} H_{\theta}(p) dp.
\]
For every $p > \frac{B}{\Phi(r)}$ we may apply (\ref{estijedentrzy}) and we arrive at
\[
t \int_{0}^{At^{-1}}e^{-pt} H_{\theta}(p) dp  \geq c(C_{1},\bar{C},\mu) t \int_{\frac{B}{\Phi(r)}}^{At^{-1}} \frac{\izj p^{\al} \sin (\pi \al)\dd}{(\izj p^{\al}\dd)^{2}}dp = \podst{w = pt}{dw = t dp}
\]
\[
=c(C_{1},\bar{C},\mu)  \int_{\frac{Bt}{\Phi(r)}}^{A} \frac{\izj w^{\al} t^{-\al}\sin (\pi \al)\dd}{(\izj w^{\al}t^{-\al}\dd)^{2}}dw \geq c(C_{1},\bar{C},\mu) \int_{B\bar{C}}^{A} \frac{\izj w^{\al} t^{-\al}\sin (\pi \al)\dd}{(\izj w^{\al}t^{-\al}\dd)^{2}}dw
\]
\[
=c(C_{1},\bar{C},\mu) \int_{\frac{A}{2}}^{A} \frac{\izj w^{\al} t^{-\al}\sin (\pi \al)\dd}{(\izj w^{\al}t^{-\al}\dd)^{2}}dw \geq c(C_{1},\bar{C},\mu) \frac{1}{A^{2}}\int_{\frac{A}{2}}^{A} \frac{\izj  t^{-\al}\sin (\pi \al)\dd}{(\izj t^{-\al}\dd)^{2}}dw
\]
\[
\ge  c(C_{1},\bar{C},\mu) \frac{1}{A^{2}} \frac{A}{2} \frac{\izj  t^{-\al}\sin (\pi \al)\dd}{(\izj t^{-\al}\dd)^{2}} = c(C_{1},\bar{C},\mu) \frac{\izj  t^{-\al}\sin (\pi \al)\dd}{(\izj t^{-\al}\dd)^{2}}.
\]
Combining this result with (\ref{ht2}) we obtain (\ref{hth}). Applying (\ref{hth}) in (\ref{hth3}) we obtain that there exists $c=c(C_{1},\bar{C},\mu)$ such that for any $t \in (0,\bar{C}\Phi(r))$ and $r\in (0,r^*]$
\[
\pi \izt r_{\theta}(\tau)d\tau \leq c  t \int_{0}^{At^{-1}}e^{-pt} H_{\theta}(p) dp \leq c  t \int_{0}^{\infty}e^{-pt} H_{\theta}(p) dp = c \pi t r_{\theta}(t).
\]
This shows (\ref{intrts}). Now (\ref{intrt}) together with (\ref{ldol}) gives
\[
r_{\theta}(t) \geq c l(t) \geq c \frac{1}{\izj t^{1-\al}\dd} \m{ for every } t \in (0,\bar{C}\Phi(r)),  r\in (0,r^*]
\]
where $c = c(\bar{C},C_{1},\mu)$ and $r^{*}$ depends only on $\mu$. Thus, we obtain (\ref{rtheta1}) for  $\mu$ not satisfying (\ref{am1}) and then the estimates (\ref{rtheta2}) are deduced as in the previous case.
\end{proof}

\subsection{Moser iterations and an abstract lemma of Bombieri and Giusti}
In  this subsection, let $U_\sigma$, $0<\sigma\le 1$, denote
a family of measurable subsets of a fixed finite measure space
endowed with a measure $\varpi$, such that $U_{\sigma'}\subset
U_\sigma$ if $\sigma'\le \sigma$. For $p\in (0,\infty)$ and
$0<\sigma\le 1$, by $L_p(U_\sigma)$ we mean the Lebesgue space
$L_p(U_\sigma,d\varpi)$ of all $\varpi$-measurable functions
$f:U_\sigma\rightarrow \iR$ with
$\|f\|_{L_p(U_\sigma)}:=(\int_{U_\sigma}|f|^p\,d\varpi)^{1/p}<\infty$.

Our proof of the weak Harnack inequality relies on methods used in the proof of the weak Harnack inequality for single order fractional derivative \cite{base}. Here we recall several general lemmas on Moser iterations (see \cite[Lemma 2.3 and Lemma 2.5]{CZ}, \cite[Lemma 2.1 and Lemma 2.2]{base}) and the lemma of Bombieri and Giusti (see \cite{BomGiu}, \cite[Lemma 2.6]{CZ}, \cite[Lemma 2.2.6]{SalCoste}).
In contrast to the treatment in \cite{base}, we need to control the power of the constant $C$ in the inequalities resulting from the iteration process.
For this reason, and for the convenience of the reader, we repeat the proof of these lemmas taking additional care of the constant $C$.

\begin{lemma} \label{moserit1}
Let $\kappa>1$, $\bar{p}\ge 1$, $C>0$, and $a>0$. Suppose
$f$ is a $\varpi$-measurable function on $U_1$ such that
\begin{equation} \label{mositer1}
\|f\|_{L_{\gamma\kappa}(U_{\sigma'})}\le
\Big(\frac{C(1+\gamma)^{a}}{(\sigma-\sigma')^{a}}\Big)^{1/\gamma}\,\|f\|_{L_{\gamma}(U_{\sigma})},
\quad 0<\sigma'<\sigma\le 1,\;\gamma>0.
\end{equation}
Then there exists a constant $M=M(a,\kappa,\bar{p})>0$  such that
\[
\esup_{U_{ \varsigma  }}{|f|} \le
\Big(\frac{M C^{\frac{\kappa}{\kappa-1}}}{(1-\varsigma)^{a_0}}\Big)^{1/p}
\|f\|_{L_{p}(U_1)}\quad \mbox{for all}\;\;\varsigma\in(0,1),\;p\in
(0,\bar{p}] ,\]
where $a_{0}= \frac{a \kappa }{\kappa -1}$.
\end{lemma}


\begin{proof}
For $q>0$ and $0<\sigma\le 1$, let
\[ \Phi(q,\sigma)=(\int_{U_\sigma} |f|^q\,d\mu)^{1/q}.\]
Let $0<p\le \bar{p}$ and $\varsigma \in (0,1)$. Set $p_i=p\kappa^{i}$,
$i=0,1,\ldots$ and define the sequence $\{\sigma_i\}$,
$i=0,1,\ldots$, by $\sigma_0=1$ and $\sigma_i=1-\sum_{j=1}^i 2^{-j}
(1-\varsigma )$, $i=1,2,\ldots$; observe that
$1=\sigma_0>\sigma_1>\ldots>\sigma_i>\sigma_{i+1}>\varsigma $ as well as
$\sigma_{i-1}-\sigma_{i}=2^{-i}(1-\varsigma )$, $i\ge 1$. Let now
$n\in \iN$. By employing (\ref{mositer1}) with $\gamma=p_i$,
$i=0,1,\ldots,n-1$, we get that
\begin{align*}
\Phi(p_n,\varsigma )& \;\le
\Phi(p_n,\sigma_n)\;=\;\Phi(p_{n-1}\kappa,\sigma_n)\;\le\;
\Big(\frac{C(1+p\kappa^{n-1})^{a}}{[2^{-n}(1-\varsigma )]^{a}}\Big)^
{\frac{1}{p}\,\kappa^{-(n-1)}}\Phi(p_{n-1},\sigma_{n-1})\\ & \;\le
\Big(\frac{C(2\bar{p}\kappa^{n-1})^{a}}{[2^{-n}(1-\varsigma )]^{a}}\Big)^
{\frac{1}{p}\,\kappa^{-(n-1)}}\Phi(p_{n-1},\sigma_{n-1})\\ &\;\le
\Big(\frac{C\tilde{C}(\bar{p},a)^n
\kappa^{a(n-1)}}{(1-\varsigma )^{a}}\Big)^
{\frac{1}{p}\,\kappa^{-(n-1)}}\Phi(p_{n-1},\sigma_{n-1})\;\le\;\ldots\\
& \;\le \Big(C^{\sum_{j=0}^{n-1}
\kappa^{-j}}\tilde{C}^{\sum_{j=0}^{n-1}
(j+1)\kappa^{-j}}\kappa^{a\sum_{j=0}^{n-1}
j\kappa^{-j}}(1-\varsigma )^{-a\sum_{j=0}^{n-1}
\kappa^{-j}}\Big)^{1/p}\,\Phi(p_0,\sigma_0)\\ & \; \le
\Big(\frac{C^{\frac{\kappa}{\kappa-1}}M(\bar{p},a,\kappa)}{(1-\varsigma )^{\frac{a\kappa}{\kappa-1}}}\Big)^
{1/p}\,\Phi(p,1).
\end{align*}
We now send $n$ to $\infty$ and use the fact that
\[ \lim_{n\to\infty}\Phi(p_n,\varsigma )=\esup_{U_{\varsigma }}{|f|}\]
to obtain that
\[ \esup_{U_{\varsigma }}{|f|} \le
\Big(\frac{C^{\frac{\kappa}{\kappa-1}}M(\bar{p},a,\kappa)}{(1-\varsigma )^{\frac{a\kappa}{\kappa-1}}}\Big)^
{1/p}\,\|f\|_{L_{p}(U_1)}.\] Hence the proof is complete.
\end{proof}

\noindent The second Moser iteration result is the following.
\begin{lemma} \label{moserit2}
Assume that $\varpi(U_1)\le 1$. Let $\kappa>1$, $0<p_0<\kappa$, and
$C>0,\,a>0$. Suppose $f$ is a $\varpi$-measurable function on
$U_1$ such that
\begin{equation} \label{mositer2}
\|f\|_{L_{\gamma\kappa}(U_{\sigma'})}\le
\Big(\frac{C}{(\sigma-\sigma')^{a}}\Big)^{1/\gamma}\,\|f\|_{L_{\gamma}(U_{\sigma})},
\quad 0<\sigma'<\sigma\le 1,\;0<\gamma\le \frac{p_0}{\kappa}<1.
\end{equation}
Then
\[ \|f\|_{L_{p_0}(U_{\varsigma})}\le
\Big(\frac{M}{(1-\varsigma)^{a_0}}\Big)^{1/p-1/p_0}
\|f\|_{L_{p}(U_1)}\quad \mbox{for all}\;\;\varsigma\in(0,1),\;p\in
(0,\frac{p_0}{\kappa}],\]
where $M=C^{\frac{\kappa (\kappa+1)}{\kappa-1}}\cdot 2^{\frac{a \kappa^{3}}{(\kappa-1)^{3}}}$ and $a_{0}= \frac{a \kappa (\kappa+1)}{\kappa-1}$.
\end{lemma}
This follows immediately from \cite[Lemma 2.2]{base}  and its proof.

$\mbox{}$

Let us now recall the lemma by Bombieri and Giusti.
\begin{lemma} \label{abslemma}
Let $\delta,\,\eta\in(0,1)$, and let $\gamma,\,C$ be positive
constants and $0<\beta_0\le \infty$. Suppose $f$ is a positive
$\varpi$-measurable function on $U_1$ which satisfies the following two
conditions:

(i)
\[
\|f\|_{L_{\beta_0}(U_{\sigma'})}\le
[C(\sigma-\sigma')^{-\gamma}\varpi(U_1)^{-1}]^{1/\beta-1/\beta_0}\|f\|_{L_{\beta}(U_{\sigma})},
\]
for all $\sigma,\,\sigma',\,\beta$ such that $0<\delta\le
\sigma'<\sigma\le 1$ and $0<\beta\le \min\{1,\eta\beta_0\}$.

(ii)
\[
\varpi(\{ \log f>\lambda\}) \le C\varpi(U_1)\lambda^{-1}
\]
for all $\lambda>0$.

Then
\[
\|f\|_{L_{\beta_0}(U_{\delta})}\le M \varpi(U_1)^{1/\beta_0},
\]
where $M$ depends only on $\delta,\,\eta,\,\gamma,\,C$, and
$\beta_0$.
\end{lemma}
\subsection{The Yosida approximation of the non-local operator} \label{SecYos}
In this section we first introduce the Yosida approximation of the non-local operator of the form $\frac{d}{dt}k*$. The special case $k(t) = \frac{t^{-\al}}{\Gamma(1-\al)}$ with $\alpha\in (0,1)$ has been discussed in \cite{base}. Here, we repeat the reasoning for general $k$ for the reader's convenience, see also \cite{Za2}.

Let  $1\le p<\infty$, $T>0$, and $X$ be a real Banach
space. Then the non-local operator $B$ defined by
\[ B u=\,\frac{d}{dt}\,(k\ast u),\;\;D(B)=\{u\in L_p((0,T);X):\,k\ast u\in \mbox{}_0 H^1_p((0,T);X)\},
\]
where the zero means vanishing at $t=0$, is known to be
$m$-accretive in $L_p((0,T);X)$, cf. \cite{Phil1}, \cite{CP}, and
\cite{Grip1}. Its Yosida approximations $B_{n}$, given by
$B_{n}=nB(n+B)^{-1},\,n\in \iN$, have the property that for any
$u\in D(B)$, one has $B_{n}u\rightarrow Bu$ in $L_p((0,T);X)$ as
$n\to \infty$. Furthermore, we have the representation
\begin{equation} \label{Yos}
B_n u=\,\frac{d}{dt}\,(k_{n}\ast u),\quad u\in
L_p((0,T);X),\;n\in \iN,
\end{equation}
where $k_{n}=n s_{n}$, and $s_{n}$ is the
unique solution of the scalar-valued Volterra equation
\[
s_{n}(t)+n(s_{n}\ast l)(t)=1,\quad
t>0,\;n\in\iN,
\]
see e.g. \cite{VZ}. Denote by $h_{n}\in L_{1,\,loc}(\iR_+)$ the
resolvent kernel associated with $nl$, that is
\begin{equation} \label{hndef}
h_{n}(t)+n(h_{n}\ast l)(t)=nl(t),\quad
t>0,\;n\in\iN.
\end{equation}
Convolving (\ref{hndef}) with $k$ and using $l\ast
k=1$, we find that
\[
(k\ast h_{n})(t)+n([k\ast
h_{n}]\ast l)(t)=n,\quad t>0,\;n\in\iN.
\]
Consequently,
\begin{equation} \label{gnprop}
k_{n}=ns_{n}=k\ast h_{n},\quad
n\in \iN.
\end{equation}
The kernels $k_{n}$ are nonnegative and nonincreasing for
all $n\in\iN$, and they belong to $H^1_1((0,T))$, cf. \cite{JanI}
and \cite{VZ}. Note that for any $f\in L_p((0,T);X)$, $1\le
p<\infty$, we have $h_{n}\ast f\to f$ in $L_p((0,T);X)$
as $n\to \infty$. In fact, setting $u=l\ast f$, we have $u\in
D(B)$, and
\[
B_n u=\,\frac{d}{dt}\,(k_{n}\ast
u)=\,\frac{d}{dt}\,(k\ast l\ast h_{n}\ast
f)=h_{n}\ast f\,\to\,Bu=f\quad\mbox{in}\;L_p((0,T);X)
\]
as $n\to \infty$. This implies in particular that $k_{n}\to k$ in
$L_1((0,T))$ as $n\to \infty$.

Next, we recall a fundamental identity for integro-differential
operators of the form $\frac{d}{dt}(k\ast u)$, cf.\ also \cite{Za2}, \cite{base}.
Suppose that $k\in H^1_1((0,T))$, $U$ is an open subset of $\iR$, and $H\in C^1(U)$. Then 
for any sufficiently smooth
function $u$ on $(0,T)$ taking values in $U$, we have for a.a. $t\in (0,T)$
\begin{align} \label{fundidentity}
H'(u(t))&\frac{d}{dt}\,(k \ast u)(t) =\;\frac{d}{dt}\,(k\ast
H(u))(t)+
\Big(-H(u(t))+H'(u(t))u(t)\Big)k(t) \nonumber\\
 & +\int_0^t
\Big(H(u(t-s))-H(u(t))-H'(u(t))[u(t-s)-u(t)]\Big)[-\dot{k}(s)]\,ds,
\end{align}
where $\dot{k}$ denotes the derivative of $k$. In particular this
identity applies to the Yosida approximations of the non-local operator. An integrated version of
(\ref{fundidentity}) can be found in \cite[Lemma 18.4.1]{GLS}.

We will  frequently
use that if $k$ is nonincreasing and $H$ is convex then the last term in (\ref{fundidentity}) is nonnegative. However, we would like to point out that, exactly as in \cite{base},
for the delicate logarithmic estimates in Section \ref{logsection}
one really needs the full identity (\ref{fundidentity}), even though
we work all the time with convex or concave functions. In
particular, similarly as in \cite{base}, the crucial fractional differential inequality
(\ref{log7}) cannot be obtained by using merely convexity
inequalities.

In view of the regularity of $l$ established in Remark \ref{lbemerk}, the subsequent two lemmas are also obtained by simple algebra. They are straightforward generalization of \cite[Lemma 2.4 and Lemma 2.5]{base}.
\begin{lemma} \label{comm}
Let $T>0$. Suppose that $v\in
{}_0H^1_1((0,T))$ and $\varphi\in C^1([0,T])$. Then
\[
\big(l\ast(\varphi \dot{v}))(t)=\varphi(t)(l\ast
\dot{v})(t)+\int_0^t
v(\sigma)\partial_\sigma\big(l(t-\sigma)
[\varphi(t)-\varphi(\sigma)]\big)\,d\sigma,\;\;\mbox{a.a.}\;t\in
(0,T).
\]
If in addition $v$ is nonnegative and $\varphi$ is nondecreasing
there holds
\[
\big(l\ast(\varphi \dot{v}))(t)\ge \varphi(t)(l\ast
\dot{v})(t)-\int_0^t l(t-\sigma)
\dot{\varphi}(\sigma)v(\sigma)\,d\sigma,\;\;\mbox{a.a.}\;t\in
(0,T).
\]
\end{lemma}
\begin{lemma} \label{comm2}
Let $T>0$, $k\in H^1_1((0,T))$, $v\in L_1((0,T))$, and $\varphi\in
C^1([0,T])$. Then
\[
\varphi(t)\,\frac{d}{dt}\,(k\ast v)(t)=\,\frac{d}{dt}\,\big(k\ast
[\varphi v]\big)(t)+\int_0^t
\dot{k}(t-\tau)\big(\varphi(t)-\varphi(\tau)\big)v(\tau)\,d\tau,\;\;\mbox{a.a.}\;t\in
(0,T).
\]

\end{lemma}
\subsection{An embedding result and a weighted Poincar\'e inequality}
We finish this chapter by recalling a fundamental result on parabolic embeddings and a weighted Poincar\'e inequality. We will apply these tools in a similar manner as in \cite{base}.

The following embedding result is a particular case of \cite[Proposition 2.1]{VZd}.
\nic{The case $p=\infty$ is contained, e.g., in \cite[p.\ 74 and
75]{LSU}.}
\begin{prop}
Let $T>0$ and $\Omega$ be a bounded domain in $\iR^N$ and assume that $\partial \Om$ satisfies the property of positive density. For $1<p\le
\infty$ we define the space
\begin{equation} \label{Vdef}
V_p:=V_p((0,T)\times \Omega)=L_{2p}((0,T);L_2(\Omega))\cap
L_2((0,T);H^1_2(\Omega)),
\end{equation}
endowed with the norm
\[
\|u\|_{V_p((0,T)\times \Omega)}:=\|u\|_{L_{2p}((0,T);L_2(\Omega))}
+\|Du\|_{L_2((0,T);L_2(\Omega))}.
\]
\nic{Set
\begin{equation} \label{kappa}
\kappa:=\kappa_p:=\,\frac{2p+N(p-1)}{2+N(p-1)}
\end{equation}
with $\kappa_\infty=1+2/N$. }
Then $V_p\hookrightarrow
L_{2\kappa}((0,T)\times\Omega)$, and for all $u\in V_p\cap L_2((0,T);\oH^1_2(\Omega))$
\eqq{
\norm{u}_{L_{2\kappa}((0,T)\times\Omega)}
\leq C(N)  \norm{Du}^{\theta}_{L_{2}((0,T)\times\Omega)}  \norm{u}_{L_{2p}((0,T);L_{2}(\Omega))}^{1-\theta}
}{interapara}
where
\eqq{
\kappa:=\kappa_{p} := \frac{2p+N(p-1)}{2+N(p-1)}, \hd \hd \theta = \frac{N(p-1)}{N(p-1)+2p}
}{defkappatheta}
with $\kappa_\infty=1+2/N$.
\end{prop}

The following result can be found in \cite[Lemma 3]{Moser64}, see
also \cite[Lemma 6.12]{Lm}.
\begin{prop} \label{WeiPI}
Let $\varphi\in C(\iR^N)$ with non-empty compact support of diameter
$d$ and assume that $0\le \varphi\le 1$. Suppose that the domains
$\{x\in\iR^N:\varphi(x)\ge a\}$ are convex for all $a\le 1$. Then
for any function $u\in H^{1}_2(\iR^N)$,
\[
\int_{\iR^N} \big(u(x)-u_\varphi\big)^2 \varphi(x)\,dx \le \,\frac{2
d^2\nuN(\mbox{{\em supp}}\,\varphi)}{|\varphi|_{L_1(\iR^N)}}\,
\int_{\iR^N} |Du(x)|^2 \varphi(x)\,dx,
\]
where
\[
u_\varphi=\frac{\int_{\iR^N} u(x)\varphi(x)\,dx}{\int_{\iR^N}
\varphi(x)\,dx}.
\]
\end{prop}
\section{Proof of the weak Harnack inequality}
\subsection{The regularized weak formulation and  time shifts}\label{SSS}
We recall a lemma which provides an
equivalent weak formulation of (\ref{MProb}). The idea is to replace the singular
kernel $k$ by its more regular approximation
$k_{n}$ ($n\in\iN$). Here, $k_n, h_n$, $n\in\iN$, are defined as
in Section \ref{SecYos}. This lemma plays an important role in deriving {\em a priori} estimates
for weak (sub-/super-) solutions of (\ref{MProb}).
\begin{lemma} \label{LemmaReg}
Let  $T>0$, and $\Omega\subset \iR^N$ be a bounded
domain. Suppose the assumptions (H1)--(H3) are satisfied and $f$ is bounded on $\Om_{T}$. Then $u\in
Z$ is a weak solution (subsolution, supersolution) of
(\ref{MProb}) in $\Omega_T$ if and only if for any nonnegative
function $\psi\in \oH^1_2(\Omega)$ one has
\begin{equation} \label{LemmaRegF}
\int_\Omega \Big(\psi \partial_t[k_{n}\ast
(u-u_0)]+(h_n\ast [ADu]|D\psi)\Big)\,dx\nonumber\\
=\,(\le,\,\ge)\,\int_{\Om}(h_{n}*f)\psi dx ,\quad\mbox{a.a.}\;t\in (0,T),\,n\in \iN.
\end{equation}
\end{lemma}
The proof of Lemma \ref{LemmaReg} is exactly the same as the proof of Lemma 3.1 in \cite{base}.
Analogously as in \cite{base}, if $u\in Z$ is a weak supersolution of (\ref{MProb}) in
$\Omega_T$ and $u_0\ge 0$ in $\Omega$, then
\begin{equation} \label{u0weg}
\int_\Omega \Big(\psi \partial_t(k_{n}\ast u)+(h_n\ast
[ADu]|D\psi)\Big)\,dx \ge \,\int_{\Om}(h_{n}*f)\psi dx,\quad\mbox{a.a.}\;t\in (0,T),\,n\in
\iN,
\end{equation}
for any nonnegative function $\psi\in \oH^1_2(\Omega)$.

Let now $t_1\in (0,T)$ be fixed. For $t\in (t_1,T)$ we introduce the
shifted time $s=t-t_1$ and put $\tilde{g}(s)=g(s+t_1)$, $s\in
(0,T-t_1)$, for functions $g$ defined on $(t_1,T)$. Using the
decomposition
\[
(k_{n}\ast u)(t,x)=\int_{t_1}^t
k_{n}(t-\tau)u(\tau,x)\,d\tau+\int_{0}^{t_1}
k_{n}(t-\tau)u(\tau,x)\,d\tau,\quad t\in (t_1,T),
\]
we then see that
\begin{equation} \label{shiftprop}
\partial_t(k_{n}\ast u)(t,x)=\partial_s(k_{n}\ast
\tilde{u})(s,x)+\int_0^{t_1}\dot{k}_{n}(s+t_1-\tau)u(\tau,x)\,d\tau.
\end{equation}
Assuming in addition that $u\ge 0$ on $(0,t_1)\times \Omega$ it
follows from (\ref{u0weg}), (\ref{shiftprop}), and the positivity
of $\psi$ and of $-\dot{k}_{n}$ that
\begin{equation} \label{shiftprob}
\int_\Omega \Big(\psi \partial_s(k_{n}\ast
\tilde{u})+\big((h_n\ast [ADu])\,\tilde{}\;|D\psi\big)\Big)\,dx \ge
\,\int_{\Om}(h_{n}*f)\,\tilde{} \, \psi dx,\quad\mbox{a.a.}\;s\in (0,T-t_1),\,n\in \iN,
\end{equation}
for any nonnegative function $\psi\in \oH^1_2(\Omega)$. This
relation will be the starting point for all of the estimates below.

\subsection{Mean value inequalities} \label{mvi}
For simplicity of the notation, for $r>0$ we set $B_r(x):=B(x, r)$. Recall that
$\nuN$ denotes the Lebesgue measure in $\iR^N$.
\begin{satz} \label{superest1}
Let  $\Omega\subset \iR^N$ be a bounded
domain and $T>0$. Suppose the assumptions (H1)--(H3) are satisfied and let $\eta>0$ and $\delta\in (0,1)$ be fixed. Then there exists
$r^{*} =r^{*}(\mu)>0 $ such that for any $0<r \leq r^{*}$, any $t_0\in(0,T]$
 with $t_0-\eta \vdr \ge 0$, any ball
$B_r(x_{0})\subset\Omega$, and any weak supersolution $u\ge
\varepsilon>0$ of (\ref{MProb}) in $(0,t_0)\times B_r(x_{0})$ with $u_0\ge 0$
in $B_r(x_{0})$  and $f\equiv 0$, there holds
\[
\esup_{U_{\sigma'}}{u^{-1}} \le \Big(\frac{C \nuNj(U_1)^{-1}
}{(\sigma-\sigma')^{\tau_0}}\Big)^{1/\gamma}
\|u^{-1}\|_{L_{\gamma}(U_\sigma)},\quad \delta\le \sigma'<\sigma\le
1,\; \gamma\in (0,1].
\]
Here $U_\sigma=(t_0-\sigma\eta \vdr,t_0)\times B_{\sigma r}(x_0)$,
$0<\sigma\le 1$, $C=C(\nu,\Lambda,\delta,\eta,\mu,N)$ and
$\tau_0=\tau_0(\mu,N)$.
\end{satz}
{\em Proof:}
In the proof we follow the idea of the proof of \cite[Theorem 3.1]{base}. The main novelty here is to introduce the cylinders with the shape dependent on the kernel $k$. Since the problem which we consider lacks a natural scaling, one has to treat the terms which depend on the radius $r$ very carefully and finally apply the crucial Lemma \ref{scaling}.
Since here, we only consider balls centered at fixed $x_0$, we abbreviate the notation $B_{r} := B_{r}(x_{0})$.

Let us fix $\sigma'$ and $\sigma$ such that $\delta\le \sigma'<\sigma\le
1$ and for $\rho\in (0,1]$ set
$V_\rho=U_{\rho\sigma}$. For any fixed $0<\rho'<\rho\le 1$, $r> 0$, we set
$t_{1} = t_{0}-\rho\sigma\eta\vdr$ and $t_{2} = t_{0}-\rho'\sigma\eta\vdr$. Clearly, we have $0\le t_1<t_2<t_0$. Further, we introduce the
shifted time ${s}=t-t_1$ and for functions $f$ defined on $(t_1,t_0)$, we set $\tilde{f}(s)=f(s+t_1)$, $s\in
(0,t_0-t_1)$ . Since
$u_0\ge 0$ in $B_r$ and $u$ is a positive weak supersolution of
(\ref{MProb}) in $(0,t_0)\times B_r$, we have (cf. (\ref{shiftprob}))
\begin{equation} \label{sup0}
\int_{B_r} \Big(v \partial_s(k_{n}\ast \tilde{u})+\big((h_n\ast
[ADu])\,\tilde{}\;|Dv\big)\Big)\,dx \ge \,0,\quad\mbox{a.a.}\;s\in
(0,t_0-t_1),\,n\in \iN,
\end{equation}
for any nonnegative function $v\in \oH^1_2(B_r)$. For $s\in
(0,t_0-t_1)$ we introduce the cut-off function $\psi\in C^1_0(B_{r\sigma})$ so that
$0\le \psi\le 1$, $\psi=1$ in $B_{\rho' r\sigma}$, supp$\,\psi\subset B_{\rho r\sigma}$, and $|D \psi|\le 2/[r \sigma (\rho-\rho')]$. Then, we choose  in (\ref{sup0}) the test function $v=\psi^2
\tilde{u}^{\beta}$ with $\beta<-1$. Applying the fundamental
identity (\ref{fundidentity}) with $k=k_{n}$ and the
convex function $H(y)=-(1+\beta)^{-1}y^{1+\beta}$, $y>0$, we have for a.a. $(s,x)\in (0,t_0-t_1)\times B_r$
\begin{align}
 -\tilde{u}^{\beta}\partial_{s}(k_{n}\ast \tilde{u}) & \ge -\,\frac{1}{1+\beta}\,\partial_{s}
(k_{n}\ast\tilde{u}^{1+\beta})+\Big(\frac{\tilde{u}^{1+\beta}}{1+\beta}\,-\tilde{u}^{1+\beta}\Big)k_{n}\nonumber\\
 & =
-\,\frac{1}{1+\beta}\,\partial_{s}
(k_{n}\ast\tilde{u}^{1+\beta})-\,\frac{\beta}{1+\beta}\,\tilde{u}^{1+\beta}
k_{n}. \label{sup1}
\end{align}
Moreover, there holds
\[ Dv=2\psi D\psi \,\tilde{u}^{\beta}+\beta\psi^2 \tilde{u}^{\beta-1}D \tilde{u}.\]
Applying this, together with (\ref{sup1}), in (\ref{sup0}) we arrive at
\begin{align}
-\,\frac{1}{1+\beta}\,& \int_{B_{r\sigma}}\psi^2\partial_{s}
(k_{n}\ast\tilde{u}^{1+\beta})\,dx+|\beta|\int_{B_{r\sigma}}\big((h_n\ast
[ADu])\,\tilde{}\;|\psi^2 \tilde{u}^{\beta-1}D
\tilde{u}\big)\,dx \nonumber\\
\le  & \,2\int_{B_{r\sigma}}\big((h_n\ast [ADu])\,\tilde{}\;|\psi D\psi
\,\tilde{u}^{\beta}\big)\,dx+\,\frac{\beta}{1+\beta}\,\int_{B_{r\sigma}}\psi^2\tilde{u}^{1+\beta}
k_{n}\,dx \label{sup2}
\end{align}
for a.a. $s\in (0,t_0-t_1)$. Next, choose $\varphi\in C^1([0,t_0-t_1])$ such that $0\le
\varphi\le 1$, $\varphi=0$ in $[0,(t_2-t_1)/2]$, $\varphi=1$ in
$[t_2-t_1,t_0-t_1]$, and $0\le \dot{\varphi}\le 4/(t_2-t_1)$.
Let us multiply (\ref{sup2}) by $-(1+\beta)\varphi(s)>0$
and convolve the resulting inequality with $l$, then
\begin{align}
\int_{B_{r\sigma}} & l\ast
\big(\varphi\partial_{s}(k_{n}\ast
[\psi^2\tilde{u}^{1+\beta}])\big)\,dx+\beta(1+\beta)\,l\ast\int_{B_{r\sigma}}\big((h_n\ast
[ADu])\,\tilde{}\;|\psi^2 \tilde{u}^{\beta-1}D
\tilde{u}\big)\varphi\,dx \nonumber\\
\le \, & \,2|1+\beta|\,l\ast\int_{B_{r\sigma}}\big((h_n\ast
[ADu])\,\tilde{}\;|\psi D\psi
\,\tilde{u}^{\beta}\big)\varphi\,dx+|\beta|\,l\ast\int_{B_{r\sigma}}\psi^2\tilde{u}^{1+\beta}
k_{n}\varphi\,dx, \label{sup3}
\end{align}
for a.a. $s\in (0,t_0-t_1)$. Applying Lemma \ref{comm}, we have
\begin{align}
\int_{B_{r\sigma}} l\ast &
\big(\varphi\partial_{s}(k_{n}\ast
[\psi^2\tilde{u}^{1+\beta}])\big)\,dx \ge \int_{B_{r\sigma}} \varphi
l\ast \big(\partial_{s}(k_{n}\ast
[\psi^2\tilde{u}^{1+\beta}])\big)\,dx\nonumber\\
& -\int_0^s l(s-\tau)\dot{\varphi}(\tau)
\big(k_{n}\ast
\int_{B_{r\sigma}}\psi^2\tilde{u}^{1+\beta}\,dx\big)(\tau)\,d\tau.
\label{sup4}
\end{align}
Moreover,
\begin{equation} \label{sup5}
l\ast \partial_{s}(k_{n}\ast
[\psi^2\tilde{u}^{1+\beta}])=\partial_s(l\ast
k_{n}\ast [\psi^2\tilde{u}^{1+\beta}])=h_n\ast
(\psi^2\tilde{u}^{1+\beta}),
\end{equation}
since $k_{n}=k\ast h_n$, $l\ast
k=1$ and
\[
k_{n}\ast [\psi^2\tilde{u}^{1+\beta}]\in
{}_0H^1_1([0,t_0-t_1];L_1(B_{r\sigma})).
\]

Combining (\ref{sup3}), (\ref{sup4}), and (\ref{sup5}), passing to the limit with
$n$ (if necessary on an appropriate subsequence), we obtain the following estimate
\begin{align}
& \int_{B_{r\sigma}}\varphi\psi^2\tilde{u}^{1+\beta}\,dx+
\beta(1+\beta)\,l\ast\int_{B_{r\sigma}}\big(\tilde{A}D\tilde{u}|\psi^2
\tilde{u}^{\beta-1}D \tilde{u}\big)\varphi\,dx\nonumber\\
\le \, &
\,2|1+\beta|\,l\ast\int_{B_{r\sigma}}\big(\tilde{A}D\tilde{u}|\psi
D\psi
\,\tilde{u}^{\beta}\big)\varphi\,dx+|\beta|\,l\ast\int_{B_{r\sigma}}\psi^2\tilde{u}^{1+\beta}
k\varphi\,dx \nonumber\\
& +\int_0^s l(s-\tau)\dot{\varphi}(\tau)
\big(k\ast
\int_{B_{r\tau}}\psi^2\tilde{u}^{1+\beta}\,dx\big)(\tau)\,d\tau,
\;\;\mbox{a.a.}\;s\in(0,t_0-t_1).
\label{sup6}
\end{align}
Put $w=\tilde{u}^{\frac{\beta+1}{2}}$. Then $Dw=\frac{\beta+1}{2}
\tilde{u}^{\frac{\beta-1}{2}} D\tilde{u}$ and by (H2), we
may estimate
\begin{align}
\beta(1+\beta)\,l\ast\int_{B_{r\sigma}}\big(\tilde{A}D\tilde{u}|\psi^2
\tilde{u}^{\beta-1}D \tilde{u}\big)\varphi\,dx & \,\ge \nu
\beta(1+\beta)\,l\ast\int_{B_{r\sigma}} \varphi
\psi^2\tilde{u}^{\beta-1}|D\tilde{u}|^2\,dx \nonumber\\
& \, = \,\frac{4\nu \beta}{1+\beta}\,l\ast\int_{B_{r\sigma}}\varphi
\psi^2|Dw|^2\,dx. \label{sup7}
\end{align}
Using (H1) and Young's inequality we get
\begin{align}
2\big|\big(\tilde{A}D\tilde{u}|\psi D\psi
\,\tilde{u}^{\beta}\big)\varphi\big| & \le 2\Lambda\psi|D\psi|\,|D
\tilde{u}|\tilde{u}^\beta \varphi=2\Lambda\psi|D\psi|\,|D
\tilde{u}|\tilde{u}^{\frac{\beta-1}{2}}\tilde{u}^{\frac{\beta+1}{2}}\varphi\nonumber\\
& \le \,\frac{\nu |\beta|}{2}\, \psi^2\varphi |D \tilde{u}|^2
\tilde{u}^{\beta-1}+\,\frac{2}{\nu |\beta|}\,\Lambda^2
|D\psi|^2\varphi \tilde{u}^{\beta+1}\nonumber\\
& = \,\frac{2\nu |\beta|}{(1+\beta)^2}\,
\psi^2\varphi|Dw|^2+\,\frac{2}{\nu |\beta|}\,\Lambda^2
|D\psi|^2\varphi w^2. \label{sup8}
\end{align}
From (\ref{sup6}), (\ref{sup7}), and (\ref{sup8}) we conclude that
\begin{equation} \label{sup9}
\int_{B_{r\sigma}}\varphi\psi^2w^2\,dx+\,\frac{2\nu
|\beta|}{|1+\beta|}\,l\ast\int_{B_{r\sigma}}\varphi \psi^2|Dw|^2\,dx
\le l\ast F,\quad\mbox{a.a.}\;s\in(0,t_0-t_1),
\end{equation}
where
\eqnsl{
F(s) =\,  \,\frac{2\Lambda^2|1+\beta|}{\nu |\beta|}\, & \int_{B_{r\sigma}}
|D\psi|^2\varphi w^2\,dx
 +|\beta|\varphi(s)k(s)\int_{B_{r\sigma}}\psi^2 w^2
\,dx \\
& \,+\dot{\varphi}(s) \big(k\ast \int_{B_{r\sigma}}\psi^2
w^2\,dx\big)(s)\ge 0,\quad\mbox{a.a.}\;s\in(0,t_0-t_1).
}{defeF}
We note that both terms on the left-hand-side of (\ref{sup9}) are nonnegative, thus we may drop the second one. Then, using the properties of
$\varphi$ and Young's inequality for convolution we obtain
\begin{equation} \label{sup10}
\Big(\int_{t_2-t_1}^{t_0-t_1} (\int_{B_{r\sigma}}
[\psi(x)w(s,x)]^2\,dx)^p\,ds\Big)^{1/p} \,\le
\|l\|_{L_p([0,t_0-t_1])} \int_0^{t_0-t_1} \!\!\!\!F(s)\,ds \hd \m{  for all } \hd p\in(1,\frac{1}{1-\gmb}).
\end{equation}
We choose any of these $p$ and fix it.

On the other hand, we may also drop the first term in (\ref{sup9}) and convolve
the resulting inequality with $k$. Evaluating it at
$s=t_0-t_1$, we arrive at
\begin{equation} \label{sup12}
\int_{t_2-t_1}^{t_0-t_1}\int_{B_{r\sigma}}\psi^2|Dw|^2\,dx\,ds \le
\,\frac{|1+\beta|}{2\nu |\beta|}\,\int_0^{t_0-t_1} \!\!\!\!F(s)\,ds.
\end{equation}
Let us estimate $\int_0^{t_0-t_1}F(s)ds$. Firstly, we get
\eqq{
\int_{0}^{t_{0}-t_{1}}\int_{B_{r\sigma}}\abs{D \psi}^{2}w^{2}dxds \leq \frac{4}{r^{2}\sigma^{2}(\rho-\rho')^{2}}\int_{0}^{t_{0}-t_{1}}\int_{B_{\rho r\sigma}}w^{2}dxds.
}{pierw1}
Next, we note that
\[
\vf(s)k(s) \leq k\left(\frac{t_{2}-t_{1}}{2}\right) =k\left(\jd (\rho-\rho') \sigma \eta \Phi(2r)  \right) = \int_{0}^{1}\frac{1}{\Gamma(1-\al)} 2^{\al} (\rho-\rho')^{-\al} \sigma^{-\al} \eta^{-\al}  \Phi(2r)^{-\al}\dd
\]
\[
\leq  2(\rho-\rho')^{-1} \delta^{-1} \max\{\eta^{-1},1 \}k(\Phi(2r)),
\]
and thus we obtain
\eqq{
\vf(s)k(s)\int_{B_{r\sigma}}\psi^{2}w^{2}dx \leq c(\eta,\delta)(\rho-\rho')^{-1}k(\vdr) \int_{B_{\rho r\sigma}}w^{2}dx.
}{fa}
Further,
\[
\dot{\vf}(s)(k*\int_{B_{r\sigma}}\psi^{2}w^{2}dx)(s)\leq \frac{4}{\sigma\eta(\rho-\rho')\vdr}(k*\int_{B_{\rho \sigma r}}w^{2}dx)(s)
\]
and consequently, we get
\[
\int_{0}^{t_{0}-t_{1}}\dot{\vf}\cdot k*\int_{B_{\rho \sigma r}}w^{2}dxd\tau \leq  \frac{4}{\sigma\eta(\rho-\rho')\vdr} \izj \frac{1}{\Gamma(2-\al)}\int_{0}^{t_{0}-t_{1}}(t_{0}-t_{1}-\tau)^{1-\al}\int_{B_{\rho \sigma r}}w^{2}dxd\tau \dd
\]
\[
 \leq \frac{4}{\sigma\eta(\rho-\rho')\vdr} \izj \frac{1}{\Gamma(2-\al)}(t_{0}-t_{1})^{1-\al}\dd\int_{0}^{t_{0}-t_{1}}\int_{B_{\rho \sigma r}}w^{2}dxd\tau
\]
\[
=\frac{4}{\sigma\eta(\rho-\rho')} \izj \frac{1}{\Gamma(2-\al)}(\rho \sigma \eta)^{1-\al} \Phi(2r)^{-\al}\dd\int_{0}^{t_{0}-t_{1}}\int_{B_{\rho \sigma r}}w^{2}dxd\tau
\]

\[
 \leq c(\mu,\eta) \frac{1}{\sigma(\rho-\rho')} \ki(\vdr)
 \int_{0}^{t_{0}-t_{1}}\int_{B_{\rho \sigma r}}w^{2}dxd\tau.
\]
Recall that $\ki(\Phi(2r))=r^{-2}/4$ and $k(\Phi(2r)) \leq \ki(\Phi(2r)) $, hence from (\ref{defeF}),  (\ref{pierw1}), (\ref{fa}) and the last estimate we get
\eqq{
\int_{0}^{t_{0}-t_{1}}|F(s)|ds  \leq c(\mu,\eta,\delta,\Lambda,\nu) \frac{1+\abs{1+\beta}}{(\rho-\rho')^{2} r^{2}}\int_{0}^{t_{0}-t_{1}}\int_{B_{\rho \sigma r}}w^{2}dxd\tau.
}{estiF1}
\nic{
Let us now estimate $\|l\|_{L_{p}}$. To this end we denote
\[
h(r) = \int_{\gmb}^{\as}r^{\al}\ma d\al.
\]
We note that
\[
h(r)\leq c(\mu) \int_{\gmb}^{\as} r^{\al}\sin(\pi \al)\ma d\al \leq c(\mu) \int_{0}^{1} r^{\al}\sin(\pi \al)\ma d\al
\]
and from (\ref{calka}) and (\ref{fin2}) we obtain
\[
l(t) \leq \frac{1}{\pi}\int_{0}^{\infty}e^{-tr}\frac{1}{\izj r^{\al}\sin(\pi \al)\ma d\al}dr \leq  c(\mu)\int_{0}^{\infty}\frac{e^{-tr}}{h(r)}dr=c(\mu)\izi \frac{e^{-w}}{h(\frac{w}{t})}\frac{dw}{t}.
\]
If we apply the estimates
\[
h\left(\frac{w}{t}\right)\geq w^{\as}h\left(\frac{1}{t}\right) \m{ \hd for \hd }  w\leq 1 \m{\hd and \hd } h\left(\frac{w}{t}\right)\geq w^{\gmb}h\left(\frac{1}{t}\right) \m{ \hd for \hd }  w> 1,
\]
then we obtain
\eqq{
l(t) \leq \frac{c(\mu)}{th(\frac{1}{t})} \left(  \izj w^{-\as} e^{-w}dw +\int_{1}^{\infty}w^{-\gmb} e^{-w}dw  \right)\leq  \frac{c(\mu)}{\int_{\gmb}^{\as}t^{1-\al}\ma d\al}.
}{ll1}
Thus,
\eqq{
\|l\|_{L_{p}(0,t)} \leq c(\mu) \left(\izt \left(\int_{\gmb}^{\as}\tau^{1-\al}\ma d\al\right)^{-p}d\tau\right)^{\frac{1}{p}} \m{ for } 1<p<\frac{1}{1-\gmb}.
}{lnormest}
}
From (\ref{sup10}) and  (\ref{estiF1}) we obtain
\[
\norm{\psi w}_{L_{2p}((t_{2}-t_{1},t_{0}-t_{1});L_{2}(B_{r\sigma}))} \leq \norm{l}_{L_{p}((0,t_{0}-t_{1}))}^{\frac{1}{2}}\left(\int_{0}^{t_{0}-t_{1}}F(s)ds\right)^{\frac{1}{2}}
\]
\[
\leq C(\mu,\eta,\delta,\Lambda, \nu)\|l\|_{L_{p}((0,\rho\sigma\eta\Phi(2r)))}^{\frac{1}{2}}\frac{1+\abs{1+\beta}}{\rho-\rho'}\frac{1}{r}\left(\int_{0}^{t_{0}-t_{1}}\int_{B_{\rho \sigma r}}w^{2}dxd\tau\right)^{\frac{1}{2}}.
\]
We note that since $\rho\sigma \le 1$ we have
\eqq{
\norm{l}^{p}_{L_{p}((0,\rho\sigma\eta\Phi(2r)))} \leq \norm{l}^p_{L_{p}((0,\eta\Phi(2r)))} = \frac{1}{\pi^p}\int_{0}^{\eta\vdr} \left(\izi e^{-r\tau}H(r)dr\right)^{p}d\tau \leq \max\{1,\eta\} \norm{l}^p_{L_{p}((0,\Phi(2r)))},
}{estiPhieta}
where  we used the representation (\ref{calka}) and the substitution $s:=\tau/\eta$ in the case $\eta > 1$.
Therefore, we have
\eqnsl{ & \norm{\psi w}_{L_{2p}((t_{2}-t_{1},t_{0}-t_{1});L_{2}(B_{r\sigma}))}   \\
& \leq C(\mu, \eta,\delta,\Lambda, \nu,p)  \norm{l}_{L_{p}((0,\Phi(2r)))}^{\frac{1}{2}}\frac{1+\abs{1+\beta}}{\rho-\rho'}\frac{1}{r}\left(\int_{0}^{t_{0}-t_{1}}\int_{B_{\rho \sigma r}}w^{2}dxd\tau\right)^{\frac{1}{2}}.}{estiwwl2p}
Furthermore, from (\ref{sup12}), (\ref{pierw1}), (\ref{estiF1})  and
\[
\int_{t_2-t_1}^{t_0-t_1}\int_{B_{r\sigma}} |D(\psi w)|^2\,dx\,ds\le
2\int_{t_2-t_1}^{t_0-t_1}\int_{B_{r\sigma}}
\big(\psi^2|Dw|^2+|D\psi|^2w^2\big)\,dx\,ds
\]
we infer that
\eqq{
\norm{D(\psi w)}_{L_{2}((t_{2}-t_{1},t_{0}-t_{1});L_{2}(B_{r\sigma}))} \le
C(\mu,p,\eta,\Lambda, \delta,\nu)\frac{1+\abs{1+\beta}}{\rho-\rho'}\frac{1}{r}\left(\int_{0}^{t_{0}-t_{1}}\int_{B_{\rho \sigma r}}w^{2}dxd\tau\right)^{\frac{1}{2}}.
}{estiDpsiw}
Applying the interpolation inequality (\ref{interapara}) we arrive at
\eqns{
\norm{\psi w}_{L_{2\kappa}((t_{2}-t_{1},t_{0}-t_{1});L_{2\kappa}(B_{r\sigma}))}  & \\
\leq C(N)  & \norm{D(\psi w)}^{\theta}_{L_{2}((t_{2}-t_{1},t_{0}-t_{1});L_{2}(B_{r\sigma}))}  \norm{\psi w}_{L_{2p}((t_{2}-t_{1},t_{0}-t_{1});L_{2}(B_{r\sigma}))}^{1-\theta}
}
where $\kappa, \theta$ are given by (\ref{defkappatheta}).
Hence, if we apply in the above inequality the estimates (\ref{estiwwl2p})  and (\ref{estiDpsiw}), then we obtain
\eqq{
\norm{\psi w}_{L_{2\kappa}((t_{2}-t_{1},t_{0}-t_{1});
L_{2\kappa}(B_{r\sigma}))}
 \leq C(\mu,p,\eta,\Lambda, \delta,\nu,N)\frac{1}{K(r)} \frac{1+\abs{1+\beta}}{\rho-\rho'}\left(\int_{0}^{t_{0}-t_{1}}\int_{B_{\rho \sigma r}}w^{2}dxd\tau\right)^{\frac{1}{2}},
}{mainMoser1}
where
\eqq{
K(r):= r^{\theta}r^{(1-\theta)}\norm{l}_{L_{p}((0,\Phi(2r)))}^{\frac{\theta-1}{2}}.
}{defKr}
We denote  $\gamma=|1+\beta|$ and we may write
\[
\norm{\psi w}_{L_{2\kappa}((t_{2}-t_{1},t_{0}-t_{1});
L_{2\kappa}(B_{r\sigma}))} \geq \left( \int_{t_{2}-t_{1}}^{t_{0}-t_{1}} \int_{B\rho' r\sigma} w^{2\kappa} dxds \right)^{\frac{1}{2\kappa}}
\]
\[
=\left(  \int\limits_{(\rho-\rho')\sigma \eta \Phi(2r)}^{\rho \sigma \eta \Phi(2r)}  \int_{B\rho' r\sigma} [u^{-1}(s+t_{0}-\rho\sigma \eta \Phi(2r),x)]^{\kappa \gamma}dx ds \right)^{\frac{1}{2\kappa}}=\| u^{-1}\|_{L_{\kappa \gamma }(V_{\rho'})}^{\frac{\gamma}{2}}
\]
and
\[
\left( \int_{0}^{t_{0}-t_{1}} \int_{B_{\rho\sigma r}} w^{2}dxds \right)^{\frac{1}{2}} = \left( \int_{0}^{\rho \sigma \eta \Phi(2r)}
\int_{B_{\rho \sigma r}} [u^{-1}(s+t_{0}-\rho\sigma \eta \Phi(2r),x) ]^{\gamma}  dxds \right)^{\frac{1}{2}} = \| u^{-1} \|_{L_{\gamma}(V_{\rho})}^{\frac{\gamma}{2}}.
\]
Therefore, (\ref{mainMoser1}) leads to the estimate
\[
\|u^{-1}\|_{L_{\gamma\kappa}(V_{\rho'})} \leq \left(\frac{C^{2}(1+\gamma)^{2}}{(\rho-\rho')^{2}(K(r))^2}\right)^{\frac{1}{\gamma}} \|u^{-1}\|_{L_{\gamma}(V_{\rho})},\quad 0<\rho'<\rho\le 1.
\]
We note that $\kappa>1$ and we apply the first  Moser iteration lemma (Lemma~\ref{moserit1}) to get
\[
\esssup_{V_{\vsi}}u^{-1}\leq \left(\frac{M_{0}}{(1-\vsi)^{\frac{2\kappa}{\kappa-1}}(K(r))^{2\frac{\kappa}{\kappa-1}}}\right)^{\frac{1}{\gamma}}\|u^{-1}\|_{L_{\gamma}(V_{1})}, \hd \gamma\in (0,1], \hd \vsi \in (0,1),
\]
where $M_{0}=M_{0}( \mu, p ,\eta, \Lambda, \delta, \nu ,N)$.
We note that from (\ref{defkappatheta}) we get
\[
\theta\frac{2\kappa}{\kappa-1} = N, \hd \hd (1-\theta)\frac{2\kappa}{\kappa-1} = \frac{2p}{p-1}.
\]
Hence, using  the definition of $K(r)$ (\ref{defKr}) we obtain
\eqq{
(K(r))^{2\frac{\kappa}{\kappa-1}} = \norm{l}_{L_{p}((0,\Phi(2r)))}^{\frac{-p}{p-1}}r^{N}r^{\frac{2p}{p-1}}.
}{Krdopotegi}
Applying estimate (\ref{zn2}) we obtain for $r \leq r^{*}(p,\mu)$, where $r^{*}$ comes from Lemma (\ref{scaling}),
\[
\esssup_{V_{\vsi}}u^{-1}\leq \left(\frac{M_{0}}{(1-\vsi)^{\frac{2\kappa}{\kappa-1}}r^{N}\vdr}\right)^{\frac{1}{\gamma}}\|u^{-1}\|_{L_{\gamma}(V_{1})}, \hd \gamma\in (0,1], \hd \vsi \in (0,1),
\]
where $M_{0}=M_{0}( \mu, p ,\eta, \Lambda, \delta, \nu ,N)$. Thus, if we take
$\vsi=\sigma'/\sigma$ and notice that $V_{\vsi}=U_{\sigma'}$, $V_{1}=U_{\sigma}$, $\nuNj(U_{1})=\eta\Phi(2r)r^{N}$ and
\[ \frac{1}{1-\vsi}\,=\,\frac{\sigma}{\sigma-\sigma'}\,\le
\frac{1}{\sigma-\sigma'},\]
then we obtain for $r \leq r^{*}(p,\mu)$
\[
\esup_{U_{\sigma'}}{u^{-1}} \le
\Big(\frac{M_0 \nuNj(U_{1})^{-1}}{(\sigma-\sigma')^{\tau_0}}\Big)^{1/\gamma}
\|u^{-1}\|_{L_{\gamma}(U_\sigma)},\quad \gamma\in (0,1],
\]
where $M_{0}=M_{0}( \mu, p ,\eta, \Lambda, \delta, \nu ,N)$ and $p\in (1,\frac{1}{1-\gmb})$ had been fixed. Hence the proof is complete. $\square$

${}$

For a fixed $\gmb\in(0,1)$ satisfying (\ref{intmugk}) we put
\eqq{
\tilde{\kappa}:= \frac{2+N\gmb}{2+N\gmb - 2\gmb}
}{defkappaf}
\begin{satz} \label{superest2}
Let $\Omega\subset \iR^N$ be a bounded
domain. Suppose the assumptions (H1)--(H3) are satisfied. Let
$\eta>0$ and $\delta\in (0,1)$ be fixed. Let $\gmb\in (0,1)$ be such that (\ref{intmugk}) is fulfilled. Then, for any $p_0\in(0,\tilde{\kappa})$, there exists $r^{*} = r^{*}(\mu,p_{0})$ such that for any $t_0\in [0,T)$
and $r\in (0, r^{*}]$ with $t_0+\eta \vdr \le T$, any ball
$B_r(x_0)\subset\Omega$ and any
nonnegative weak supersolution $u$ of (\ref{MProb}) in $(0,t_0+\eta
\vdr )\times B_r(x_0)$ with $u_0\ge 0$ in $B_r(x_0)$ and $f\equiv 0 $,  there holds
\eqq{
\|u\|_{L_{p_0}(U_{\sigma'}')}\le \left( \frac{C\cdot \nuNj(U'_1)^{-1}
}{(\sigma-\sigma')^{\gamma_0}}\right)^{1/\gamma-1/p_0}
\|u\|_{L_{\gamma}(U'_\sigma)},\quad \delta\le \sigma'<\sigma\le 1,\;
0<\gamma\le p_0/\tilde{\kappa}.
}{fin3}
Here $U'_\sigma=(t_0,t_0+\sigma\eta \vdr)\times B_{\sigma r}(x_0)$,
$C=C(\nu,\Lambda,\delta,\eta,\mu,N,p_0)$ and
$\gamma_0=\gamma_0(\mu,p_{0},N)$.
\end{satz}
\begin{proof}
It is enough to prove (\ref{fin3}) only for $p_0>1$, because otherwise we first apply H\"older's inequality with exponents $(q,\frac{q}{q-1})$ where $qp_{0}>1$ and next apply (\ref{fin3}) with $p_{0}>1$. We proceed similarly as in the proof of Theorem \ref{superest1}. Here again, we follow the approach from \cite{base} and then we use Lemma \ref{scaling}.

Without loss of generality, we may may further assume that $u$ is bounded away from zero. Otherwise, we replace $u$ with $u+\varepsilon$ and $u_0$ with
$u_0+\varepsilon$ and eventually pass with $\varepsilon$ to zero.
To abbreviate the notation, we again denote $B_{r}:=B_{r}(x_{0}).$

Fix $\sigma'$, $\sigma$ such that $\delta\le \sigma'<\sigma\le 1$. For $\rho\in (0,1]$ we set
$V'_\rho=U'_{\rho\sigma}$. Given $0<\rho'<\rho\le 1$ and $r>0$, let
$t_1=t_0+\rho'\sigma\eta\vdr$ and $t_2=t_0+\rho\sigma\eta\vdr$. We notice that then $0\le
t_0<t_1<t_2$. We shift the time putting ${s}=t-t_0$ and
$\tilde{f}(s)=f(s+t_0)$, $s\in (0,t_2-t_0)$, for functions $f$
defined on $(t_0,t_2)$.
We begin similarly as in the proof of Theorem  \ref{superest1}, the only
difference is that now we take $\beta\in (-1,0)$. In this case, (\ref{sup1}) simplifies to
\[
-\tilde{u}^{\beta}\partial_{s}(k_{n}\ast \tilde{u}) \ge
-\,\frac{1}{1+\beta}\,\partial_{s}
(k_{n}\ast\tilde{u}^{1+\beta}),\quad
\mbox{a.a.}\;(s,x)\in (0,t_2-t_0)\times B_r,
\]
hence taking $\psi\in C_0^1(B_{r\sigma})$ as above, we arrive at the following estimate
\begin{align}
-\,\frac{1}{1+\beta}\,& \int_{B_{r\sigma}}\psi^2\partial_{s}
(k_{n}\ast\tilde{u}^{1+\beta})\,dx+|\beta|\int_{B_{r\sigma}}\big((h_n\ast
[ADu])\,\tilde{}\;|\psi^2 \tilde{u}^{\beta-1}D
\tilde{u}\big)\,dx \nonumber\\
\le  & \,2\int_{B_{r\sigma}}\big((h_n\ast [ADu])\,\tilde{}\;|\psi D\psi
\,\tilde{u}^{\beta}\big)\,dx,\quad\quad
\mbox{a.a.}\;s\in(0,t_2-t_0). \label{L1}
\end{align}

Next, choose $\varphi\in C^1([0,t_2-t_0])$ such that $0\le
\varphi\le 1$, $\varphi=1$ in $[0,t_1-t_0]$, $\varphi=0$ in
$[t_1-t_0+(t_2-t_1)/2,t_2-t_0]$, and $0\le -\dot{\varphi}\le
4/(t_2-t_1)$. We multiply (\ref{L1}) by $(1+\beta)\varphi(s)$ (recall that $1+\beta > 0$) and apply Lemma \ref{comm2} to the first term
to get
\begin{align}
-\int_{B_{r\sigma}}  &
\partial_{s}(k_{n}\ast
[\varphi\psi^2\tilde{u}^{1+\beta}]\big)\,dx+|\beta|(1+\beta)\,
\int_{B_{r\sigma}}\big(\tilde{A}D\tilde{u}|\psi^2 \tilde{u}^{\beta-1}D
\tilde{u}\big)\varphi\,dx \nonumber\\
\le & \,\int_0^s
\dot{k}_{n}(s-\tau)\big(\varphi(s)-\varphi(\tau)\big)
\big(\int_{B_{r\sigma}}\psi^2\tilde{u}^{1+\beta}\,dx\big)(\tau)\,d\tau\nonumber\\
& \;\,+2(1+\beta)\,\int_{B_{r\sigma}}\big(\tilde{A}D\tilde{u}|\psi D\psi
\,\tilde{u}^{\beta}\big)\varphi\,dx+\mathcal{R}_n(s) ,\quad
\mbox{a.a.}\;s\in(0,t_2-t_0), \label{L2}
\end{align}
where
\begin{align*}
\mathcal{R}_n(s)= &\,\,-|\beta|(1+\beta)\, \int_{B_{r\sigma}}\big((h_n\ast
[ADu])\,\tilde{}\;-\tilde{A}D\tilde{u}|\psi^2 \tilde{u}^{\beta-1}D
\tilde{u}\big)\varphi\,dx\\
&\,+2(1+\beta)\,\int_{B_{r\sigma}}\big((h_n\ast
[ADu])\,\tilde{}\;-\tilde{A}D\tilde{u}|\psi D\psi
\,\tilde{u}^{\beta}\big)\varphi\,dx,\quad
\mbox{a.a.}\;s\in(0,t_2-t_0).
\end{align*}
We set again $w=\tilde{u}^{\frac{\beta+1}{2}}$ and estimate
exactly as in the proof of Theorem \ref{superest1}, using (H1), (H3) and
(\ref{sup8}), to the result
\begin{align}
-\int_{B_{r\sigma}}  &
\partial_{s}(k_{n}\ast
[\varphi\psi^2w^2]\big)\,dx+\,\frac{2\nu
|\beta|}{1+\beta}\,\int_{B_{r\sigma}}\varphi
\psi^2|Dw|^2\,dx \nonumber\\
\le & \,\int_0^s
\dot{k}_{n}(s-\tau)\big(\varphi(s)-\varphi(\tau)\big)
\big(\int_{B_{r\sigma}}\psi^2w^2\,dx\big)(\tau)\,d\tau\nonumber\\
& \;\,+\,\frac{2\Lambda^2(1+\beta)}{\nu |\beta|}\, \int_{B_{r\sigma}}
|D\psi|^2\varphi w^2\,dx+\mathcal{R}_n(s) ,\quad
\mbox{a.a.}\;s\in(0,t_2-t_0). \label{L3}
\end{align}
Recall that $k_{n}=k\ast h_n$. We denote  the right-hand side of (\ref{L3}) by $F_n(s)$ and set
\[
W(s)=\int_{B_{r\sigma}}\varphi(s)\psi(x)^2w(s,x)^2\,dx.
\]
Then it
follows from (\ref{L3}) that
\[
G_n(s):=\partial_s [k* (h_n\ast W)](s)+F_n(s)\ge 0,\quad\quad
\mbox{a.a.}\;s\in(0,t_2-t_0).
\]
By complete positivity of $l$, $h_n$ is nonnegative for every $n\in \iN$ (\cite{CN}). Hence, applying (\ref{sup5}), we have
\[
0\le h_n\ast W =l\ast \partial_s[ k* (h_n\ast W)]\le
l\ast G_n+l\ast [-F_n(s)]_+
\]
a.e. in $(0,t_2-t_0)$, where $[y]_+$ stands for the positive part
of $y\in \iR$. For any $p\in (1,1/(1-\gmb))$ and any
$t_*\in[t_2-t_0-(t_2-t_1)/4,t_2-t_0]$ using Young's
inequality we obtain
\begin{equation} \label{L4}
\|h_n\ast W\|_{L_p((0,t_*))}\le
\|l\|_{L_p((0,t_*))}\big(\|G_n\|_{L_1((0,t_*))}+
\|[-F_n]_+\|_{L_1((0,t_*))}\big).
\end{equation}
Since $G_n$ is positive we have
\eqq{
\|G_n\|_{L_1([0,t_*])}=(k_{n}\ast
W)(t_*)+\int_0^{t_*}\!\!\!F_n(s)\,ds.
}{Gnsuma}
Observe that $\mathcal{R}_n\rightarrow 0$ in $L_1((0,t_2-t_0))$ as
$n\to \infty$. The functions $k_{n}$ and $\varphi$ are nonincreasing,  hence
$[-F_n]_+ \leq [-\mathcal{R}_n]_{+}$ and   $  \|[-F_n]_+\|_{L_1((0,t_*))}\to 0$ as
$n\to\infty$. Further, since $\varphi$ is nonincreasing there holds

\eqnsl{
\int_0^{t_*}\!\!&\!\int_0^s
\dot{k}_{n}(s-\tau)\big(\varphi(s)-\varphi(\tau)\big)
\big(\int_{B_{r\sigma}}\psi^2w^2\,dx\big)(\tau)\,d\tau\,ds\\
=&\,\int_0^{t_*}
k_{n}(t_*-\tau)\big(\varphi(t_*)-\varphi(\tau)\big)
\big(\int_{B_{r\sigma}}\psi^2w^2\,dx\big)(\tau)\,d\tau\\
&\,-\int_0^{t_*}\!\!\!\dot{\varphi}(s)\int_0^s
k_{n}(s-\tau)
\big(\int_{B_{r\sigma}}\psi^2w^2\,dx\big)(\tau)\,d\tau\,ds\\
\le&\,-\int_0^{t_*}\!\!\!\dot{\varphi}(s)\int_0^s
k_{n}(s-\tau)
\big(\int_{B_{r\sigma}}\psi^2w^2\,dx\big)(\tau)\,d\tau\,ds.
}{kndot}
We also know that
$k_{n}\ast W\to k\ast W$ in
$L_1((0,t_2-t_0))$. Thus, we can fix some
$t_*\in[t_2-t_0-(t_2-t_1)/4,t_2-t_0]$ such that for a
subsequence $(k_{n_m}\ast W)(t_*)\to (k\ast
W)(t_*)$ as $m\to \infty$. From (\ref{Gnsuma}) for  such $t_*$ we have
\[
\limsup_{m\rightarrow \infty} \|G_{n_{m}}\|_{L_{1}((0,t_*))}\leq (k*W)(t_*)+\limsup_{m\rightarrow \infty} \int_0^{t_*}\!\!\!F_{n_{m}}(s)\,ds,
\]
and using the previous estimate we get
\[
\limsup_{m\rightarrow \infty} \int_0^{t_*}\!\!\!F_{n_{m}}(s)\,ds \leq
\limsup_{m\rightarrow \infty}  \int_0^{t_*}\!\!\!- \dot{\varphi}(s)\int_0^s
k_{n_{m}}(s-\tau)
\big(\int_{B_{r\sigma}}\psi^2w^2\,dx\big)(\tau)\,d\tau\,ds
\]
\[
+\frac{2\Lambda^2(1+\beta)}{\nu |\beta|}\, \int_{0}^{t_*}\int_{B_{r\sigma}}
|D\psi|^2\varphi w^2\,dx ds +\limsup_{m\rightarrow \infty} \int_{0}^{t_*} |\mathcal{R}_{n_{m}}(s)|ds
\]
\[
=\int_0^{t_*}\!\!\!- \dot{\varphi}(s)\int_0^s
k(s-\tau)
\big(\int_{B_{r\sigma}}\psi^2w^2\,dx\big)(\tau)\,d\tau\,ds
+\frac{2\Lambda^2(1+\beta)}{\nu |\beta|}\, \int_{0}^{t_*}\int_{B_{r\sigma}}
|D\psi|^2\varphi w^2\,dx ds.
\]
Recall that  $  \|[-F_n]_+\|_{L_1((0,t_*))}\to 0$ hence, from the above estimates and (\ref{L4}) we obtain
\[
\|W\|_{L_{p}((0,t_*))} \leq \|l\|_{L_{p}((0,t_*))} \Big((k\ast
W)(t_*)+\|F\|_{L_1((0,t_*))}\Big),
\]
with
\eqq{
F(s)=\,\frac{2\Lambda^2(1+\beta)}{\nu |\beta|}\, \int_{B_{r\sigma}}
|D \psi|^2\varphi w^2\,dx-\dot{\varphi}(s)\big(k\ast
\int_{B_{r\sigma}}\psi^2w^2\,dx\big)(s).
}{Fsdef}
Recall that  $\varphi=1$ in
$[0,t_1-t_0]$, so we have
\begin{equation} \label{L5}
\Big(\int_{0}^{t_1-t_0} (\int_{B_{r\sigma}}
[\psi(x)w(s,x)]^2\,dx)^p\,ds\Big)^{1/p}\le
\|l\|_{L_{p}((0,t_{2}-t_{0}))}\Big((k\ast
W)(t_*)+\|F\|_{L_1((0,t_2-t_0))}\Big).
\end{equation}

On the other hand, we can integrate (\ref{L3}) over $(0,t_*)$, apply (\ref{kndot}) and then
take the limit as $m\to \infty$ for the same subsequence as
before, thereby obtaining
\begin{equation} \label{L6}
\int_{0}^{t_1-t_0}\!\!\!\int_{B_{r\sigma}}
\psi^2|Dw|^2\,dx\,ds\le\,\frac{1+\beta}{2\nu
|\beta|}\,\Big((k\ast
W)(t_*)+\|F\|_{L_1((0,t_2-t_0))}\Big).
\end{equation}
We note that $\varphi=0$ on $[\frac{1}{2}(\rho+\rho')\sigma \eta \Phi(2r), \rho\sigma \eta \Phi(2r)]$ and $t_*\in [(\frac{3}{4}\rho+\frac{1}{4}\rho')\sigma \eta \Phi(2r), \rho\sigma \eta \Phi(2r)]$, thus we have
\eqnsl{
(k*W)(t_{*})= & \int_{0}^{\frac{1}{2}(\rho+\rho')\sigma \eta \Phi(2r)} k(t_*-\tau) \varphi(\tau ) \left( \int_{B_{r\sigma}}
\psi^2 w^2\,dx\right) (\tau ) d\tau \\
& \leq k\left(t_* -\frac{1}{2}(\rho+\rho')\sigma \eta \Phi(2r)\right) \int_{0}^{t_{2}-t_{0}} \int_{B_{\rho r\sigma} } w^{2} dx d\tau \\
& \leq k\left( \frac{1}{4}(\rho-\rho')\sigma \eta \Phi(2r)\right) \int_{0}^{t_{2}-t_{0}} \int_{B_{\rho r\sigma} } w^{2} dx d\tau \\
& \leq c(\eta,\delta)\frac{k(\vdr)}{(\rho-\rho')}\int_{0}^{t_{2}-t_{0}}\int_{B_{\rho \sigma r}}w^{2}dxd\tau.
}{estikW}
We also have
\eqq{
\int_{0}^{t_{2}-t_{0}}\int_{B_{r\sigma}}\abs{D \psi}w^{2}dxds \leq \frac{4}{r^{2}\sigma^{2}(\rho-\rho')^{2}}\int_{0}^{t_{2}-t_{0}}\int_{B_{\rho r\sigma}}w^{2}dxds.
}{estinabpsi}
Now, we shall estimate the $L_1$-norm of (\ref{Fsdef})
\[
\int_{0}^{t_{2}-t_{0}} - \dot{\varphi}(s) \left( k* \int_{B_{r\sigma} } \psi^{2} w^{2}dx   \right)(s)ds \leq \frac{4}{t_{2}-t_{1}} \int_{0}^{t_{2}-t_{0}}   k* \int_{B_{r\sigma} } \psi^{2} w^{2}dx   ds
\]
\[
= \frac{4}{(\rho - \rho ') \sigma \eta \Phi(2r)} \int_{0}^{1}\int_{0}^{t_{2}-t_{0}} \frac{(t_2-t_0-s)^{1-\al} }{\Gamma(2-\al)}    \int_{B_{ r\sigma} } \psi^2(x) w^{2}(x,s)dx ds \dd
\]
\[
\leq \frac{4}{(\rho - \rho ') \sigma \eta \Phi(2r)} \int_{0}^{1} \frac{(t_2-t_0)^{1-\al}}{\Gamma(2-\al)} \dd  \int_{0}^{t_{2}-t_{0}}  \int_{B_{ r\sigma} } \psi^2(x) w^{2}(x,s)dx ds
\]
\[
\leq  \frac{4}{(\rho - \rho ') \sigma \eta } c(\eta) k_1(\vdr) \int_{0}^{t_{2}-t_{0}}  \int_{B_{\rho r\sigma} } w^{2}(x,s)dx ds.
\]
Using (\ref{zn1}) we find that
\eqq{\int_{0}^{t_{2}-t_{0}} - \dot{\varphi}(s) \left( k* \int_{B_{r\sigma} } \psi^{2} w^{2}dx   \right)(s)ds \leq
\frac{c( \eta, \delta)}{(\rho - \rho ')^{2} r^{2}} \int_{0}^{t_{2}-t_{0}}    \int_{B_{\rho r\sigma} }  w^{2}dx   ds.}{estivphidot}
Hence, from (\ref{Fsdef}), (\ref{estinabpsi}) and (\ref{estivphidot}) we obtain
\eqq{
\norm{F}_{L_{1}((0,t_{2}-t_{0}))} \leq c(\mu,\delta,\eta,\Lambda, \nu )\frac{1+(1+\beta)}{|\beta|(\rho-\rho')^{2}r^{2}}\int_{0}^{t_{2}-t_{0}}\int_{B_{\rho r\sigma}}w^{2}dxds.
}{estiFlj}
 Using (\ref{L5}), (\ref{estikW}), (\ref{zn1}) and (\ref{estiFlj}) we get
\eqnsl{ \| \psi w \|_{L_{2p}((0,t_{1}- t_{0});L_{2}(B_{r\sigma })  )}&  \\
\leq c(\mu,\delta,\eta,\Lambda, \nu )&  \frac{1+(1+\beta)}{|\beta|(\rho-\rho')r} \| l \|_{L_{p}((0,t_{2}-t_{0}))}^{\frac{1}{2}} \left( \int_{0}^{t_{2}-t_{0}}\int_{B_{\rho r\sigma}}w^{2}dxds\right)^{\frac{1}{2}}. }{estildpphiw}
Recall that $t_{2}-t_{0}= \rho \sigma \eta \Phi(2r)$, and hence employing  (\ref{estiPhieta}) we see that
\[
\| l\|_{L_{p}((0,t_{2}-t_{0}))} \leq c(\eta  )\| l\|_{L_{p}((0,\vdr))},
\]
and (\ref{estildpphiw}) gives
\eqnsl{ \| \psi w \|_{L_{2p}((0,t_{1}- t_{0});L_{2}(B_{r\sigma })  )} &  \\
\leq c(\mu,\delta,\eta,\Lambda, \nu ) \frac{1+(1+\beta)}{|\beta|(\rho-\rho')r} &  \| l\|_{L_{p}((0,\vdr))}^{\frac{1}{2}}   \left( \int_{0}^{t_{2}-t_{0}}\int_{B_{\rho r\sigma}}w^{2}dxds\right)^{\frac{1}{2}}. }{estildpphiw2}
Next, from (\ref{L6}), (\ref{estikW}), (\ref{zn1}) and (\ref{estiFlj}) we obtain
\[
\int_{0}^{t_{1}-t_{0}}\int_{B_{r\sigma }} \psi^{2} |Dw|^{2} dx ds \leq  c(\mu, \delta, \eta, \Lambda,\nu) \frac{1+\beta}{2\nu |\beta|} \frac{1+(1+\beta)}{|\beta|(\rho - \rho ')^{2} r^{2}} \int_{0}^{t_{2}-t_{0}}\int_{B_{\rho r\sigma }}  w^{2} dx ds.
\]
Using the above estimate and (\ref{estinabpsi}) yields
\eqq{\| D(\psi w)\|_{L_{2}((0,t_{1}-t_{0});L_{2}(B_{r \sigma }))} \leq  c(\mu, \delta, \eta, \Lambda, \nu )  \frac{1+(1+\beta)}{|\beta|(\rho - \rho ') r} \left( \int_{0}^{t_{2}-t_{0}}\int_{B_{\rho r\sigma }}  w^{2} dx ds\right)^{\frac{1}{2}}.
}{estinabpsiw}
Having (\ref{estildpphiw2}) and (\ref{estinabpsiw}) we shall again apply  the interpolation inequality (\ref{interapara}) with $\kappa, \theta$ from (\ref{defkappatheta}) to obtain
\eqns{
\norm{\psi w}_{L_{2\kappa}((0,t_{1}-t_{0});L_{2\kappa}(B_{r\sigma}))}  & \\
\leq C(N)  & \norm{D(\psi w)}^{\theta}_{L_{2}((0,t_{1}-t_{0});L_{2}(B_{r\sigma}))}  \norm{\psi w}_{L_{2p}((0,t_{1}-t_{0});L_{2}(B_{r\sigma}))}^{1-\theta}.
}
Then we arrive at
\eqq{
\norm{\psi w}_{L_{2\kappa}((0,t_{1}-t_{0});
L_{2\kappa}(B_{r\sigma}))}
 \leq C(\mu,\eta,\Lambda, \delta,\nu, N)\frac{1}{K(r)} \frac{1+\abs{1+\beta}}{\rho-\rho'}\left(\int_{0}^{t_{2}-t_{0}}\int_{B_{\rho \sigma r}}w^{2}dxd\tau\right)^{\frac{1}{2}},
}{mainMoser12}
where again
\eqq{
K(r):= r^{\theta}r^{(1-\theta)}\| l\|_{L_{p}((0,\vdr))}^{\frac{\theta-1}{2}}.
}{defKr2}
In particular, we have
\eqq{
\|w\|_{L_{2\kappa}((0,t_{1}-t_{0})\times B_{\rho'r\sigma})} \leq  C(\mu,\eta,\Lambda, \delta,\nu, N)\frac{1+|1+\beta|}{|\beta|(\rho-\rho')K(r)}\|w\|_{L_{2}((0,t_{2}-t_{0})\times B_{\rho r\sigma})}.
}{estiVrho}
If we denote $\gamma:=1+\beta$, then we may write
\[
\|w\|_{L_{2\kappa}((0,t_{1}-t_{0})\times B_{\rho'r\sigma})} = \left( \int_{0}^{\rho' \sigma \eta \Phi(2r)} \int_{B_{\rho' \sigma r}} w^{2\kappa} dx ds \right)^{\frac{1}{2\kappa}}=\left( \int_{0}^{\rho' \sigma \eta \Phi(2r)} \int_{B_{\rho' \sigma r}} u(s+t_{0},x)^{\gamma \kappa } dx ds \right)^{\frac{1}{2\kappa}}
\]
\[
=\left( \int_{t_{0}}^{t_{0}+\rho' \sigma \eta \Phi(2r)} \int_{B_{\rho' \sigma r}} u(t,x)^{\gamma \kappa } dx dt \right)^{\frac{1}{2\kappa}} = \| u \|^{\frac{\gamma}{2}}_{L_{\gamma \kappa }(V_{\rho'}')}.
\]
Similarly we get $\|w\|_{L_{2}((0,t_{2}-t_{0})\times B_{\rho r\sigma})} =  \| u \|^{\frac{\gamma}{2}}_{L_{\gamma  }(V_{\rho}')}$ and (\ref{estiVrho}) gives
\[
\|u\|_{L_{\kappa \gamma}(V'_{\rho'})} \leq \left(\frac{C}{|\beta|^{2}(\rho-\rho')^{2}(K(r))^{2}}\right)^{\frac{1}{\gamma}}\|u\|_{L_{\gamma}(V'_{\rho})},
\]
where $C=C(\mu,\eta,\Lambda, \delta,\nu, N)$. If $p\in (1, \frac{1}{1-\gmb})$, then $\kappa = \frac{2p+N(p-1)}{2+N(p-1)}$ belongs to $(1,\kaf)$, where $\kaf$ was defined in (\ref{defkappaf}), hence we may choose  $p=p(p_{0})\in (1, \frac{1}{1-\gmb})$ such that $\kappa = (p_{0}+\kaf)/2$  (recall the   assumption $p_{0}\in (0, \kaf)$). We have $\frac{p_{0}}{\ka}<1$, so  for any $\gamma = 1+\beta \in (0,\frac{p_{0}}{\ka}]$ we get the estimate
\eqq{
\|u\|_{L_{\kappa \gamma}(V'_{\rho'})} \leq \left(\frac{C}{(\rho-\rho')^{2}(K(r))^{2}}\right)^{\frac{1}{\gamma}}\|u\|_{L_{\gamma}(V'_{\rho})},
}{mainestiVprim}
where $C=C(\mu,\eta,\Lambda, \delta,\nu, N, p_{0})$.  We multiply (\ref{mainestiVprim}) by $(\eta \om_{N}r^{N}\vdr)^{-\frac{1}{\gamma \kappa}}$, where $\om_{N}$ is the measure of the unit ball in $\R^{N}$, and we have
\eqns{
\left(\int_{V_{\rho'}'} |u|^{\kappa \gamma}\frac{1}{\eta \om_{N}r^{N}\vdr}dxdt\right)^{\frac{1}{\gamma \kappa}} &  \\
\leq \left(\frac{C}{(\rho-\rho')^{2}(K(r))^{2}}\right)^{\frac{1}{\gamma}} & \left(\int_{V_{\rho}'} |u|^{ \gamma}\frac{1}{\eta \om_{N}r^{N}\vdr}dxdt\right)^{\frac{1}{\gamma}} \left(\eta \om_{N} r^{N}\vdr\right)^{\frac{(\kappa-1)}{\kappa \gamma}}.
}
Hence, if we denote by $d\muf_{N+1} $ the measure $ (\eta \om_{N} r^{N}\vdr)^{-1} dxdt $, then we get
\[
\|u\|_{L_{\gamma\kappa}(V'_{\rho'},d\muf_{N+1})} \leq \left(\frac{C}{(\rho-\rho')^{2}(K(r))^{2}(\eta \om_{N}r^{N}\vdr)^{-\frac{\kappa-1}{\kappa}}}\right)^{\frac{1}{\gamma}}
\|u\|_{L_{\gamma}(V'_{\rho},d\muf_{N+1})},
\]
for $0<\rho'<\rho\leq 1$ and $\gamma\in (0,\frac{p_{0}}{\kappa}]$. Since $\muf_{N+1}(V_{1}')= \muf_{N+1}(U_{\sigma}')= \sigma^{N+1}\leq 1$ we may apply the second Moser iteration lemma (see Lemma~\ref{moserit2}). This gives
\[
\|u\|_{L_{p_{0}}(V'_{\vs},d\muf_{N+1})} \leq \left(\frac{C^{\frac{\kappa(\kappa+1)}{\kappa-1}} \cdot 2^{\frac{2\kappa^{3}}{(\kappa-1)^{3}}} }{(1-\vs)^{\frac{2\kappa(\kappa+1)}{\kappa-1}}[(K(r))^{2}(\eta \om_{N}r^{N}\vdr)^{-\frac{\kappa-1}{\kappa}}]^{(\frac{\kappa(\kappa+1)}{\kappa-1})}}\right)^{\frac{1}{\gamma}-\frac{1}{p_{0}}}
\|u\|_{L_{\gamma}(V'_{1},d\muf_{N+1})},
\]
where  $ \vs\in (0,1)$ and $ \gamma\in (0,\frac{p_{0}}{\kappa}]$.  Multiplying by $(\eta \om_{N}r^{N}\vdr)^{\frac{1}{p_{0}}} =(\eta \om_{N}r^{N}\vdr)^{\frac{1}{\gamma}}(\eta \om_{N}r^{N}\vdr)^{\frac{1}{p_{0}}-\frac{1}{\gamma}}$ then yields
\[
\|u\|_{L_{p_{0}}(V'_{\vs})} \leq
\left(\frac{C^{\frac{\kappa(\kappa+1)}{\kappa-1}} \cdot 2^{\frac{2\kappa^{3}}{(\kappa-1)^{3}}} }{(1-\vs )^{\frac{2\kappa(\kappa+1)}{\kappa-1}}\eta \om_{N}r^{N}\vdr [(K(r))^{2}(\eta \om_{N}r^{N}\vdr)^{-\frac{\kappa-1}{\kappa}}]^{(\frac{\kappa(\kappa+1)}{\kappa-1})}}\right)^{\frac{1}{\gamma}-\frac{1}{p_{0}}}
\|u\|_{L_{\gamma}(V'_{1})}.
\]
Observe that from (\ref{Krdopotegi}) and (\ref{zn2}) we get for $r \leq r^{*}(\mu,p(p_{0}))$, where $r^{*}$ comes from Lemma \ref{scaling},
\eqns{
\left[(K(r))^{2}\left(\eta \om_{N}r^{N}\vdr\right)^{-\frac{\kappa-1}{\kappa}}\right]^{\frac{\kappa}{\kappa-1}} & =
(K(r))^{\frac{2\kappa}{\kappa-1}}\frac{1}{\eta \om_{N}r^{N}\vdr} \\
& = \frac{r^{\frac{2p}{p-1}}}{\eta \om_{N}\vdr \| l\|_{L_{p}((0,\vdr))}^{\frac{p}{p-1}}}\geq \frac{C(\mu)}{\eta \om_{N}}.
}
Therefore, for $r \leq r^{*}(\mu,p(p_{0}))$, we have
\[
\|u\|_{L_{p_{0}}(V'_{\vs})} \leq
\left(\frac{C }{(1-\vs )^{\gamma_{0}}\eta \om_{N}r^{N}\vdr }\right)^{\frac{1}{\gamma}-\frac{1}{p_{0}}}
\|u\|_{L_{\gamma}(V'_{1})} \hd \m{ for } \hd \vs \in (0,1), \hd \gamma\in (0,\frac{p_{0}}{\kappa}],
\]
where $C=C(\mu,\eta,\Lambda, \delta,\nu, N, p_{0})$ and $\gamma_{0}= \gamma_{0}(p_{0}, \mu, N)$. If we take $\vs=\sigma'/\sigma$, then $V_{\vs}'=U_{\sigma'}'$ and   we obtain  for $r \leq r^{*}(\mu,p(p_{0}))$
\begin{equation*}
\|u\|_{L_{p_0}(U_{\sigma'}')}\le \Big(\frac{C}{(1-\frac{\sigma'}{\sigma})^{\gamma_{0}} \eta \om_{N} r^{N} \vdr }\Big)^{1/\gamma-1/p_0}
\|u\|_{L_{\gamma}(U'_\sigma)}, \hd \delta< \sigma'<\sigma\leq 1,  \quad 0<\gamma\le p_0/\kappa.
\end{equation*}
In particular, if we notice that $\nuNj(U_{1}') = \eta \om_{N} r^{N} \vdr$, then  we have  for $r \leq r^{*}(\mu,p(p_{0}))$
\eqnsl{
\|u\|_{L_{p_0}(U_{\sigma'}')}\le \Big(\frac{C  \nuNj(U_{1}')^{-1}}{(\sigma-\sigma')^{\gamma_{0}}  }\Big)^{1/\gamma-1/p_0}
\|u\|_{L_{\gamma}(U'_\sigma)}, \hd \delta< \sigma'<\sigma\leq 1,  \quad 0<\gamma\le p_0/\kappa.
}{L11}
Since $\kappa<\tilde{\kappa}$ the above estimate holds for
all $\gamma\in (0,p_0/\tilde{\kappa}]$ and the proof is finished.
\end{proof}
\subsection{Logarithmic estimates} \label{logsection}
\begin{satz} \label{logest}
Let  $T>0$ and $\Omega\subset \iR^N$ be a bounded
domain. We assume that (H1)--(H3) are satisfied and fix
$\tau>0$, $\delta,\,\eta\in(0,1)$. Then there exists $r^{*}=r^{*}(\mu)>0$ such that for every $0<r \leq r^{*}$ for any $t_0\ge
0$ with $t_0+\tau \vr\le T$, any ball
$B_r(x_0)\subset\Omega$, and any weak supersolution $u\ge
\varepsilon>0$ of (\ref{MProb}) in $(0,t_0+\tau \vr)\times
B_r(x_0)$ with $u_0\ge 0$ in $B_r(x_0)$ and $f\equiv 0$, there is a constant $c=c(u)$ such that
\begin{equation} \label{logestleft}
\nuNj\big(\{(t,x)\in K_-: \log u(t,x)>c(u)+\lambda\}\big)\le C
\vr \nuN(B_r)\lambda^{-1},\quad \lambda>0,
\end{equation}
and
\begin{equation} \label{logestright}
\nuNj\big(\{(t,x)\in K_+: \log u(t,x)<c(u)-\lambda\}\big)\le C
\vr \nuN(B_r)\lambda^{-1},\quad \lambda>0,
\end{equation}
where $K_-=(t_0,t_0+\eta \tau \vr)\times B_{\delta r}(x_0)$ and
$K_+=(t_0+\eta \tau \vr,t_0+\tau \vr)\times
B_{\delta r}(x_0)$. Here the constant $C$ depends only on $\delta, \eta,
\tau, N, \mu, \nu$, and $\Lambda$.
\end{satz}
\begin{proof}
Similarly as in the previous proofs we abbreviate $B_{r} = B_{r}(x_{0})$. Let us take $\gmb\in (0,1)$ satisfying (\ref{intmugk}), fix $1 < p <\frac{1}{1-\gmb}$ and let $r \in (0,r^{*}]$, where $r^{*}$ denotes the minimum of $r^{*}(\mu)$ from Lemma~\ref{rtheta} and $r^{*}(\mu,p)$ from Lemma~\ref{scaling}. Then, in particular $\Phi(r^{*}) \leq 1$. Since $u_0\ge 0$ in $B_r$ and $u$ is a positive weak
supersolution we may assume that $u_0=0$. Without loss of generality we also may take $t_0=0$. Indeed, if $t_0>0$ we shift the time as $t\to
t-t_0$ and we obtain an inequality of the same type on the
time-interval $J:=[0,\tau \vr]$. Due to
$k\ast u\in C([0,t_0+\tau \vr];L_2(B))$, we have
$k\ast \tilde{u}\in C(J;L_2(B))$ for the shifted function
$\tilde{u}(s,x)=u(s+t_0,x)$. Thus $u$ satisfies
\begin{equation} \label{log1}
\int_{B_r} \Big(v \partial_t(k_{n}\ast u)+(h_n\ast
[ADu]|Dv)\Big)\,dx \ge \,0,\quad\mbox{a.a.}\;t\in J,\,n\in \iN,
\end{equation}
for any nonnegative test function $v\in \oH^1_2(B_r)$.

We begin similarly as in the proof of \cite[Theorem 3.3]{base}. We choose $\psi\in C^1_0(B_r)$ such that supp$\,\psi\subset B_r$, $\psi=1$ in
$B_{\delta r}$, $0\le \psi \le 1$, $|D \psi|\le 2/[(1-\delta)r]$ and
the domains $\{x\in B_r:\psi(x)^2 \ge b\}$ are convex for all $b\le
1$. Then for $t\in J$ we take the test function $v=\psi^2 u^{-1}$. We have
\[ Dv=2\psi D\psi \,u^{-1}-\psi^2 u^{-2}Du.\]
Inserting this into (\ref{log1}) we obtain for a.a. $t\in
J$
\begin{align} \label{log1a}
-\int_{B_r} \psi^2 u^{-1}\partial_t & (k_{n}\ast
u)\,dx+\int_{B_r}\big(ADu|u^{-2}Du\big)\psi^2\,dx\nonumber\\ &\le
2\int_{B_r}\big(ADu|u^{-1}\psi D\psi\big)\,dx+\mathcal{R}_n(t),
\end{align}
where
\[
\mathcal{R}_n(t)=\int_{B_r}\big(h_n\ast [ADu]-ADu|Dv\big)\,dx.
\]
Using (H1) and Young's inequality we find that
\[
\big|2\big(ADu|u^{-1}\psi D\psi\big)\big|\le 2\Lambda \psi
|D\psi|\,|Du| u^{-1}\le \frac{\nu}{2}\,\psi^2|Du|^2
u^{-2}+\frac{2}{\nu}\,\Lambda^2|D\psi|^2.
\]
Using this, assumption (H2) and the estimate $|D \psi|\le 2/[(1-\delta)r]$, we infer from
(\ref{log1a}) that for a.a. $t\in J$
\begin{equation} \label{log2}
-\int_{B_r} \psi^2 u^{-1}\partial_t(k_{n}\ast
u)\,dx+\frac{\nu}{2}\,\int_{B_r} |Du|^2 u^{-2} \psi^2\,dx \le \frac{8
\Lambda^2 \nuN(B_r)}{\nu (1-\delta)^2 r^2}\,+\mathcal{R}_n(t).
\end{equation}
We set $w=\log u$ (not to confuse with the weight function appearing in the definition of the kernel $k$), then  $Dw=u^{-1} Du$. Applying Proposition \ref{WeiPI} with weight $\psi^2$gives
\begin{equation} \label{log3}
\int_{B_r} (w-W)^2 \psi^2 dx \le \frac{8 r^2 \nuN(B_r)}{\int_{B_r} \psi^2
dx}\,\int_{B_r} |Dw|^2 \psi^2 dx,\quad \mbox{a.a.}\;t\in J,
\end{equation}
where
\[ W(t)=\,\frac{\int_{B_r} w(t,x) \psi(x)^2 dx}{\int_{B_r} \psi(x)^2
dx}\,,\quad\quad \mbox{a.a.}\;t\in J.
\]
From (\ref{log2}) and (\ref{log3}) it follows that
\[
-\int_{B_r} \psi^2 u^{-1}\partial_t(k_{n}\ast u)
\,dx+\,\frac{\nu \int_{B_r} \psi^2 dx}{16r^2 \nuN(B_r)}\,\int_{B_r} (w-W)^2
\psi^2 dx \le \frac{8 \Lambda^2 \nuN(B_r)}{\nu (1-\delta)^2
r^2}\,+\mathcal{R}_n(t),
\]
which in turn implies
\begin{equation} \label{log4}
\frac{-\int_{B_r} \psi^2 u^{-1}\partial_t(k_{n}\ast u)
\,dx}{\int_{B_r} \psi^2 dx}+\,\frac{\nu}{16r^2 \nuN(B_r)}\,\int_{{B_{\delta r}}} (w-W)^2 dx \le \frac{C_1}{r^2}\,+S_n(t),
\end{equation}
for a.a. $t\in J$, with some constant
$C_1=C_1(\delta,N,\nu,\Lambda)$ and
$S_n(t)=\mathcal{R}_n(t)/\int_{B_r} \psi^2 dx$.

Suppressing the spatial variable $x$, the identity (\ref{fundidentity}) with $H(y)=-\log y$
reads
\begin{align*}
-u^{-1}\partial_t(k_{n} & \ast
u)=-\partial_t(k_{n}\ast \log u)+(\log
u-1)k_{n}(t)\nonumber\\
&+\int_0^t \Big(-\log u(t-s)+\log
u(t)+\frac{u(t-s)-u(t)}{u(t)}\Big)[-\dot{k}_{n}(s)]\,ds.
\end{align*}
Rewriting this identity in terms of $w=\log u$ we obtain
\begin{align} \label{log5}
-u^{-1}\partial_t(k_{n}  \ast u)= & \,
-\partial_t(k_{n}\ast
w)+(w-1)k_{n}(t)\nonumber\\
& \,\,+\int_0^t
\Psi\big(w(t-s)-w(t)\big)[-\dot{k}_{n}(s)]\,ds,
\end{align}
where $\Psi(y)=e^y-1-y$. Since $\Psi$ is convex, it follows from
Jensen's inequality that
\[
\frac{\int_{B_r} \psi^2 \Psi\big(w(t-s,x)-w(t,x)\big)\,dx}{\int_{B_r}
\psi^2 dx} \ge \Psi \Big( \frac{\int_{B_r} \psi^2
\big(w(t-s,x)-w(t,x)\big)\,dx}{\int_{B_r} \psi^2 dx}\Big).
\]
Using this inequality in (\ref{log5}) we have
\begin{align}
\frac{-\int_{B_r} \psi^2 u^{-1}\partial_t(k_{n}\ast u)
\,dx}{\int_{B_r} \psi^2 dx} & \ge -\partial_t(k_{n}\ast
W)+(W-1)k_{n}(t)\nonumber\\
&\quad +\int_0^t
\Psi\big(W(t-s)-W(t)\big)[-\dot{k}_{n}(s)]\,ds \nonumber\\
& = -e^{-W} \partial_t(k_{n}\ast e^W), \label{log6}
\end{align}
where the last equality holds again due to (\ref{log5}) with $u$
replaced by $e^W$. From (\ref{log4}) and (\ref{log6}) we deduce
that
\begin{equation} \label{log7}
\frac{\nu}{16r^2 \nuN(B_r)}\,\int_{B_{\delta r}} (w-W)^2 dx \le e^{-W}
\partial_t(k_{n}\ast
e^W)+\,\frac{C_1}{r^2}\,+S_n(t),\quad \mbox{a.a.}\;t\in J.
\end{equation}

Next, we define
\eqq{
c(u) = \log \left(\frac{(k*e^{W})(\eta \tau \vr)}{A\int_{0}^{\eta \tau \vr}k(t)dt}\right),
}{cwahl}
where $A$ is a positive constant depending only on $C_{1}, \mu,\tau,\eta$ which will be chosen later.
We note that this definition makes sense, since $k\ast e^W\in C(J)$.
The latter is a consequence of $k\ast u\in C(J;L_2(B_r))$
and
\[
e^{W(t)}\le \,\frac{\int_{B_r} u(t,x) \psi(x)^2 dx}{\int_{B_r} \psi(x)^2
dx}\,,\quad\quad \mbox{a.a.}\;t\in J,
\]
where we apply again Jensen's inequality.

Similarly as in \cite{base}, to prove (\ref{logestleft}) and (\ref{logestright}), we use the inequalities
\begin{align}
\nuNj(\{(t,x) & \in K_-:\; w(t,x)>c(u)+\lambda\})\nonumber\\
\le &\;\nuNj(\{(t,x)\in K_-:
w(t,x)>c(u)+\lambda\;\,\mbox{and}\,\;W(t)\le c(u)+\lambda/2 \})\nonumber\\
&\; +\nuNj(\{(t,x)\in K_-:\,W(t)> c(u)+\lambda/2
\})=:I_1+I_2,\quad \lambda>0,\label{mainleft}\\
\nuNj(\{(t,x) & \in K_+:\; w(t,x)<c(u)-\lambda\})\nonumber\\
\le &\;\nuNj(\{(t,x)\in K_+:
w(t,x)<c(u)-\lambda\;\,\mbox{and}\,\;W(t)\ge c(u)-\lambda/2 \})\nonumber\\
&\; +\nuNj(\{(t,x)\in K_+:\,W(t)< c(u)-\lambda/2
\})=:I_3+I_4,\quad \lambda>0.\label{mainright}
\end{align}
We will estimate each of the four terms $I_j$ separately. As to $I_{1}$ and $I_3$ we follow the lines of proof of the analogous estimates from \cite[Theorem 3.3]{base}. However, it is important to notice that in the case when (\ref{am1}) does not hold, the arguments from \cite{base} to estimate $I_2$ and $I_4$ break down. Here, we present a new approach to estimate $I_2$ and also modify the estimate of $I_4$.

We begin with the estimates for $W$.  We set $J_-:=(0,\eta
\tau \vr)$, $J_+:=(\eta\tau \vr,\tau
\vr)$, and introduce for $\lambda>0$ the sets
$J_-(\lambda):=\{t\in J_-:\,W(t)> c(u)+\lambda \}$ and
$J_+(\lambda):=\{t\in J_+:\,W(t)< c(u)-\lambda \}$.

{\bf Estimate of $I_{2}$.}
Let us denote $\overline{w}:=e^{W}$. Multiplying (\ref{log7}) by $e^{W}$ we see that
\[
\partial_t (k_n * \overline{w}) + \frac{C_{1}}{r^{2}} \overline{w} + S_{n}(t)\overline{w} \geq 0.
\]
Denoting $\rho=\tau \vr$ and integrating the above inequality from $t$ to $\eta \rho$ we get
\[
(k_n * \overline{w})(\eta\rho) - (k_n * \overline{w})(t)  + \frac{C_{1}}{r^{2}} \int_{t}^{\eta\rho}\overline{w}(s)ds + \int_{t}^{\eta\rho}S_{n}(s)\overline{w}(s)ds \geq 0.
\]
We may choose a subsequence $n_m$ such that $(k_{n_m} * \overline{w})(\eta\rho)\rightarrow (k* \overline{w})(\eta\rho)$ for almost all $r$ and $(k_{n_m} * \overline{w})(t)\rightarrow (k* \overline{w})(t)$ for almost all $t \in J_{-}$. We proceed for such $r$ and $t$. Then we obtain
\[
(k* \overline{w})(\eta\rho) - (k* \overline{w})(t) + \frac{C_{1}}{r^{2}} \int_{t}^{\eta\rho}\overline{w}(s)ds \geq 0.
\]
Let us define $v(t) = \frac{(k*\overline{w})(t)}{(k*\overline{w})(\eta\rho)}$. Since $l*v = \frac{1*\overline{w}}{(k*\overline{w})(\eta\rho)}$ we have
\[
v(t) + \frac{C_{1}}{r^{2}}(l*v)(t) \leq 1 + \frac{C_{1}}{r^{2}}(l*v)(\eta\rho).
\]
Using positivity of $v$ and monotonicity of $l$ we arrive at
\begin{align*}
v(t) & \leq 1 - \frac{C_{1}}{r^{2}}\int_{0}^{t}[l(t-s)-l(\eta\rho-s)]v(s)ds +\frac{C_{1}}{r^{2}}\int_{t}^{\eta\rho}l(\eta\rho-s)v(s)ds \\
& \leq 1 + \frac{C_{1}}{r^{2}}\int_{t}^{\eta\rho}l(\eta\rho-s)v(s)ds.
\end{align*}
Let us denote $\overline{v}(t) = v(\eta\rho-t)$ for $t \in J_{-}$. Then, the last estimate may be rewritten as
\[
\overline{v}(t) \leq 1 + \frac{C_{1}}{r^{2}}\izt l(s)\overline{v}(s)ds,\quad t \in J_{-}.
\]
Next, we multiply the obtained inequality by $e^{-\beta t}$, where $\beta > 0$ is to be chosen later. Denoting $\overline{v}_\beta(t):=e^{-\beta t}\overline{v}(t)$ we have
\[
\overline{v}_{\beta}(t) \leq 1+ \frac{C_{1}}{r^{2}}(e^{-\beta \cdot}*l)(t)\norm{\overline{v}_{\beta}}_{L_{\infty}(J_-)} \leq 1+\frac{C_{1}}{r^{2}}\norm{e^{-\beta \cdot}}_{L_{p'}((0,t))}\norm{l}_{L_{p}((0,t))}\norm{\overline{v}_{\beta}}_{L_{\infty}(J_-)}
\]
\[
\leq 1 + \frac{C_{1}}{r^{2}}(\beta p')^{-\frac{1}{p'}}\norm{l}_{L_{p}((0,t))}\norm{\overline{v}_{\beta}}_{L_{\infty}(J_-)}
\]
where $p < \frac{1}{1-\gmb}$ is fixed.
We note that $r/2 \leq r^{*}(\mu,p)$, where $r^{*}(\mu,p)$ comes from Lemma \ref{scaling}. We proceed as in (\ref{estiPhieta}) and next we apply (\ref{zn2}) to obtain
\[
\overline{v}_{\beta}(t) \leq 1+ 4 c(\tau,\mu,p) (\Phi(r))^{-\frac{p-1}{p}}C_{1}(\beta p')^{-\frac{1}{p'}}\norm{\overline{v}_{\beta}}_{L_{\infty}(J_-)}.
\]
We choose
\[
\beta = \frac{1}{\vr}\frac{p-1}{p} (8 c(\tau,\mu,p)C_{1})^{\frac{p}{p-1}},
\]
then
\[
4 c(\tau,\mu,p) (\Phi(r))^{-\frac{p-1}{p}}C_{1}(\beta p')^{-\frac{1}{p'}} = \frac{1}{2} \hd \m{ and } \hd \norm{\overline{v}_{\beta}}_{L_{\infty}(J_-)} \leq 2.
\]
Thus, we have
\[
\frac{\int_{0}^{\eta\rho}{\overline{w}(s)ds}}{(k*\overline{w})(\eta\rho)} = (l*v)(\eta\rho) = \int_{0}^{\eta\rho}l(s)\overline{v}(s)ds \leq 2 \int_{0}^{\eta\rho}l(s)e^{\beta s}ds\leq 2e^{\beta \eta\rho} \int_{0}^{\eta\rho}l(s) ds.
\]
By the definition of $\beta$ and $\rho$ we arrive at
\eqq{
\frac{\int_{0}^{\eta\rho}{\overline{w}(s)ds}}{(k*\overline{w})(\eta\rho)} \leq c(\tau,C_{1},p,\mu) \int_{0}^{\eta\rho}l(s) ds \leq c(\tau,C_{1},p,\mu) r^{2},
}{rz2}
where in the last estimate we used (\ref{zn2}) with $p=1$. In order to abbreviate the notation let us also denote by $|\cdot| $ the one-dimensional Lebesgue measure. Then, for $\lambda>0$ we obtain
\begin{align*}
e^\lambda | J_-(\lambda) | & =
e^\lambda\big|\{t\in J_-:\,
e^{W(t)}>e^{c(u)}e^{\lambda}\}\big|=\int_{J_-(\lambda)}e^\lambda \,dt\\
& \le \int_{J_-(\lambda)}e^{W(t)-c(u)} \,dt\le
\int_{J_-}e^{W(t)-c(u)} \,dt =\frac{A \int_{0}^{\eta \rho}k(t)dt}{(k*e^{W})(\eta \rho)} \int_{0}^{\eta \rho}e^{W(t)}dt.
\end{align*}
Using this and (\ref{rz2}) we get
\[
e^\lambda |J_-(\lambda)| \leq  c(\mu,\tau,C_{1},\eta)Ar^2 \int_{0}^{\eta \rho}k(t)dt
\leq c(\mu,\tau,C_{1},\eta) A\vr r^{2} k_{1}(\vr) = c(\mu,\tau,C_{1},\eta) A\vr.
\]
\nic{and therefore
\[
e^\lambda |J_-(\lambda)| \leq  c(\mu,\tau,\delta,N,\nu,\Lambda,\eta)\vr.
\]}
Hence, we have
\begin{equation} \label{I2est}
I_2=| J_-(\lambda/2)| \nuN( B_{\delta r})\le\,A\frac{c(\mu,\tau,\delta,N,\nu,\Lambda,\eta)}{\lambda}\,\vr\nuN(B_r),\quad \lambda>0.
\end{equation}

{\bf Estimate of $I_{4}$.} We define the function $H_m$ on $\iR$
by $H_m(y)=y$, $y\le m$, and $H_m(y)=m+(y-m)/(y-m+1)$, $y\ge m$, $m>0$.
Note that $H_m\in C^1(\iR)$ is increasing, concave, and bounded above by $m+1$.
Moreover, by concavity
\begin{equation} \label{Hmprop}
0\le yH_m'(y)\le H_m(y)\le m+1,\quad y\ge 0.
\end{equation}
Then, from (\ref{fundidentity}) we obtain
\begin{equation} \label{FIappl}
\partial_t\Big(k_{n}\ast
H_m\big(e^W\big)\Big) \geq H_{m}'(e^{W})\partial_t\Big(k_{n}\ast
e^W\Big),
\end{equation}
and thus multiplying (\ref{log7}) by $e^W H_m'\big(e^W\big)$ and employing
\eqref{FIappl}, we infer that
\begin{equation} \label{West1}
\partial_t\Big(k_{n}\ast
H_m\big(e^W\big)\Big)+\,\frac{C_1}{r^2}\,H_m\big(e^W\big)\ge - S_n
e^W H_m'\big(e^W\big),\quad \mbox{a.a.}\;t\in J.
\end{equation}
For $t\in J_+$ we shift the time by setting $s=t-\eta\tau
\vr=t-\eta \rho$ and put $\tilde{f}(s)=f(s+\eta \rho)$,
$s\in (0,(1-\eta)\rho)$, for functions $f$ defined on $J_+$. Using
the time-shifting identity (\ref{shiftprop}), (\ref{West1})
implies that for a.a. $s\in (0,(1-\eta)\rho)$
\begin{equation} \label{West2}
\partial_s\Big(k_{n}\ast
H_m\big(e^{\tilde{W}}\big)\Big)+\,\frac{C_1}{r^2}\,H_m\big(e^{\tilde{W}}\big)\ge
\Upsilon_{n,m}(s)- \tilde{S}_n e^{\tilde{W}}
H_m'\big(e^{\tilde{W}}\big),
\end{equation}
where $\Upsilon_{n,m}$ denotes the history term
\[
\Upsilon_{n,m}(s)=\int_0^{\eta\rho}\big[-\dot{k}_{n}(s+\eta\rho-\sigma)\big]
H_m\big(e^{W(\sigma)}\big)\,d\sigma, \hd s\in (0,(1-\eta)\rho).
\]

Now, we shall deduce the estimate from below for $e^{\tilde{W}}$. For this purpose we set $\theta = \frac{C_{1}}{r^{2}}$ and convolve (\ref{West2}) with
$r_{\theta}$, where $r_{\theta}$ satisfies (\ref{reso}). We have a.e. in $(0,(1-\eta)\rho)$
\begin{align*}
r_{\theta}\ast \partial_s\Big(k_{n} & \ast
H_m\big(e^{\tilde{W}}\big)\Big)\,=\partial_s\Big(r_{\theta}\ast
k_{n}\ast H_m\big(e^{\tilde{W}}\big)\Big)\\
&\,=\partial_s\Big([l-\theta (r_{\theta}\ast
l)]\ast k_{n}\ast H_m\big(e^{\tilde{W}}\big)\Big)\\
&\,=h_n\ast H_m\big(e^{\tilde{W}}\big) -\theta
r_{\theta}\ast h_n\ast H_m\big(e^{\tilde{W}}\big),
\end{align*}
thus
\begin{align} \label{West3}
h_n\ast H_m\big(e^{\tilde{W}}\big)\ge &
\,\,r_{\theta}\ast \Upsilon_{n,m}-
r_{\theta}\ast\big[\tilde{S}_n
e^{\tilde{W}} H_m'\big(e^{\tilde{W}}\big)\big]\nonumber\\
& \,\,+\theta h_n\ast r_{\theta}\ast
H_m\big(e^{\tilde{W}}\big) -\theta r_{\theta}\ast
H_m\big(e^{\tilde{W}}\big)  \hd\m{ a.e. in } \hd (0,(1-\eta)\rho).
\end{align}
Passing to the limit with  $n$ (on an appropriate subsequence), it follows that
\begin{equation} \label{West4}
H_m\big(e^{\tilde{W}}\big)\ge r_{\,\theta}\ast
\Upsilon_{m},\quad\quad \mbox{a.a.}\;s\in(0,(1-\eta)\rho),
\end{equation}
where
\[
\Upsilon_{m}(s)=\int_0^{\eta\rho}\big[-\dot{k}(s+\eta\rho-\sigma)\big]
H_m\big(e^{W(\sigma)}\big)\,d\sigma.
\]
Applying Fubini's theorem we obtain
\eqq{
H_m\big(e^{\tilde{W}(s)}\big)\ge \int_{0}^{\eta\rho}H_m\big(e^{W(\xi)}\big)\int_{0}^{s} r_{\theta}(s-\sigma)\big[-\dot{k}(\sigma+\eta\rho-\xi)\big]d\sigma d\xi, \hd s\in (0,(1-\eta)\rho).
}{West33}
In order to estimate the inner integral we apply (\ref{rtheta1}) from  Lemma \ref{rtheta} with $\bar{C} = (1-\eta)\tau$. Then we have 
\[
\int_{0}^{s}r_{\theta}(s-\sigma)\izj \frac{\al}{\Gamma(1-\al)}(\sigma +  \eta \rho -\xi)^{-\al-1}\dd d\sigma = \podst{\sigma = s\zeta}{d\sigma = s d\zeta}
\]
\[
=s \izj r_{\theta}(s(1-\zeta))\izj \frac{\al}{\Gamma(1-\al)}(s\zeta + \eta \rho-\xi)^{-\al-1}\dd d\zeta
\]
\[
\geq c(\mu,C_{1},\eta,\tau)s  \izj \frac{1}{\int_{0}^{1}s^{1-\beta}(1-\zeta)^{1-\beta} d\mu(\beta)} \izj \frac{\al}{\Gamma(1-\al)}(s\zeta + \eta \rho-\xi)^{-\al-1}\dd d\zeta
\]
\[
\geq c(\mu,C_{1},\eta,\tau) \frac{1}{\int_{0}^{1}s^{1-\beta}d\mu(\beta)} \izj  \frac{\al}{\Gamma(1-\al)}s^{-\al}\izj \left(\zeta + \frac{\eta \rho-\xi}{s}\right)^{-\al-1} d\zeta \dd
\]
\[
=c(\mu,C_{1},\eta,\tau) \frac{1}{\int_{0}^{1}s^{1-\beta}d\mu(\beta)} \izj  \frac{1}{\Gamma(1-\al)} s^{-\al} \left[\left(\frac{\eta\rho-\xi}{s}\right)^{-\al} - \left(1+\frac{\eta\rho-\xi}{s}\right)^{-\al}\right]\dd
\]

\[
= c(\mu,C_{1},\eta,\tau) \frac{1}{\int_{0}^{1}s^{1-\beta}d\mu(\beta)} \izj  \frac{1}{\Gamma(1-\al)}  (\eta\rho-\xi)^{-\al} \left(1-\left(\frac{\eta\rho-\xi}{s+\eta\rho-\xi}\right)^{\al}\right) \dd
\]
\[
\geq c(\mu,C_{1},\eta,\tau) \frac{1}{\int_{0}^{1}s^{1-\beta}d\mu(\beta)} \izj \frac{\al}{\Gamma(1-\al)} (\eta \rho-\xi)^{-\al} \dd\frac{s}{s+\eta\rho-\xi},
\]
where in the last inequality we used the estimate $1-x^{\al}\geq \al (1-x)$ for $x\in (0,1)$.
We note that since $\rho = \tau \vr$, we have
\[
\izj \frac{\al}{\Gamma(1-\al)} (\eta \rho-\xi)^{-\al} \dd = \izj \frac{\al}{\Gamma(1-\al)} (\eta \tau)^{-\al}(\vr - \frac{\xi}{\eta \tau})^{-\al}\dd.
\]
 Recalling that $\Phi(r) \leq 1$ we may apply (\ref{estioddolualfa}) to get
\[
\izj \frac{\al}{\Gamma(1-\al)} (\eta \rho-\xi)^{-\al} \dd
\geq \min\{(\eta \tau)^{-1},1\}\izj \frac{\al}{\Gamma(1-\al)} \left(\vr - \frac{\xi}{\eta \tau}\right)^{-\al}\dd
\]
\[
\geq \min\{(\eta \tau)^{-1},1\} c(\mu)\izj \frac{1}{\Gamma(1-\al)} (\vr - \frac{\xi}{\eta \tau})^{-\al}\dd
\]
\[
\geq \min\{(\eta \tau)^{-1},1\} \min\{(\eta \tau),1\} c(\mu)\izj \frac{1}{\Gamma(1-\al)} (\eta \rho -\xi)^{-\al}\dd = c(\eta,\tau,\mu)k(\eta\rho-\xi).
\]
Thus, we have
\[
\int_{0}^{s}r_{\theta}(s-\sigma)\izj \frac{\al}{\Gamma(1-\al)}(\sigma + \eta \rho -\xi)^{-\al-1}\dd d\sigma \geq c(\mu,C_{1},\eta,\tau)k(\eta \rho -\xi) \frac{1}{\int_{0}^{1}s^{-\al}\dd (\eta\rho+s)}.
\]
If we denote the constant from the estimate above by $C_{2}=c(\mu,C_{1},\eta,\tau)$ and apply this estimate in (\ref{West33})  we arrive at
\[
H_m\big(e^{\tilde{W}(s)}\big)\ge C_{2}\frac{1}{\rho\int_{0}^{1}s^{-\al}\dd  }\big(k\ast
H_m\big(e^{W}\big)\big) (\eta\rho),\quad \mbox{a.a.}\;s\in
(0,(1-\eta)\rho).
\]
Passing to the limit with $m$ we conclude
that
\begin{equation} \label{West5}
e^{\tilde{W}(s)}\ge C_{2}\frac{1}{\rho\int_{0}^{1}s^{-\al}\dd  }\big(k\ast
e^{W}\big) (\eta\rho),\quad \mbox{a.a.}\;s\in (0,(1-\eta)\rho),
\end{equation}
where we applied Fatou's lemma and the fact that  $H_m(y)\nearrow y$ as $m\to \infty$.
Thanks to (\ref{West5}) we may estimate as follows,
\[
\lambda |J_+(\lambda)| = \int_{J_+(\lambda)} \lambda dt
 \le \int_{J_+(\lambda)} \big(c(u)-W(t)\big) dt= \int_{J_+(\lambda)-\eta\rho} \big(c(u)-\tilde{W}(s)\big) ds\\
\]
\[
\leq \int_{J_+(\lambda)-\eta\rho} \Big[\log \left(\frac{(k*e^{W})(\eta \rho)}{A\int_{0}^{\eta \rho}k(t)dt}\right) - \log \left(\frac{(k*e^{W})(\eta \rho)}{\frac{1}{C_{2}}\rho\izj s^{-\al}\dd}\right)\Big]ds
\]
\[
= \int_{J_+(\lambda)-\eta\rho} \log \left(\frac{\frac{1}{C_{2}}\rho\izj s^{-\al}\dd}{A\int_{0}^{\eta \rho}k(t)dt}\right)ds.
\]
We note that for any  $A < \frac{1}{2C_{2}}$ the expression under the integral is nonnegative on the whole interval $(0,(1-\eta)\rho)$. Indeed, for any $s \in (0,(1-\eta)\rho)$ we have,
\[
\frac{\frac{1}{C_{2}}\rho\izj s^{-\al}\ma d\al}{\int_{0}^{\eta \rho}k(t)dt} \geq \frac{\frac{1}{C_{2}}\izj (1-\eta)^{-\al}\rho^{1-\al}\dd}{\int_{0}^{1}\frac{1}{\Gamma(2-\al)}\eta^{1-\al}\rho^{1-\al}\dd} \geq \frac{1}{2C_{2}}.
\]
We choose
\eqq{
A=\frac{1}{4 C_{2}}
}{numerek}
and continue
\begin{align*}
\lambda |J_+(\lambda)| & \leq \int_{0}^{(1-\eta)\rho} \Big[\log \left(\rho\izj s^{-\al}\dd\right) - \log \left(\frac{1}{4 }\int_{0}^{1}\frac{1}{\Gamma(2-\al)}\eta^{1-\al}\rho^{1-\al}\dd\right)\Big]ds \\
& =(1-\eta)\rho \log \left(\izj (1-\eta)^{-\al}\rho^{1-\al}\dd\right) \\
& \quad + \int_{0}^{(1-\eta)\rho} \frac{\izj \al s^{-\al}\dd}{\izj s^{-\al}\dd}ds - (1-\eta)\rho \log \left(\frac{1}{4}\int_{0}^{1}\frac{ 1}{\Gamma(2-\al)}\eta^{1-\al}\rho^{1-\al}\dd\right)\\
& \leq (1-\eta)\rho + (1-\eta)\rho \log \left(\frac{4\izj (1-\eta)^{-\al}\rho^{1-\al}\dd}{\int_{0}^{1}\frac{1}{\Gamma(2-\al)}\eta^{1-\al}\rho^{1-\al}\dd}\right).
\end{align*}
We note that
\[
\log \left(\frac{4\izj (1-\eta)^{-\al}\rho^{1-\al}\dd}{\int_{0}^{1}\frac{1}{\Gamma(2-\al)}\eta^{1-\al}\rho^{1-\al}\dd}\right) \leq \log \left(\frac{4}{(1-\eta)\eta}\right)
\]
and consequently, we obtain
\eqnsl{
\lambda|J_{+}(\lambda)| \leq
c(\eta,\tau) \vr.
}{estidwadwa}
Hence,
\begin{equation} \label{I4est}
I_4=|J_+(\lambda/2)|\nuN(
B_{\delta r})\le\,\frac{c(\eta,\tau)
\delta^N}{\lambda}\,\vr\nuN(B_r),\quad \lambda>0.
\end{equation}

{\bf Estimate of $I_{1}$.}
In order to estimate $I_1$ we set $J_1(\lambda)=\{t\in
J_-:\,c(u)-W(t)+\lambda/2\ge 0\}$ and $\Omega^-_t(\lambda)=\{x\in
B_{\delta r}:\,w(t,x)>c(u)+\lambda\},\,t\in J_1(\lambda)$, where $c(u)$
is given by (\ref{cwahl}). Then, for $t\in J_1(\lambda)$, we have
\[ w(t,x)-W(t)>c(u)-W(t)+\lambda\ge \lambda/2,\quad x\in
\Omega^-_t(\lambda),\] and thus (\ref{log7}) gives

\begin{equation} \label{log8}
\frac{\nu}{16r^2
\nuN(B_r)}\,\,\nuN\big(\Omega^-_t(\lambda)\big)\le
\frac{1}{(c(u)-W+\lambda)^2}\,\Big(e^{-W}
\partial_t(k_{n}\ast
e^W)+\,\frac{C_1}{r^2}\,+S_n\Big) \hd \m{ a.e. in } \hd J_1(\lambda).
\end{equation}
We set $\chi(t,\lambda)=\nuN\big(\Omega^-_t(\lambda)\big)$, if $t\in
J_1(\lambda)$, and $\chi(t,\lambda)=0$ in case $t\in J_-\setminus
J_1(\lambda)$. We define $H(y)=(c(u)-\log y+\lambda)^{-1},\,0<y\le
y_*:=e^{c(u)+\lambda/2}$. Then, $H'(y)= (c(u)-\log
y+\lambda)^{-2}y^{-1}$ and
\[ H''(y)=\,\frac{1}{(c(u)-\log y+\lambda)^2 y^2}\,\Big(\frac{2}{c(u)-\log
y+\lambda}-1\Big),\quad 0<y\le y_*,\] and so we see that $H$ is
concave in $(0,y_*]$ whenever $\lambda\ge 4$. We will assume the latter in
what follows.

We next choose a $C^1$ extension $\bar{H}$ of $H$ on $(0,\infty)$
such that $\bar{H}$ is concave, $0\le\bar{H}'(y)\le
\bar{H}'(y_*),\,y_*\le y \le 2 y_*$, and $\bar{H}'(y)=0,\,y\ge 2
y_*$. Then
\begin{equation} \label{log8a}
 0\le y\bar{H}'(y)\le
\,\frac{2}{\lambda},\quad y>0.
\end{equation}
Indeed, for $y\in(0,y_*]$ we have
\begin{equation} \label{log8aa}
y\bar{H}'(y)=\,\frac{1}{(c(u)-\log y+\lambda)^2}\,\le
\,\frac{1}{(c(u)-\log y_*+\lambda)^2}\,\le \,\frac{4}{\lambda^2}\,\le
\,\frac{1}{\lambda},
\end{equation}
while in case $y\in[y_*,2y_*]$ we may simply estimate
\[
y\bar{H}'(y)\le 2y_*\bar{H}'(y_*)\le \,\,\frac{2}{\lambda}.
\]
Next, we shall show that
\begin{equation} \label{log8b}
\bar{H}(y)\le \,\frac{3}{\lambda},\quad y>0.
\end{equation}
Firstly note that since $\bar{H}$ is nondecreasing with
$\bar{H}'(y)=0$ for all $y\ge 2 y_*$, the claim follows if the
inequality is valid for all $y\in [y_*,2y_*]$. For  $y\in [y_*,2y_*]$
by concavity of $\bar{H}$ and (\ref{log8aa}) there holds
\[
\bar{H}(y)\le \bar{H}(y_*)+\bar{H}'(y_*)(y-y_*)\le
\bar{H}(y_*)+y_*\bar{H}'(y_*)\le \,\frac{3}{\lambda}
\]
and we arrive at (\ref{log8b}). Furthermore,
\[ e^{W(t)} H'(e^{W(t)})=\,\frac{1}{(c(u)-W(t)+\lambda)^2
},\quad \mbox{a.a.}\;t\in J_1(\lambda).\] Since $\bar{H}'\ge 0$,
and $e^{-W}
\partial_t(k_{n}\ast
e^W)+C_1 r^{-2}+S_n\ge 0$ on $J_-$ by virtue of (\ref{log7}), we
deduce from (\ref{log8}) and (\ref{log8a}) that
\begin{align} \label{log9}
\frac{\nu}{16r^2 \nuN(B_r)}\,\,\chi(t,\lambda) & \le
e^W\bar{H}'(e^{W})\Big(e^{-W}\partial_t(k_{n}\ast
e^W)+\,\frac{C_1}{r^2}\,+S_n\Big)\nonumber\\
& \le \bar{H}'(e^{W})\partial_t(k_{n}\ast
e^W)+\,\frac{2C_1}{\lambda r^2}\,+\,\frac{2|S_n(t)|}{\lambda},
\quad\mbox{a.a.}\; t\in J_-.
\end{align}
Using the identity (\ref{fundidentity}) with concave $\bar{H}$ we obtain
\begin{align*}
\bar{H}'(e^{W})\partial_t(k_{n}\ast e^W) & \le
\partial_t\Big(k_{n}\ast \bar{H}\big(e^W\big)\Big)
+ \Big(-\bar{H}(e^{W})+\bar{H}'(e^{W})
e^W\Big)k_{n}\\
& \le \partial_t\Big(k_{n}\ast
\bar{H}\big(e^W\big)\Big)+\,\frac{2}{\lambda}\,k_{n},\quad
\mbox{a.a.}\; t\in J_-,
\end{align*}
which, together with (\ref{log9}), gives a.e. in $J_-$
\begin{equation} \label{log10}
\frac{\nu}{16r^2 \nuN(B_r)}\,\,\chi(t,\lambda) \le
\partial_t\Big(k_{n}\ast \bar{H}\big(e^W\big)\Big)
+\,\frac{2}{\lambda}\,k_{n}+ \,\frac{2C_1}{\lambda
r^2}\,+\,\frac{2|S_n(t)|}{\lambda}\,.
\end{equation}
We then integrate (\ref{log10}) over $J_-=(0,\eta \rho)$ and
employ (\ref{log8b}) for the estimate
\[ \Big(k_{n}\ast
\bar{H}\big(e^W\big)\Big)(\eta\rho)\le
\,\frac{3}{\lambda}\,\int_0^{\eta\rho}k_{n}(t)\,dt.
\]
Passing to the limit $n\to \infty$ yields
\[
\int_{J_1(\lambda)}\nuN   \big(\Omega^-_t(\lambda)\big) \,dt = \int_0^{\eta\rho} \chi(t,\lambda)\,dt
\leq \frac{16 r^{2}\nuN(B_r)}{\nu}\left(\frac{5}{\lambda} \int_{0}^{\eta \rho } k(t)dt + \frac{2C_{1}}{\lambda r^{2}} \eta \rho
\right)
\]
\[
= \frac{16 r^{2}\nuN(B_r)}{\nu \lambda } \left(5
\izj \frac{1}{\Gamma(2-\al)}(\eta \rho)^{1-\al}\dd + \frac{2C_{1}\eta\rho}{ r^{2}}\right).
\]
Note that
\[
\izj \frac{1}{\Gamma(2-\al)}(\eta \rho)^{1-\al}
\dd \leq  c(\mu,\eta,\tau)\vr \ki(\vr)=c(\mu,\eta,\tau) \frac{\vr}{r^{2}},
\]
where we used (\ref{zn1}). Therefore, we obtain
\[
I_{1}=\int_{J_1(\lambda)}\nuN   \big(\Omega^-_t(\lambda)\big) \,dt \leq c(\mu,\eta,\tau, \nu,\delta)\frac{\vr\nuN(B_r)}{\lambda},\quad\lambda\ge
4.
\]
For $\lambda <4$ we trivially have
\[
I_{1} \leq |K_{-}|= (1-\eta) \tau \vr \delta^{N} \nuN(B_r) \leq 4\tau \delta^{N} \frac{\vr\nuN(B_r)}{\lambda},
\]
and hence
\begin{equation} \label{I1est}
I_1\le \,c(\mu,\eta,\tau, \nu,\delta)\frac{ \vr\nuN(B_r)}{\lambda},\quad\lambda>0.
\end{equation}

{\bf Estimate of $I_{3}$.}
In order to estimate $I_3$ we shift again the time by putting $s=t-\eta\rho$, and
denote the corresponding transformed functions as above by
$\tilde{W}$, $\tilde{w}$, ... and so forth. Further, we set
$\tilde{J}_+:=(0,(1-\eta)\rho)$. By the time-shifting property
(\ref{shiftprop}) and by positivity of $e^W$, relation
(\ref{log7}) then yields
\begin{equation} \label{log10a}
\frac{\nu}{16r^2 \nuN(B_r)}\,\int_{B_{\delta r}}
(\tilde{w}-\tilde{W})^2 dx \le e^{-\tilde{W}}
\partial_s(k_{n}\ast
e^{\tilde{W}})+\,\frac{C_1}{r^2}\,+\tilde{S}_n(s),\quad
\mbox{a.a.}\;s\in \tilde{J}_+.
\end{equation}
Next, set $J_2(\lambda)=\{s\in
\tilde{J}_+:\tilde{W}(s)-c(u)+\lambda/2\ge 0\}$ and
$\Omega_{s}^+(\lambda)=\{x\in B_{\delta r}:
\tilde{w}(s,x)<c(u)-\lambda\},\,s\in J_2(\lambda)$. For $s\in
J_2(\lambda)$, we have
\[ \tilde{W}(s)-\tilde{w}(s,x)\ge
\tilde{W}(s)-c(u)+\lambda\ge \lambda/2,\quad
x\in\Omega_{s}^+(\lambda),
\]
and thus (\ref{log10a}) implies that a.e. in $J_2(\lambda)$
\begin{equation} \label{log12}
\frac{\nu}{16r^2
\nuN(B_r)}\,\,\nuN\big(\Omega^+_s(\lambda)\big)\le
\frac{1}{(\tilde{W}-c(u)+\lambda)^2}\,\Big(e^{-\tilde{W}}
\partial_s(k_{n}\ast
e^{\tilde{W}})+\,\frac{C_1}{r^2}\,+\tilde{S_n}\Big).
\end{equation}

We proceed now similarly as above for the term $I_1$. Set
$\chi(s,\lambda)=\nuN\big(\Omega^+_{s}(\lambda)\big)$, if $s\in
J_2(\lambda)$, and $\chi(s,\lambda)=0$ in case $s\in
\tilde{J}_+\setminus J_2(\lambda)$. We consider this time the
convex function $H(y)=(\log y-c(u)+\lambda)^{-1}$ for $y\ge
y_*:=e^{c(u)-\lambda/2}$ with derivative $H'(y)=-(\log
y-c(u)+\lambda)^{-2} y^{-1}<0$. We define a $C^1$ extension $\bar{H}$
of $H$ on $[0,\infty)$ by means of
\[
\bar{H}(y)=\left\{ \begin{array}{l@{\;:\;}l}
H'(y_*)(y-y_*)+H(y_*) & 0\le y< y_* \\
H(y) & y\ge y_*.
\end{array} \right.
\]
Evidently, $-\bar{H}$ is concave in $[0,\infty)$ and
\begin{equation} \label{log13}
0\le-\bar{H}'(y)y \le \,\frac{1}{(\log
y_*-c(u)+\lambda)^2}\,\le\,\frac{1}{(\lambda/2)^2}\,\le
\,\frac{4}{\lambda},\quad y \ge 0,\;\lambda\ge 1 .
\end{equation}
We will assume $\lambda\ge 1$ in the following lines.
Observe that
\[ -e^{\tilde{W}(s)} H'(e^{\tilde{W}(s)})=\,\frac{1}{(\tilde{W}(s)-c(u)+\lambda)^2
},\quad \mbox{a.a.}\;s\in J_2(\lambda).\]
Since $-\bar{H}'\ge 0$,
and $e^{-\tilde{W}}
\partial_s(k_{n}\ast
e^{\tilde{W}})+C_1 r^{-2}+\tilde{S_n}\ge 0$ on $\tilde{J}_+$ due
to (\ref{log10a}), it thus follows from (\ref{log12}) and
(\ref{log13}) that
\begin{align} \label{log14}
\frac{\nu}{16r^2 \nuN(B_r)}\,\,\chi(s,\lambda) & \le
-e^{\tilde{W}}\bar{H}'(e^{\tilde{W}})\Big(e^{-\tilde{W}}\partial_s(k_{n}\ast
e^{\tilde{W}})+\,\frac{C_1}{r^2}\,+\tilde{S}_n\Big)\nonumber\\
& \le -\bar{H}'(e^{\tilde{W}})\partial_s(k_{n}\ast
e^{\tilde{W}})+\,\frac{4C_1}{\lambda
r^2}\,+\,\frac{4|\tilde{S}_n(s)|}{\lambda}, \quad\mbox{a.a.}\;
s\in \tilde{J}_+.
\end{align}
Using the fundamental identity
(\ref{fundidentity}) and concavity of $-\bar{H}$  we obtain
\begin{align*}
-\bar{H}'(e^{\tilde{W}}) & \partial_s(k_{n}\ast
e^{\tilde{W}})  \le -\partial_s\Big(k_{n}\ast
\bar{H}\big(e^{\tilde{W}}\big)\Big)+\Big(\bar{H}(e^{\tilde{W}})-
\bar{H}'(e^{\tilde{W}})e^{\tilde{W}}\Big)k_{n} \\
& \le -\partial_s\Big(k_{n}\ast
\bar{H}\big(e^{\tilde{W}}\big)\Big)+\bar{H}(0)k_{n} \le
-\partial_s\Big(k_{n}\ast
\bar{H}\big(e^{\tilde{W}}\big)\Big)
+\,\frac{6}{\lambda}\,k_{n},
\end{align*}
a.e. in $\tilde{J}_+$. This together with (\ref{log14})
gives
\[
\frac{\nu}{16r^2 \nuN(B_r)}\,\,\chi(s,\lambda) \le
-\partial_s\Big(k_{n}\ast
\bar{H}\big(e^{\tilde{W}}\big)\Big)
+\,\frac{6}{\lambda}\,k_{n}+\,\frac{4C_1}{\lambda
r^2}\,+\,\frac{4|\tilde{S}_n(s)|}{\lambda},
\]
for a.a. $s\in \tilde{J}_+$. We integrate this estimate over
$\tilde{J}_+$ and send $n\to \infty$ to the result
\[
I_{3}=\int_{J_2(\lambda)}\nuN\big(   \Omega^+_s(\lambda)\big) \,ds
  = \int_0^{(1-\eta)\rho} \!\!\!\!\chi(s,\lambda)\,ds\le
  \,\frac{16r^2\nuN(B_r)}{\nu}\,\Big(\,\frac{6}{\lambda}\,\int_{0}^{(1-\eta)\rho} k(s)ds +\,\frac{4C_1(1-\eta)\rho}{\lambda
  r^2}\,\Big)
\]
\[
= \,\frac{16r^2\nuN(B_r)}{\nu}\,\Big(\,\frac{6}{\lambda}\,\izj \frac{1}{\Gamma(2-\al)}((1-\eta) \rho)^{1-\al}\dd+\,\frac{4C_1(1-\eta)\rho}{\lambda
  r^2}\,\Big)
\]
\[
\leq c(\tau , C_{1}, \nu) \frac{r^{2}\nuN(B_r)}{\lambda} \Big( \kjr \vr+\,\frac{\vr}{
  r^2}\,\Big)
\]
\[
= c(\tau , C_{1}, \nu, \mu) \frac{\nuN(B_r) \vr}{\lambda},\quad \lambda\ge 1,
\]
where in the last equality we applied (\ref{zn1}). For $\lambda \leq 1$ we simply have
\[
I_{3}\leq |K_{+}|=(1-\eta)\tau \vr \delta^{N} \nuN(B_r)\leq \tau \delta^{N} \frac{\vr \nuN(B_r)}{\lambda},
\]
therefore we obtain
\begin{equation} \label{I3est}
I_3\le \, c(\tau , C_{1}, \nu, \mu,\delta) \frac{\vr\nuN(B_r)}{\lambda},\quad\lambda>0.
\end{equation}

Finally, combining (\ref{mainleft}), (\ref{mainright}),
(\ref{I2est}),  (\ref{numerek}),  (\ref{I4est}), (\ref{I1est}), and (\ref{I3est})
we obtain the claim.
\end{proof}
\begin{bemerk1}
    We notice that as a byproduct of the proof of Theorem \ref{logest}, we obtain robust logarithmic estimates for $\mu=\delta(\cdot-\alpha)$
    as $\alpha \rightarrow 1$. We point out that in \cite{base}, the constants in the estimates of $I_2$ and $I_4$ blow up as $\alpha \rightarrow 1$. Here, we provide new estimates which in the case of a single order fractional derivative are uniform with respect to the order of derivative $\alpha\in [\alpha_0,1)$, with an arbitrarily fixed $\alpha_0\in (0,1).$
\end{bemerk1}
\subsection{The final step}
We are now in position to prove Theorem \ref{localweakHarnack}.
Without loss of generality we may assume that $u\ge \varepsilon$ for
some $\varepsilon>0$; otherwise replace $u$ by $u+\varepsilon$,
which is a supersolution of (\ref{MProb}) with $u_0+\varepsilon$
instead of $u_0$, and eventually let $\varepsilon\to 0+$.

For $0<\sigma\le 1$, we set $U_\sigma=(t_0+(2-\sigma)\tau
\Phi(2r),t_0+2\tau \Phi(2r))\times B_{\sigma r}(x_{0})$ and
$U'_\sigma=(t_0,t_0+\sigma\tau \Phi(2r))\times B_{\sigma r}(x_{0}) $.
Clearly, from (\ref{defQpm}) we have $Q_-(t_0,x_0,r,\delta)=U'_\delta$ and $Q_+(t_0,x_0,r,\delta)=U_\delta$.

From  Theorem \ref{superest1} applied with $\eta:=\tau$ and $t_{0}:=t_{0}+2\tau \Phi(2r)$ we obtain for $0 < r \leq r^{*}=r^{*}(\mu)$
\[
\esup_{U_{\sigma'}}{u^{-1}} \le \Big(\frac{C \nuNj(U_1)^{-1}
}{(\sigma-\sigma')^{\tau_0}}\Big)^{1/\gamma}
\|u^{-1}\|_{L_{\gamma}(U_\sigma)},\quad \delta\le \sigma'<\sigma\le
1,\; \gamma\in (0,1],
\]
where $C=C(\nu,\Lambda,\delta,\tau,\mu,N)$ and
$\tau_0=\tau_0(\mu,N)$. This shows that the first hypothesis of
Lemma~\ref{abslemma} is satisfied by any positive constant multiple
of $u^{-1}$ with $\beta_0=\infty$.

Take now $f_1=u^{-1}e^{c(u)}$ where $c(u)$ is the constant from
Theorem \ref{logest}. If we apply Theorem~\ref{logest} with $\tau:=2\tau$, $\eta:=\jd, $ $\delta:= \jd$ and $r:=2r$, then  $K_-=U'_1$ and $K_+=U_1$, and  since $\log
f_1=c(u)-\log u$, we see that
(\ref{logestright}) gives
\[
\nuNj(\{(t,x)\in U_1:\;\log f_1(t,x)>\lambda\})\le
M\Phi(2r)\nuN(B_r) \lambda^{-1}\le M\nuNj(U_1)\lambda^{-1},\quad \lambda>0
\]
where $M=M(\nu,\Lambda,\tau,\mu,N)$ and $0 <r \leq r^{*}(\mu)$. 
Hence we may apply
Lemma \ref{abslemma} with $\beta_0=\infty$ to $f_1$ and the family
$U_\sigma$; thereby we obtain
\[
\esup_{U_\delta} f_1\le M_1\] with
$M_1=M_1(\nu,\Lambda,\delta,\tau,\eta,\mu,N)$. In terms of $u$
this means that
\begin{equation} \label{HH1}
e^{c(u)}\le M_1\, \einf_{U_\delta} u.
\end{equation}

On the other hand, Theorem \ref{superest2} applied with $\eta:=\tau$ and $p_{0}:=p$  yields for  $0<r \leq r^{*}(\mu,p)$
\[
\|u\|_{L_{p}(U_{\sigma'}')}\le \Big(\frac{C\nuNj(U'_1)^{-1}
}{(\sigma-\sigma')^{\tau_1}}\Big)^{1/\gamma-1/p}
\|u\|_{L_{\gamma}(U'_\sigma)},\quad \delta\le \sigma'<\sigma\le 1,\;
0<\gamma\le p/\tilde{\kappa}.
\]
Here $C=C(\nu,\Lambda,\delta,\tau,\mu,N,p)$ and
$\tau_1=\tau_1(\mu,N,p)$. Thus the first hypothesis of Lemma
\ref{abslemma} is satisfied by any positive constant multiple of $u$
with $\beta_0:=p$, $\beta:=\gamma$ and $\eta:=1/\tilde{\kappa}$. Taking $f_2=u
e^{-c(u)}$ with $c(u)$ as above, we have $\log f_2=\log u-c(u)$
and so Theorem \ref{logest}, estimate (\ref{logestleft}) applied with the same parameters as earlier, gives
\[
\nuNj(\{(t,x)\in U'_1:\;\log f_2(t,x)>\lambda\})\le
M\nuNj(U'_1)\lambda^{-1},\quad \lambda>0,
\]
where $M$ is as above. 
Therefore we may again apply Lemma \ref{abslemma}, this time
to the function $f_2$ and the sets $U'_\sigma$, and with $\beta_0=p$
and $\eta=1/\tilde{\kappa}$; we get
\[
\|f_2\|_{L_p(U'_\delta)}\le M_2 \nuNj(U'_1)^{1/p},
\]
where $M_2=M_2(\nu,\Lambda,\delta,\tau,\eta,\mu,N,p)$. Rephrasing
then yields
\begin{equation} \label{HH2}
\nuNj(U'_1)^{-1/p}\|u\|_{L_p(U'_\delta)}\le M_2 e^{c(u)}.
\end{equation}
Finally, we combine (\ref{HH1}) and (\ref{HH2}) to the result
\[
\nuNj(U'_1)^{-1/p}\|u\|_{L_p(U'_\delta)}\le M_1 M_2\, \einf_{U_\delta}
u
\]
for  $0<r \leq r^{*}(\mu,p)$, which proves the assertion, because $\nuNj(U'_\delta) = \delta^{N+1} \nuNj(U'_1)$. \hfill $\square$


\subsection{Application to strong maximum principle}

Similarly as in the single-order fractional derivative case, the strong maximum principle for weak subsolutions of (\ref{MProb})
may be easily derived as a consequence of the weak Harnack
inequality.
\begin{satz} \label{strongmax}
Let $T>0$, and $\Omega\subset \iR^N$ be a bounded
domain. Suppose the assumptions (H1)--(H3) are satisfied. Let $u\in
Z$ be a weak subsolution of (\ref{MProb}) in $\Omega_T$ with $f\equiv 0$ and
assume that $0\le \esup_{\Omega_T}u<\infty$ and that $\esup_{\Omega}
u_0\le \esup_{\Omega_T}u$. Then, if for some cylinder
$Q=(t_0,t_0+\tau \vdr)\times B(x_0,r/2)\subset \Omega_T$ with
$t_0,\tau,r>0$ and $\overline{B(x_0,r)}\subset \Omega$ we have
\begin{equation} \label{strrel}
\esup_{Q}u \,=\,\esup_{\Omega_T}u,
\end{equation}
then the function $u$ is constant on $(0,t_0)\times \Omega$.
\end{satz}
\begin{proof} Let $M=\esup_{\Omega_T}u$. Then $v:=M-u$ is a
nonnegative weak supersolution of (\ref{MProb}) with $u_0$ replaced
by $v_0:=M-u_0\ge 0$. For any $0\le t_1< t_1+\eta \vdr<t_0$
the weak Harnack inequality with $p=1$ applied to $v$ yields an
estimate of the form
\[
r^{-N}(\vdr)^{-1}\int_{t_1}^{t_1+\eta
\vdr}\int_{B(x_0,r)}(M-u)\,dx\,dt\le C\,\einf_Q (M-u)\,=\,0.
\]
Hence $u=M$ a.e. in $(0,t_0)\times B(x_0,r)$ and the assertion now follows
by a chaining argument.
\end{proof}

\subsection{Weak Harnack inequality with inhomogeneity}
We finish this chapter with a simple proof of Theorem \ref{localweakHarnackinhomo}. We will later apply this result to deduce H\"older regularity of weak solutions to (\ref{holdeq}) (Theorem~\ref{holder}).
We recall the notation $\bvr:=\Phi(2r)$.


\begin{proof}[Proof of Theorem \ref{localweakHarnackinhomo}.]
If $u$ is a weak supersolution to (\ref{MProbG}) then it satisfies
\begin{equation} \label{BWFG}
\int_{0}^{2\tau
\bvr} \int_{B(x_{0},r)} \Big(-\eta_t [k\ast (u-u_0)]+
(ADu|D \eta)\Big)\,dxdt \geq \int_{0}^{2\tau
\bvr} \int_{B(x_{0},r)}\eta f dxdt
\end{equation}
for any nonnegative $\eta \in H^{1}_{2}((0,2\tau \bvr);L_{2}(B(x_{0}, r)))\cap L_{2}((0,2\tau \bvr);{\overset{\circ}{H}}{}^{1}_{2}(B(x_{0}, r)))$ with \m{$\eta_{|t=2\tau \bvr}=0$.} Let us denote $M = \|f^{-}\|_{L_{\infty}((0,2\tau \bvr)\times B(x_{0}, r))}$ and define $w = u + M 1*l$. Then
\[
\int_{0}^{2\tau
\bvr} \int_{B(x_{0},r)}\Big(-\eta_t [k\ast (w-u_0)]+
(ADw|D \eta)\Big)\,dxdt
\]
\[
= \int_{0}^{2\tau
\bvr} \int_{B(x_{0},r)} \Big(-\eta_t [k\ast (u-u_0)] - M \eta_t [k*1*l]+
(ADu|D \eta)\Big)\,dxdt
\]
\begin{align*}
& \geq \int_{0}^{2\tau
\bvr} \int_{B(x_{0},r)} \eta f dxdt + M\int_{0}^{2\tau
\bvr} \int_{B(x_{0},r)}\eta \cdot \partial_t(k*l*1)dxdt \\
& =\int_{0}^{2\tau
\bvr} \int_{B(x_{0},r)} \eta (M+f) dxdt \geq 0.
\end{align*}
Consequently $w\geq u \geq 0 $, and from  Theorem \ref{localweakHarnack} applied to the function $w$ we obtain
\[
 \Big(\frac{1}{\nuNj\big(Q_-(0,x_0,r, \delta)\big)}\,\int_{Q_-(0,x_0,r,\delta)}u^p\,d\nuNj\Big)^{1/p} \leq \Big(\frac{1}{\nuNj\big(Q_-(0,x_0,r, \delta)\big)}\,\int_{Q_-(0,x_0,r,\delta)}w^p\,d\nuNj\Big)^{1/p}
 \]
 \[
\le C \einf_{Q_+(0,x_0,r, \delta)} w \leq C \left( \einf_{Q_+(0,x_0,r, \delta)} u +M\int_{0}^{2\tau\bvr}l(s)ds \right),
\]
where $C=C(\nu,\Lambda,\delta,\tau,\mu,N,p)$. Applying estimate  (\ref{zn2}) with $p=1$ together with (\ref{estiPhieta}) we arrive at
\[
\int_{0}^{2\tau\bvr}l(s)ds \leq  C(\tau) r^{2},
\]
which finishes the proof.
\end{proof}

\section{Proof of the H\"older regularity}
In this chapter we prove Theorem \ref{holder}. We will deduce the H\"older regularity of weak solutions to (\ref{holdeq}) via the weak Harnack estimate. This is a novel approach compared to  \cite{Zhol}, where H\"older regularity for solutions to the problem with single order fractional time derivative was derived by means of growth lemmas without the use of Harnack inequalities. Here we provide a much simpler and shorter argument than the one in \cite{Zhol}.

Let $u$ be a bounded weak solution to
\eqq{
\partial_t (k*(u-u_{0}))-\mbox{div}\,\big(A(t,x)Du\big)=f \m{ in } (0,2\eta\bvr)\times B(x_{0},2r),
}{osc4}
where we assume that $r\in (0,r_{*}]$  and $r^{*}=r^{*}(\mu)$ comes from Theorem \ref{localweakHarnackinhomo} with $p=1$. Then $\bvr \leq 1$. We assume further that $u_{0} \in L_{\infty}(B(x_{0},2r)),f \in L_{\infty}((0,2\eta\bvr) \times B(x_{0},2r)).$
Again, in order to abbreviate the notation, we will often write $B_r(x)$ instead of $B(x,r)$. Moreover, in this section, by $l$ we denote an integer number, not to confuse with the kernel $l$ associated with the kernel $k$.

Let us define $F(t,x) = f(t,x) + u_{0}(x)k(t)$. We normalize the equation by putting
\eqq{
v:=\frac{u}{2D}, \hd G = \frac{F}{2D}, \hd D:=\|u\|_{L_{\infty}((0,2\eta\bvr)\times B_{2r}(x_{0}))} +  r^{2}\|F\|_{L_{\infty}((\frac{\eta}{2}\bvr,2\eta\bvr)\times B_{2r}(x_{0}))}.
}{osc3}
Then $v$ is a weak solution to
\[
\partial_t (k* v)-\mbox{div}\,\big(A(t,x)Dv\big)=G \m{ in } (0,2\eta\bvr)\times B(x_{0},2r)
\]
and
\eqq{
\|v\|_{L_{\infty}((0,2\eta\bvr)\times B_{2r}(x_{0}))} \leq \frac{1}{2}, \hd \|G\|_{L_{\infty}((\frac{\eta}{2}\bvr,2\eta\bvr)\times B_{2r}(x_{0}))}\leq \frac{1}{2r^{2}}, \hd \eosc_{(0,2\eta\bvr)\times B_{2r}(x_{0})} v \leq 1.
}{vges}

We point out that in (\ref{vges}) it is crucial to take the full working cylinder for the $L_{\infty}$-bound of $v$ and a sub-cylinder that has positive
distance from the points with time $t=0$ for the $G$-term.
Let $(t_{1},x_{1}) \in (\eta\bvr,2\eta\bvr) \times B_{r}(x_{0})$. We consider the family of nested cylinders with arbitrary, fixed $\theta > 0$
\[
Q(\rho) = (t_{1}-\theta \overline{\Phi}(\rho r),t_{1}) \times B_{\rho r}(x_{1}), \hd \rho = 2^{-l}, \hd l \in \mathbb{Z}
\]
and we denote  $Q_{dom}:=(0,2\eta\bvr) \times B_{2r}(x_{0})$. Further, let us denote by $\tilde{l} \geq 0$ the integer that corresponds to the largest of those cylinders $Q(2^{-l})$ that are properly  contained in  $Q_{dom}$. Then
\eqq{
|x_{1}-x_0|+2^{-\tl}r < 2r \m{\hd  and \hd  } t_{1}-\theta\overline{\Phi}(2^{-\tl}r) > 0.
}{alter2}

\no We will show that there exists $\gamma = \gamma(\theta,\eta) \geq 0$ such that
\eqq{
\tilde{l} \leq \gamma(\theta,\eta).
}{ltildees}
Indeed, if $\tl$ corresponds to the largest cylinder contained in $Q_{dom}$, it means that $Q(2^{-(\tl-1)})$ is not contained in $Q_{dom}$. Hence, either
\eqq{
|x_{1}-x_{0}|+2^{-(\tl-1)}r \geq  2r \m{\hd  or \hd  } t_{1}-\theta\overline{\Phi}(2^{-(\tl-1)}r) \leq  0.
}{alter1}
In the first case we get
\[
2^{-(\tl-1)} \geq  2-\frac{|x_{1}-x_{0}|}{r} > 1, \m{ thus } \tl< 1,
\]
while in the second case
\[
\overline{\Phi}(2^{-(\tl-1)}r) \geq  \frac{t_{1}}{\theta} > \frac{\eta \bvr}{\theta}.
\]
Applying Proposition (\ref{philambda}) we obtain
\[
\frac{\eta \bvr}{\theta} < \overline{\Phi}(2^{-(\tl-1)}r) \leq 2^{-2(\tl-1)}\overline{\Phi}(r)
\]
and thus

\[
\tl \leq 1 + \log_{4}\max\left\{\frac{\theta}{\eta},1 \right\},
\]
and we arrive at (\ref{ltildees}).

Let $l_{0}\geq \tl$ be an integer that will be fixed later. We set
\[
a_{l}:=\essinf_{Q_{dom}\cap Q(2^{-l})}v, \hd \hd  b_{l} = a_{l}+2^{-(l-l_{0})\kappa} \hd \m{ for } \hd l \leq l_{0},
\]
where $\kappa \in (0,1)$ will also be chosen later. Then, by definition, for all $j \leq l_{0}$, $j \in \mathbb{Z}$  we have
\eqq{a_{j} \leq v \leq b_{j} \hd \m{ a.e. in  } \hd Q_{dom}\cap Q(2^{-j}), \hd b_{j}-a_{j} = 2^{-(j-l_{0})\kappa},}{abclaim}
because $\eosc_{Q_{dom}} v \leq 1$ and $b_{j}-a_{j} \geq 1$ for $j \leq l_{0}$. We would like to prove the property (\ref{abclaim}) for $j > l_{0}$
with appropriate sequences $a_{j}$, $b_{j}$. More precisely,
we will show that for $j>l_{0}$ there exist a nondecreasing sequence  $ (a_{j})$  and nonincreasing sequence  $(b_{j})$  such that (\ref{abclaim}) holds. We will prove it by induction. Firstly, we note that for $j\leq l_{0}$ the condition (\ref{abclaim}) trivially holds. Now, let $l \geq l_{0}$ and  assume that (\ref{abclaim}) holds for all $j\leq l$. Under this assumption, we will construct $a_{l+1} \geq a_{l}$ and $b_{l+1}\leq b_{l}$ such that (\ref{abclaim}) holds  also for $j=l+1$.

Let us denote  $\tilde{t} = t_{1} - \theta \overline{\Phi}(2^{-l}r)$. Then $(\tilde{t},t_{1}) \times B(x_{1},2^{-l}r) = Q(2^{-l})$. At first we will show that the induction hypothesis implies the following estimate for
the memory term.
\begin{lemma}
Let $l \geq l_{0}$. Assume that (\ref{abclaim}) holds for all $j \leq l$, and set $m_{l} := \frac{a_{l}+b_{l}}{2}$. Then there exists $c(\theta) \geq 1$ such that
\eqq{
|v(t,x)-m_{l}| \leq (b_{l}-m_{l})\left(2\cdot 2^{\kappa}\left(c(\theta)\frac{\ki(t_{1}-\tilde{t})}{\ki(t_{1}-t)}\right)^{\frac{\kappa}{2}}-1\right) \m{ for a.a. } (t,x) \in (0,\tilde{t})\times B_{2^{-l}r}(x_{1}).
}{history}
\end{lemma}
\begin{proof}
Let $(t,x) \in (0,\tilde{t})\times B_{2^{-l}r}(x_{1})$. We note that $\frac{t_{1}-t}{\theta} \geq \frac{t_{1}-\tilde{t}}{\theta} = \overline{\Phi}(2^{-l}r)$. Since $\overline{\Phi}$ is increasing, continuous and onto $[0,\infty)$, there exists $l_{*} \leq l$ such that
\[
\overline{\Phi}(2^{-l_{*}}r) \leq \frac{t_{1}-t}{\theta} < \overline{\Phi}(2^{-(l_{*}-1)}r).
\]
Hence,
\[
t_{1}-\theta \overline{\Phi}(2^{-(l_{*}-1)}r) <t < t_{1},
\]
which together with $B_{2^{-l}r}(x_{1})\subset B_{2^{-(l_{*}-1)}r}(x_{1})$ implies $(t,x) \in Q(2^{-(l_{*}-1)})$. Since $l_{*}-1 < l$ we may apply the induction hypothesis to get
\[
v(t,x) - m_{l} \leq b_{l_{*}-1} - m_{l} \leq b_{l_{*}-1}-a_{l_{*}-1}+a_{l} - m_{l} = 2^{-(l_{*}-1-l_{0})\kappa} - \frac{1}{2} 2^{-(l-l_{0})\kappa}
\]
\[
=(b_{l}-m_{l})(2\cdot 2^{-(l_{*}-1-l)\kappa}-1).
\]
We recall that $t_{1}-\tilde{t} = \theta \overline{\Phi}(2^{-l}r)$ and $t_{1}-t \geq \theta \overline{\Phi}(2^{-l_{*}}r)$. Using this together with (\ref{zn1}) and the fact that $\ki$ is decreasing, we have
\eqq{
2^{-2(l_{*}-l)} = \frac{\frac{1}{4}2^{2l}r^{-2}}{\frac{1}{4}2^{2l_{*}}r^{-2}}
=\frac{\ki(\overline{\Phi}(2^{-l}r))}{\ki(\overline{\Phi}(2^{-l_{*}}r))} \leq c(\theta)
\frac{\ki(\theta\overline{\Phi}(2^{-l}r))}{\ki(\theta\overline{\Phi}(2^{-l_{*}}r))}
\leq  c(\theta) \frac{\ki(t_{1}-\tilde{t})}{\ki(t_{1}-t)},
}{twoes}
where $c(\theta)=\frac{\max\{1,\theta\}}{\min\{1,\theta\}}$. This way we obtain the upper bound. The lower bound we obtain analogously. Indeed,
\[
v(t,x) - m_{l} \geq a_{l_{*}-1} - m_{l} \geq a_{l_{*}-1}-b_{l_{*}-1}+b_{l} - m_{l} = -2^{-(l_{*}-1-l_{0})\kappa} + \frac{1}{2} 2^{-(l-l_{0})\kappa}
\]
\[
=-(b_{l}-m_{l})(2\cdot 2^{-(l_{*}-1-l)\kappa}-1).
\]
Making use of (\ref{twoes}) we arrive at the lower bound and thus the claim is proven.
\end{proof}
Now we will construct smaller cylinders inside $Q(2^{-l})$.
Let us introduce $\theta_{1},\theta_{2}$ such that $\frac{1}{4} < \theta_{1}<\theta_{2}<1$. We define the numbers
\[
t_{**} = t_{1}-\theta\theta_{2}\overline{\Phi}(2^{-l}r), \hd  t_{*} = t_{1}-\theta\theta_{1}\overline{\Phi}(2^{-l}r)
\]
and the cylinder
\eqq{
Q^{-} :=(t_{**},t_{*})\times B_{2^{-(l+1)}r}(x_{1}).
}{ab3}
We will show that
\eqq{
\tilde{t} < t_{**}<t_{*}<t_{1}-\theta\overline{\Phi}(2^{-(l+1)}r),
}{tstars}
which implies that $Q^{-}$ and $Q(2^{-(l+1)}) = (t_{1}-\theta\overline{\Phi}(2^{-(l+1)}r),t_{1}) \times B_{2^{-(l+1)}r}(x_{1})$ are disjoint and contained in $Q(2^{-l})$.
Only the last inequality in (\ref{tstars}) needs explanation. By the definition of $t_{*}$,  it is equivalent to
\[
\overline{\Phi}(2^{-(l+1)}r) < \theta_{1}\overline{\Phi}(2^{-l}r).
\]
Actually we will show that there exists $b \in (0,1)$ which depends only on $\theta_{1}$ such that
\eqq{
\overline{\Phi}(2^{-(l+1)}r) < b\theta_{1}\overline{\Phi}(2^{-l}r).
}{boxdist}
\nic{Since $\ki$ is decreasing the last one is equivalent with
\[
\ki(b\theta_{1}\overline{\Phi}(2^{-l}r)) < \ki(\overline{\Phi}(2^{-(l+1)}r)) = 2^{2l}r^{-2}.
\]
We note that
\[
\ki(b\theta_{1}\overline{\Phi}(2^{-l}r))  \leq (b\theta_{1})^{-1}\ki(\overline{\Phi}(2^{-l}r)) = (b\theta_{1})^{-1}\cdot 2^{2l}\frac{1}{4}r^{-2}.
\]
Since $\theta_{1} \in (\frac{1}{4},1)$, one may find $b \in (0,1)$ such that
\[
(b\theta_{1})^{-1}\cdot 2^{2l}r^{-2} < 4\cdot 2^{2l}r^{-2}
\]
i.e. for $b\in (1/4\theta_{1},1)$ the inequality (\ref{boxdist}) holds and we arrive at (\ref{tstars}).}
From Proposition \ref{philambda} we infer that
\[
\overline{\Phi}(2^{-(l+1)}r)  \leq \frac{1}{4}\overline{\Phi}(2^{-l}r)
\]
Since $\theta_{1} \in (\frac{1}{4},1)$, one may find $b \in (0,1)$ such that (\ref{boxdist}) holds.

Now, our strategy is to discuss two cases $(A)$ and $(B)$:
\eqq{
(A) \hd \nuNj(\{(t,x)\in Q^{-}:v(t,x) \leq m_{l}\}) \geq \frac{1}{2}\nuNj(Q^{-}),
}{A}
\eqq{
(B)\hd  \nuNj(\{(t,x)\in Q^{-}:v(t,x) \leq m_{l}\}) \leq \frac{1}{2}\nuNj(Q^{-}).
}{B}
In any case we will apply the weak Harnack inequality for a certain shifted problem with cylinders $Q_{-} \supseteq Q^{-}-t_{**}$ and $Q_{+} \supseteq Q(2^{-(l+1)})-t_{**}$, where $\cdot-t_{**}$ denotes the shift only in time variable. In this way we will obtain the required estimates in the cylinder $Q(2^{-(l+1)})$.

Suppose (A) holds. Set $w = b_{l}-v$. Then, by the induction hypothesis $w \geq 0$ on $Q(2^{-l})$.   Furthermore,
\eqq{
\partial_{t}(k*w)(t,x) - \divv (A(t,x)D w(t,x)) = b_{l}\cdot k(t) - G(t,x)
}{trzy1}
in a weak sense for $(t,x) \in (0,2\eta\bvr)\times B_{2r}(x_{0})$.
Let $t \in (t_{**},t_{1})$. We shift the time introducing $s=t-t_{**}$ and  $\ti{w}(s,x)= w(s+\tss,x)$. Then we may write
\[
(k*w)\hspace{0.05cm} \ti{} \hspace{0.05cm}  (s,x)= (k*w)(s+\tss,x) = \left(\int_{0}^{\ti{t}}+\int_{\ti{t}}^{\tss}+\int_{\tss}^{s+\tss} \right)k(s+\tss - \tau )w(\tau,x)d\tau
\]
\[
= \int_{0}^{\ti{t}}  k(s+\tss - \tau )w(\tau,x)d\tau +\int_{0}^{\tss - \ti{t}}  k(s+(\tss -\ti{t})- p )w(p+\tit,x )dp + \int_{0}^{s} k(s-p )\ti{w}(p,x)dp,
\]
where in the second integral we substitute $p:=\tau - \tit$ and in the third $p:=\tau - \tss$ and \m{$s\in (0, t_{1}-\tss)$.} Differentiating with respect to $s$ gives
\begin{align*}
\partial_{s}(k*w)  (s,x) & = \int_{0}^{\ti{t}}  \dot{k}(s+\tss - \tau )w(\tau,x)d\tau\\
& \quad \;+\int_{0}^{\tss - \ti{t}}  \dot{k}(s+(\tss -\ti{t})- p )w(p+\tit,x )dp + \partial_{s} (k*\ti{w})(s,x).
\end{align*}
Therefore, from (\ref{trzy1}) we get

\[
\partial_{s}(k*\tilde{w}) - \divv(\tilde{A}D\tilde{w}) = \izj \frac{\al}{\Gamma(1-\al)} \int_{0}^{t_{**}-\tilde{t}}(s+(t_{**}-\tilde{t})-p)^{-\al-1}w(p+\tilde{t},x)dp \dd
\]
\[
+ H(w)(s,x) + b_{l}k(t_{**}+s) - \tilde{G}(s,x), \hd \hd s \in (0,t_{1}-t_{**}),\,x\in B_{2r}(x_{0}),
\]
where
\[
H(w)(s,x) = \izj \frac{\al}{\Gamma(1-\al)} \int_{0}^{\tilde{t}}(s+t_{**}-p)^{-\al-1}w(p,x)dp\dd.
\]
Recall that $\tit = t_{1}- \theta \overline{\Phi}(2^{-l}r)$ and $\tss<t_{1}$, hence $(\tilde{t},t_{**}) \times B_{2^{-l}r}(x_{1}) \subset Q(2^{-l})$. Since $w \geq 0 $ on $Q(2^{-l})$ we  see that the first term on the RHS is nonnegative and we deduce that  $\tilde{w}$ satisfies in a weak sense
\begin{align}
\partial_{s}(k*\tilde{w})(s,x) -  \divv(\tilde{A}D\tilde{w})(s,x) & \geq H(w)(s,x) + b_{l}k(t_{**}+s) - \tilde{G}(s,x) \nonumber\\
& =:\Psi(s,x), \hd (s,x) \in (0,t_{1}-t_{**})\times B_{2^{-l}r}(x_{1}).
 \label{defPsi}
\end{align}
In particular, we have
\eqq{
\partial_{s}(k*\tilde{w}) - \divv(\tilde{A}D\tilde{w}) \geq -\Psi^{-} \hd \m{ in \hd } (0,t_{1}-t_{**})\times B_{2^{-l}r}(x_{1}),
}{nad}
where $\Psi=\Psi^{+}- \Psi^{-}$ and $\Psi^{-}\geq 0$ denotes the negative part of $\Psi$.


The nonnegativity of $w$ in $Q(2^{-l})$ implies that $\tw(s,x)=w(t_{**}+s, x)\geq 0$ on $\Qdl-t_{**}$. Furthermore,
\begin{align*}
\Qdl-t_{**} & =(t_{1}-\theta \vkrl - t_{**}, t_{1}-t_{**})\times \blrj\\
& =(\theta(\theta_{2}-1) \vkrl , \theta \theta_{2}\vkrl)\times \blrj .
\end{align*}
In particular, $\tw \geq 0 $ on $(0, t_{1}-t_{**})\times \blrj $. Thus, $\tw$ is nonnegative weak supersolution to (\ref{nad}) in $(0, t_{1}-t_{**})\times \blrj $. Therefore, we may apply Theorem~\ref{localweakHarnackinhomo} to $\tw$ with parameters: $u_{0}:=0$, $r:=2^{-l}r$, $\delta:=\frac{3}{4}$, $x_{0}:=x_{1}$, $p:=1$ and we obtain
\eqq{\frac{1}{\nuNj(Q_{-})}\int_{Q_{-}} \tw dxds \leq C\left[ \essinf_{Q_{+}} \tw + 2^{-2l}r^{2} \| \Psi^{-} \|_{L_{\infty}((0,2\tau \vkrl)\times \blrj)}
\right],
}{zharnacka}
where
\[
Q_{-}:=Q_{-}(0,x_{1}, \dml, \frac{3}{4} )= \left(0,\frac{3}{4} \tau \vkrl \right)\times B_{2^{-(l+1)}\cdot \frac{3}{2}r} (x_{1}),
\]
\[
Q_{+}:=Q_{+}(0,x_{1}, \dml, \frac{3}{4} )= \left(\frac{5}{4} \tau \vkrl, 2\tau \vkrl \right)\times B_{2^{-(l+1)}\cdot \frac{3}{2}r} (x_{1})
\]
and $C=C(\mu, \Lambda, \tau , \nu, N)$, provided $2\tau \vkrl \leq \theta \theta_{2}\vkrl $, i.e. $2 \tau  \leq \theta \theta_{2}$. We note that $\essinf_{Q_{+}}\tw = b_{l}- \esssup_{Q_{+}}\tv = b_{l}- \esssup_{t_{**}+Q_{+}} v$, hence
\eqq{
\frac{1}{\nuNj(Q_{-})} \int_{Q_{-}} \tw dxds \leq  C\left[ b_{l}-\esssup_{t_{**}+Q_{+}} v + 2^{-2l}r^{2} \| \Psi^{-} \|_{L_{\infty}((0,2\tau \vkrl)\times \blrj)}
\right].
}{ab1}
Now, we shall determine $\tau\in (0,\jd \theta\theta_{2}]$ and $\theta_{2}\in (\theta_{1},1)$ such that
\eqq{\Qdll \subseteq Q_{+}+t_{**}.}{zawieranie}
Since
\[
Q_{+}+t_{**}= \left(t_{1}- \theta\theta_{2} \vkrl  + \frac{5}{4} \tau \vkrl  , t_{1}- \theta\theta_{2}\vkrl+ 2\tau \vkrl \right)\times \bllrt,
\]
and
\[
\Qdll = (t_{1}- \theta\vkrll , t_{1})\times \bllrj,
\]
the inclusion (\ref{zawieranie}) holds, provided
\eqq{\theta\theta_{2}\vkrl - \frac{5}{4}\tau \vkrl \geq \theta \vkrll \hd \m{ and \hd } \theta\theta_{2}\leq 2\tau. }{pierwsza}
With regard to the second inequality we choose $\tau=\jd \theta\theta_{2}$ and
consequently we get the condition for $\theta_{2}$
\eqq{\frac{3}{8}\theta_{2} \vkrl \geq  \vkrll. }{nateta}
From Proposition (\ref{philambda}) we have
\[
\vkrll \leq \frac{1}{4}\vkrl.
\]
\nic{The last inequality is equivalent to
\[
\ki\left(\frac{3\theta_{2}}{8} \vkrl\right) \leq \ki\left(\vkrll \right)= 2^{2l}r^{-2}.
\]
The left-hand side of the above inequality may be estimated as follows
\[
\ki\left(\frac{3\theta_{2}}{8} \vkrl\right) =\int_{0}^{1}  \left( \frac{3\theta_{2}}{8} \right)^{-\al} \vkrl^{-\al} \dd
\]
\[
\leq \left( \frac{3\theta_{2}}{8} \right)^{-1} \int_{0}^{1}   \vkrl^{-\al} \dd = \frac{8}{3\theta_{2}}  \ki(\vkrl) = \frac{8}{3\theta_{2}} 2^{2(l-1)}r^{-2}.
\]
}
Therefore, (\ref{nateta}) is satisfied if we choose $\theta_{2}\in \left( \max\{  \theta_{1}, \frac{2}{3} \},1\right)$. Hence, we fix such $\theta_{2}$ and for $\tau = \jd \theta\theta_{2}$ we have (\ref{zawieranie}).
Combining (\ref{ab1}) and (\ref{zawieranie}) we arrive at the following estimate
\eqq{
\frac{1}{\nuNj(Q_{-})} \int_{Q_{-}} \tw \,dxds \leq  C\left[ b_{l}-\esssup_{\Qdll} v + 2^{-2l}r^{2} \| \Psi^{-} \|_{L_{\infty}((0,2\tau \vkrl)\times \blrj)}
\right],
}{ab2}
where $C=C(\mu,\Lambda, \nu,N,\theta,\theta_2)$.

Next, we would like to estimate from below the term on the LHS of (\ref{ab2}) by $(b_{l}-m_{l})$. We note that $\nuNj(Q_{-})=\left(\frac{3}{2}\right)^{N+1}\frac{\theta_{2}}{4(\theta_{2}-\theta_{1})}\nuNj(Q^{-})$ (see (\ref{ab3})) and from the assumption (A) we obtain
\begin{align*}
\jd \nuNj(Q_{-})& = \left(\frac{3}{2}\right)^{N+1}\frac{\theta_{2}}{4(\theta_{2}-\theta_{1})} \jd \nuNj(Q^{-})\\
&\leq  \left(\frac{3}{2}\right)^{N+1}\frac{\theta_{2}}{4(\theta_{2}-\theta_{1})} \nuNj(\{(t,x)\in Q^{-}:\hd v(t,x)\leq m_{l} \}).
\end{align*}
Thus,
\begin{align*}
\left(\frac{2}{3}\right)^{N+1}  \frac{2(\theta_{2}-\theta_{1})}{\theta_{2}}
& \nuNj(Q_{-}) \leq  \nuNj(\{(s,x)\in Q^{-}-\tss:\hd \tv (s,x)\leq m_{l} \})\\
& = \nuNj(\{(s,x)\in Q^{-}-\tss:\hd b_{l}-m_{l} \leq \tw (s,x) \})\\
& =  (b_{l}-m_{l})^{-1}
\int_{ \{(s,x)\in Q^{-}-\tss:\hd b_{l}-m_{l} \leq \tw (s,x) \}  } (b_{l}- m_{l})\, dxds \\
& \leq (b_{l}-m_{l})^{-1}
\int_{ Q^{-}-\tss  } \tw(s,x) \,dxds
\leq  (b_{l}-m_{l})^{-1}
\int_{  Q_{-}  } \tw(s,x)\, dxds,
\end{align*}
provided
\eqq{Q^{-}-\tss \subseteq Q_{-}.}{pop1}
The inclusion (\ref{pop1}) holds iff $\frac{5}{8}\theta_{2}\leq \theta_{1}$. Assuming further that $\theta_{1}$ and $\theta_{2}$ also  satisfy this last condition we  get
\eqq{\jd (b_{l}-m_{l})\leq \left(\frac{3}{2}\right)^{N+1} \frac{\theta_{2}}{4(\theta_{2}-\theta_{1})} \frac{1}{\nuNj(Q_{-})} \int_{Q_{-}} \tw dxds.}{ab4}
Combining (\ref{ab2}) and (\ref{ab4}) we obtain
\[
\jd (b_{l}-m_{l}) \leq  C\left[ b_{l}-\esssup_{\Qdll} v+ 2^{-2l}r^{2} \| \Psi^{-} \|_{L_{\infty}((0,2\tau \vkrl)\times \blrj)}
\right],
\]
so,
\eqq{
\esssup_{\Qdll } v \leq b_{l}+ 2^{-2l}r^{2} \| \Psi^{-} \|_{L_{\infty}((0,\theta\theta_{2} \vkrl)\times \blrj)}  - \frac{1}{2C} (b_{l}-m_{l}),
}{wazne2}
where $C=C(\mu, \Lambda, \nu , N, \theta, \theta_{1}, \theta_{2})$.

It remains to estimate the terms contained in $\Psi^{-}$, which was defined in (\ref{defPsi}). Using estimate (\ref{history}) we have
\[
H(w)(s,x) = \izj  \frac{\al}{\Gamma(1-\al)} \int_{0}^{\tilde{t}}(s+t_{**}-\tau)^{-\al-1}(b_{l}-v(\tau,x))d\tau\dd
\]
\[
= \izj  \frac{\al}{\Gamma(1-\al)} \int_{0}^{\tilde{t}}(s+t_{**}-\tau)^{-\al-1}[b_{l}-m_{l}+m_{l}-v(\tau,x)]d\tau\dd
\]
\[
\geq -(b_{l}-m_{l})\izj  \frac{\al}{\Gamma(1-\al)} \int_{0}^{\tilde{t}}(s+t_{**}-\tau)^{-\al-1}\left[2\cdot 2^{\kappa}\left(c(\theta)\frac{\ki(t_{1}-\tilde{t})}{\ki(t_{1}-\tau)}\right)^{\frac{\kappa}{2}}-2\right]d\tau \dd
\]
\[
= -(b_{l}-a_{l})\izj  \frac{\al}{\Gamma(1-\al)} \int_{0}^{\tilde{t}}(s+t_{**}-\tau)^{-\al-1}\left[ 2^{\kappa}\left(c(\theta)\frac{\ki(t_{1}-\tilde{t})}{\ki(t_{1}-\tau)}\right)^{\frac{\kappa}{2}}-1\right]d\tau \dd =\podst{p=\frac{t_{1}-\tau}{t
_{1}-\tilde{t}}}{dp = -\frac{1}{t_{1}-\tilde{t}}d\tau}
\]
\[
=-(b_{l}-a_{l})(t_{1}-\tilde{t})\izj  \frac{\al}{\Gamma(1-\al)} \int_{1}^{\frac{t_{1}}{t_{1}-\tilde{t}}}(t_{**}-t_{1}+s+p(t_{1}-\tilde{t}))^{-\al-1}\left[ 2^{\kappa}\left(c(\theta)\frac{\ki(t_{1}-\tilde{t})}{\ki(p(t_{1}-\tilde{t}))}\right)^{\frac{\kappa}{2}}-1\right]dp \dd
\]
\[
\geq -(b_{l}-a_{l})\izj  \frac{\al}{\Gamma(1-\al)} \int_{1}^{\frac{t_{1}}{t_{1}-\tilde{t}}}\left(\frac{t_{**}-t_{1}}{t_{1}-\tilde{t}}+p\right)^{-\al-1}(t_{1}-\tilde{t})^{-\al}[2^{\kappa}(c(\theta)p)^{\frac{ \kappa}{2}}-1]dp\dd,
\]
where in the last inequality we applied for $p \geq 1$
\[
\ki(p(t_{1}-\tilde{t})) = \izj  p^{-\al}(t_{1}-\tilde{t})^{-\al}\dd \geq p^{-1}\ki(t_{1}-\tilde{t}).
\]
We note that $\frac{t_{**}-t_{1}}{t_{1}-\tilde{t}} = -\theta_{2}$ and $t_{1}-\tilde{t} = \theta\overline{\Phi}(2^{-l}r)$. Thus we arrive at
\[
H(w)(s,x) \geq -(b_{l}-a_{l})\izj  \frac{\al}{\Gamma(1-\al)} \int_{1}^{\frac{t_{1}}{\theta\overline{\Phi}(2^{-l}r)}}\left(p-\theta_{2}\right)^{-\al-1}(\theta\overline{\Phi}(2^{-l}r))^{-\al}[2^{\kappa}(c(\theta)p)^{\frac{ \kappa}{2}}-1]dp\dd.
\]
We choose  $\gamma \in (0,\jd)$ such that $\int_{2\gamma}^{1}\ma d\al > 0$. Then we may decompose the integral as follows
\[
\izj  \frac{\al}{\Gamma(1-\al)} \int_{1}^{\frac{t_{1}}{\theta\overline{\Phi}(2^{-l}r)}}
\left(p-\theta_{2}\right)^{-\al-1}(\theta\overline{\Phi}(2^{-l}r))^{-\al}[2^{\kappa}(c(\theta)p)^{\frac{ \kappa}{2}}-1]dp\dd
\]
\[
=\left(\int_{\gamma}^{1} + \int_{0}^{\gamma}\right) \frac{\al}{\Gamma(1-\al)} \int_{1}^{\frac{t_{1}}{\theta\overline{\Phi}(2^{-l}r)}}\left(p-\theta_{2}\right)^{-\al-1}(\theta\overline{\Phi}(2^{-l}r))^{-\al}[2^{\kappa}(c(\theta)p)^{\frac{ \kappa}{2}}-1]dp\dd
=:I_{1}+I_{2}.
\]

In order to estimate $I_{1}$ we apply the inequality
\[
(p-\theta_{2})^{-\al-1}\leq (1- \theta_{2})^{-1} (p-\theta_{2})^{-\gamma-1} \hd \m{ for \hd } \al\in (\gamma, 1), \hd p\geq 1.
\]
Recalling (\ref{znk}), we obtain
\[
I_{1}\leq c(\theta_{2})\int_{\gamma}^{1} \frac{1}{\Gamma(1-\al)} (\theta\overline{\Phi}(2^{-l}r))^{-\al}\dd \int_{1}^{\infty}(p-\theta_{2})^{-\gamma-1}[(4c(\theta)p)^{\frac{ \kappa}{2}}-1]dp
\]
\[
\leq c(\theta_{2},\theta)\ki(\overline{\Phi}(2^{-l}r))\int_{1}^{\infty}(p-\theta_{2})^{-\gamma-1}[(4c(\theta)p)^{\frac{ \kappa}{2}}-1]dp=:\ki(\overline{\Phi}(2^{-l}r)) \ve_{1}(\theta,\theta_{2},\mu,\kappa)
\]
and $ \ve_{1} \rightarrow 0$ as $\kappa \rightarrow 0$ by the dominated convergence theorem.

To estimate $I_{2}$ we first notice that $\overline{\Phi}(r)\leq 1$, because $r\in (0,r_{*})$, hence, from conditions $l\geq l_{0}\geq \tilde{l}\geq 0$ we get  $\overline{\Phi}(2^{-l}r)\leq 1$. Thus, we have
\[
I_{2} \leq c(\theta)(\overline{\Phi}(2^{-l}r))^{-\gamma}\int_{0}^{\gamma} \frac{1}{\Gamma(1-\al)} \int_{1}^{\frac{t_{1}}{\theta\overline{\Phi}(2^{-l}r)}}\left(p-\theta_{2}\right)^{-\al-1}[2^{\kappa}(c(\theta)p)^{\frac{ \kappa}{2}}-1]dp\dd.
\]
We note that for $q\in (0,1)$ there holds
\[
q^{-\gamma} = q^{\gamma}q^{-2\gamma}\left(\int_{2\gamma}^{1}\frac{1}{\Gamma(1-\al)} \dd \right) \left(\int_{2\gamma}^{1}\frac{1}{\Gamma(1-\al)} \dd\right)^{-1}
\]
\[
\leq q^{\gamma} \left(\int_{2\gamma}^{1}q^{-\al}\frac{1}{\Gamma(1-\al)} \dd\right) \left(\int_{2\gamma}^{1}\frac{1}{\Gamma(1-\al)} \dd \right)^{-1}
\leq q^{\gamma} c(\mu)k(q).
\]
Applying this to $q=\overline{\Phi}(2^{-l}r)$ and recalling (\ref{znk}) we have
\[
I_{2} \leq c(\mu,\theta) (\overline{\Phi}(2^{-l}r))^{\gamma} \ki(\overline{\Phi}(2^{-l}r))\int_{0}^{\gamma}\frac{1}{\Gamma(1-\al)} \int_{1}^{\frac{t_{1}}{\theta\overline{\Phi}(2^{-l}r)}}\left(p-\theta_{2}\right)^{-\al-1}[2^{\kappa}(c(\theta)p)^{\frac{ \kappa}{2}}-1]dp\dd.
\]
We note that $t_{1}\leq 2\eta \overline{\Phi}(r)\leq 2\eta$, hence for  $p \leq \frac{2\eta}{\theta \overline{\Phi}(2^{-l}r)}$ we have $\overline{\Phi}(2^{-l}r)^{\gamma} \leq (\frac{2\eta}{\theta})^{\gamma}p^{-\gamma}$. Thus, we arrive at
\begin{align*}
I_{2} & \leq c(\mu,\theta,\eta)\ki(\overline{\Phi}(2^{-l}r))\int_{0}^{\gamma}\frac{1}{\Gamma(1-\al)} \int_{1}^{\frac{2\eta}{\theta\overline{\Phi}(2^{-l}r)}}p^{-\gamma}\left(p-\theta_{2}\right)^{-\al-1}[(4c(\theta)p)^{\frac{ \kappa}{2}}-1]dp\dd \\
& \leq c(\mu,\theta,\eta)\ki(\overline{\Phi}(2^{-l}r))\int_{0}^{\gamma}\frac{1}{\Gamma(1-\al)} \int_{1}^{\infty}\left(p-\theta_{2}\right)^{-\al-\gamma-1}[(4c(\theta)p)^{\frac{ \kappa}{2}}-1]dp\dd\\
& =:\ki(\overline{\Phi}(2^{-l}r)) \ve_{2}(\mu,\gamma,\theta,\theta_{2},\eta,\kappa),
\end{align*}
and $\ve_{2} \rightarrow 0$ as $\kappa \rightarrow 0$ by the dominated convergence theorem.

All in all we see that for $(s,x)\in (0,\theta\theta_{2} \vkrl)\times \blrj$
\eqq{
H(w)(s,x) \geq -(b_{l}-a_{l})k_1(\overline{\Phi}(2^{-l}r)) \ve(\mu,\theta,\theta_{2},\eta,\kappa) \m{ where } \ve \rightarrow 0 \m{ as } \kappa \rightarrow 0.
}{hw}
In view of (\ref{zn1})
\eqq{
4\cdot 2^{-2l}r^{2}\|H(w)^{-}\|_{L_{\infty}((0,\theta\theta_{2} \vkrl)\times \blrj)} \leq (b_{l}-a_{l})\ve(\mu,\theta,\theta_{2},\eta,\kappa).
}{estiHw}

To estimate $\Gf$ we recall that $\Gf(s,x)= G(s+\tss,x)$,\hd  $s\in (0,t_{1}-\tss)$, $x\in \blrj$. We have
\[
\| \Gf \|_{L_{\infty}((0,t_{1}-\tss)\times \blrj)}\leq \| G  \|_{L_{\infty}((\tilde{t},t_{1})\times \blrj)} = \| G \|_{L_{\infty}(Q(2^{-l}r))}.
\]
Having in mind (\ref{vges}) we would like to verify the following inclusion
\eqq{\Qdl=(\tilde{t}, t_{1})\times \blrj \subseteq \left(\frac{\eta}{2}\Pk(r) ,2\eta\Pkr \right)\times B_{2r}(x_{0}).}{zawieranie1}
Since $x_{1}\in B_{r}(x_{0})$ and $l\geq l_{0} \geq \lf \geq 0 $ we get $\blrj\subseteq B_{2r}(x_{0})$. Next, from the condition $t_{1}\in (\eta \Pkr, 2\eta \Pkr)$ we deduce that (\ref{zawieranie1}) holds if
\eqq{\theta\vkrl \leq \frac{\eta}{2} \Pkr.}{zaw2}
The parameters $\theta, \eta>0$ are fixed, so proceeding as earlier we deduce that there exists  $l_{0}\geq \lf$ sufficiently large such that (\ref{zaw2}) holds for $l\geq l_{0}$ and $l_{0}= l_{0}(\eta, \theta)$.
Consequently, from (\ref{vges}) and for $l_{0}$ as above we obtain
\eqq{\| \Gf\|_{L_{\infty} ((0,t_{1}-\tss)\times \blrj) }\leq \frac{1}{2r^{2}}.}{estigf}
\nic{

We recall that $|\tilde{G}|_{L_{\infty}(\hat{Q})} \leq \frac{1}{8r^{2}}$, hence
\[
4\cdot2^{-2l}r^{2}|\tilde{G}|_{L_{\infty}(\hat{Q})} \leq 2^{-2l-1}.
\]}

Concerning the $b_{l}$-term in $\Psi$, we firstly note that $l \geq l_{0}\geq \tilde{l}$ implies that
\[
t_{**}+s \geq \tss = t_{1}-\theta\theta_{2}\overline{\Phi}(2^{-l}r) \geq t_{1}-\theta\theta_{2}\overline{\Phi}(2^{-\tl}r) = \theta\overline{\Phi}(2^{-\tl}r)\left[\frac{t_{1}}{\theta\overline{\Phi}(2^{-\tl}r)}-\theta_{2}\right] \geq \theta\overline{\Phi}(2^{-\tl}r)[1-\theta_{2}],
\]
where the last inequality is a consequence of (\ref{alter2}). Next, $a_{l_{0}}\leq a_{l} \leq b_{l} \leq b_{l_{0}}$ implies $b_{l}\geq -|a_{l_{0}}| $, hence since $k$ is decreasing we may estimate as follows
\eqnsl{
b_{l}k(t_{**}+s) \geq -|a_{l_{0}}|k(t_{**}+s) \geq -|a_{l_{0}}|k\left(\theta\overline{\Phi}(2^{-\tl}r)[1-\theta_{2}]\right) \geq -|a_{l_{0}}| c(\theta,\theta_{2})k(\overline{\Phi}(2^{-\tl}r))
\\
\geq -\jd c(\theta,\theta_{2})k(\overline{\Phi}(2^{-\tl}r))\geq -\jd c(\theta,\theta_{2})\ki(\overline{\Phi}(2^{-\tl}r)) \geq -  \jd c(\theta,\theta_{2})2^{2\tl-2}r^{-2},
}{estibl}
where we used the normalization condition $|a_{l_{0}}| \leq \|v\|_{L_{\infty}(Q_{dom})} \leq \frac{1}{2}$  (see (\ref{vges})) and (\ref{znk}).

Finally, from (\ref{defPsi}), (\ref{estiHw}), (\ref{estibl}) and (\ref{estigf}) we have
\eqq{
2^{-2l}r^{2}\|\Psi^{-}\|_{L_{\infty}((0, \theta\theta_{2} \vkrl )\times \blrj)} \leq (b_{l}-a_{l})\ve(\mu, \theta, \eta, \kappa) + c(\theta,\theta_{2})2^{2(\tl-l)} + 2^{-2l-1}.
}{ab8}
Thus by (\ref{wazne2}) we obtain the following bound on the essential supremum of $v$ on a smaller cylinder
\eqq{
\esssup_{\Qdll}v \leq b_{l}-\frac{1}{4C}(b_{l}-a_{l}) + (b_{l}-a_{l})\ve(\mu, \theta, \eta, \kappa) + c(\theta,\theta_{2})2^{2(\tl-l)} + 2^{-2l-1}=:\beta_{l+1}.
}{estisup1}
We define
\[
a_{l+1} = a_{l}, \hd b_{l+1} = a_{l}+2^{-(l+1-l_{0})\kappa}.
\]
Hence, from (\ref{abclaim}) $b_{l+1}\leq b_{l}$ and  we have
\[
a_{l+1}=a_{l}\leq \essinf_{\Qdl} v \leq \essinf_{\Qdll} v \leq v \hd \m{ on \hd } \Qdll.
\]
Therefore, if $\beta_{l+1}\leq b_{l+1}$, then $v\leq b_{l+1}$ on $\Qdll$ and (\ref{abclaim}) holds for $j=l+1$ in case (A).
We will show that choosing $\kappa$ small enough and $l_{0}$ big enough the inequality $\beta_{l+1}\leq b_{l+1}$ holds. Indeed,
\[
\beta_{l+1} \leq b_{l+1} \iff (b_{l}-a_{l})\left(1-\frac{1}{4C}+\ve(\mu, \theta, \eta, \kappa)\right)+c(\theta,\theta_{2})2^{2(\tl-l)} + 2^{-2l -1} \leq 2^{-(l+1-l_{0})\kappa}.
\]
If we apply the induction hypothesis (\ref{abclaim}) for $j=l$, then the inequality  above is satisfied if
\[
2^{\kappa}\left(1-\frac{1}{4C}+\ve(\mu, \theta, \eta, \kappa)\right)+c(\theta,\theta_{2})2^{\kappa+ 2(\tl-l)+(l-l_{0})\kappa} + 2^{\kappa -2l -1+(l-l_{0})\kappa} \leq 1,
\]
i.e.
\[
2^{\kappa}\left(1-\frac{1}{4C}+\ve(\mu, \theta, \eta, \kappa)\right)+2^{\kappa-(l-l_{0})(2-\kappa) -2l_{0}}\left[ c(\theta,\theta_{2})2^{2\lf} + 2^{-1}\right] \leq 1.
\]
Since $l \geq l_{0}$, $\tilde{l}\leq \gamma(\theta,\eta)$ and $-(l-l_{0})(2-\kappa) \leq 0$ for $\kappa\in (0,2)$ it is enough to show that
\[
2^{\kappa}\left(1-\frac{1}{4C}+\ve(\mu, \theta, \eta, \kappa)\right)+2^{\kappa -2l_{0}}\left[ c(\theta,\theta_{2})2^{2\gamma(\theta,\eta)} + 2^{ -1}\right] \leq 1.
\]
We recall that here $C=C(\mu, \Lambda, \nu, N, \theta,\theta_{1}, \theta_{2})$.
\nic{
Then by induction hypothesis
\[
\beta_{l+1} \leq b_{l+1} \iff (b_{l}-a_{l})(1-\frac{1}{4C}+\ve)+\frac{1}{2}c(\theta,\theta_{2})2^{-2(l-\tl)} + 2^{-2l -1} \leq 2^{-(l+1-l_{0})\kappa}.
\]
\eqq{
\iff 2^{\kappa}(1-\frac{1}{4C}+\ve)+2^{\kappa-(l-l_{0})(2-\kappa)}(\frac{1}{2}c(\theta,\theta_{2})2^{-2(l_{0}-\tl)}+2^{-2l_{0}-1}) \leq 1.
}{lessone}
We note that
\[
2^{\kappa-(l-l_{0})(2-\kappa)}(\frac{1}{2}c(\theta,\theta_{2})2^{-2(l_{0}-\tl)}+2^{-2l_{0}-1}) = 2^{l(\kappa-2)-l_{0}\kappa+\kappa}(\frac{1}{2}c(\theta,\theta_{2})2^{2\tl}+\frac{1}{2})
\]
and
\[
l(\kappa-2)-l_{0}\kappa+2l_{0} = (l-l_{0})(\kappa-2) \leq 0.
\]
Since $l \geq l_{0}$ and $\tilde{l}\leq \gamma(\theta,\eta)$ it is enough to show that
\[
2^{\kappa}(1-\frac{1}{4C}+\ve)+2^{\kappa-2l_{0}}(c(\theta,\theta_{2})2^{2\gamma}+\frac{1}{2}) \leq 1.
\]}
In view of $\ve(\mu, \theta, \eta, \kappa)\rightarrow 0$ as $\kappa \rightarrow 0$ we may choose the para\-meter $\kappa=\kappa(\mu, \Lambda, \nu, N, \theta,\theta_{1}, \theta_{2}, \eta)$ small enough so that the first summand is smaller then $1-\frac{1}{8C}$. Having fixed $\kappa$ we may then choose $l_{0}\geq \tl$ so large that the second summand is smaller then $\frac{1}{8C}$. In this way we arrive at $\beta_{l+1}\leq b_{l+1}$, thus (\ref{abclaim}) is satisfied for $j=l+1$.

In the case (B) we proceed similarly.
Now, we consider $w=v-a_{l}$. Making the same shifts as before we arrive at
\begin{align*}
\partial_s (k*\tilde{w})(s,x) - \divv(\tilde{A}D\tilde{w})(s,x) &\geq H(w)(s,x) - a_{l}k(t_{**}+s) + \tilde{G}(s,x)\\
& =:\Psi(s,x), \hd (s,x) \in (0,t_{1}-t_{**})\times B_{2^{-l}r}(x_{1}).
\end{align*}
We apply Theorem~\ref{localweakHarnackinhomo} with the same parameters and sets to $\tilde{w}$ and we obtain   the analog of~(\ref{ab2})
\eqq{
\frac{1}{\nuNj(Q_{-})} \int_{Q_{-}} \tw dxds \leq  C\left[ \essinf_{\Qdll} v -a_{l}+ 2^{-2l}r^{2} \| \Psi^{-} \|_{L_{\infty}((0,2\tau \vkrl)\times \blrj)}
\right].
}{ab5}
Proceeding as earlier,  from the assumption (B) we obtain
\[
\jd \nuNj(Q_{-})= \left(\frac{3}{2}\right)^{N+1}\frac{\theta_{2}}{4(\theta_{2}-\theta_{1})} \jd \nuNj(Q^{-})\leq \left(\frac{3}{2}\right)^{N+1} \frac{\theta_{2}}{4(\theta_{2}-\theta_{1})}  \nuNj(\{(t,x)\in Q^{-}:\hd v(t,x)> m_{l} \}).
\]
Thus,
\begin{align*}
 \left(\frac{2}{3}\right)^{N+1} \frac{2(\theta_{2}-\theta_{1})}{\theta_{2}} & \nuNj(Q_{-}) \leq  \nuNj(\{(s,x)\in Q^{-}-\tss:\hd \tv (s,x)> m_{l} \})\\
& = \nuNj(\{(s,x)\in  Q^{-}-\tss:\hd m_{l}-a_{l} < \tw (s,x) \})\\
& =  (m_{l}-a_{l})^{-1}
\int_{ \{(s,x)\in Q^{-}-\tss:\hd m_{l}-a_{l} < \tw (s,x) \}  } (m_{l}-a_{l}) dxds \\
& \leq   (m_{l}-a_{l})^{-1}
\int_{  Q^{-}-\tss  } \tw(s,x)\, dxds
\leq   (m_{l}-a_{l})^{-1}
\int_{ Q_{-}  } \tw(s,x) \,dxds,
\end{align*}
where in the last inequality we applied (\ref{pop1}). Consequently, we get
\eqq{\jd (m_{l}-a_{l})\leq \left(\frac{3}{2}\right)^{N+1}\frac{\theta_{2}}{4(\theta_{2}-\theta_{1})} \frac{1}{\nuNj(Q_{-})} \int_{Q_{-}} \tw dxds.}{ab6}
From (\ref{ab5}) and (\ref{ab6}) we obtain
\[
\frac{1}{2}(m_{l}-a_{l}) \leq   C\left[ \essinf_{\Qdll} v -a_{l}+ 2^{-2l}r^{2} \| \Psi^{-} \|_{L_{\infty}((0,2\tau \vkrl)\times \blrj)}
\right],
\]
where $C=C(\mu,\Lambda, \nu, N, \theta,\theta_1,\theta_2)$. Thus,
\eqq{
   a_{l} + \frac{1}{2C}(m_{l}-a_{l}) - 2^{-2l}r^{2} \| \Psi^{-} \|_{L_{\infty}((0,2\tau \vkrl)\times \blrj)} \leq  \essinf_{\Qdll} v .
}{ab7}
Proceeding as in case $(A)$, we obtain the same estimate of the function $\Psi^{-}$ as in the previous case and we arrive at (\ref{ab8}). Hence,
\eqq{
\al_{l+1}:= a_{l} + \frac{1}{4C}(b_{l}-a_{l}) -  (b_{l}-a_{l})\ve(\mu, \theta, \eta, \kappa) - c(\theta,\theta_{2})2^{2(\tl-l)} - 2^{-2l-1} \leq \essinf_{\Qdll} v .
}{ab9}
In this case we set $b_{l+1}=b_{l}$ and $a_{l+1} = b_{l}-2^{-(l+1-l_{0})\kappa}$. Then using (\ref{abclaim}) for $j=l$ we get
\[
a_{l} = b_{l} - 2^{-(l-l_{0})\kappa} \leq b_{l} - 2^{-(l+1-l_{0})\kappa}= a_{l+1}
\]
and
\[
b_{l+1}-a_{l+1}= 2^{-(l+1-l_{0})\kappa}.
\]
Furthermore,
\[
v\leq \esssup_{\Qdll}{ v }\leq \esssup_{\Qdl} v \leq b_{l+1}=b_{l} \hd \m{ on } \hd \Qdll.
\]
Thus, if we show that $a_{l+1}\leq \al_{l+1}$, then  $a_{l+1}\leq v$ on $\Qdll$ and (\ref{abclaim}) holds for $j=l+1$ in case (B).  We note that
\[
a_{l+1} \leq \al_{l+1} \iff (b_{l}-a_{l})\left(1-\frac{1}{4C}+\ve(\mu, \theta, \eta, \kappa)\right)+c(\theta,\theta_{2})2^{2(\tl-l)} + 2^{-2l -1} \leq 2^{-(l+1-l_{0})\kappa},
\]
thus we obtain the same condition as in case (A). Proceeding further as in case (A) we deduce that (\ref{abclaim}) is satisfied for $j=l+1$ in case (B).

By the induction principle,  property  (\ref{abclaim}) holds for all $j \in \mathbb{Z}$ and
\eqq{
\eosc_{Q_{dom}\cap Q(2^{-j})} v \leq 2^{-(j-l_{0})\kappa} \hd \m{ for } \hd j\in \mathbb{Z}.
}{osc1}
Recalling that $u=2D v$ we get
\[
\eosc_{Q_{dom}\cap Q(2^{-j})}{u} \leq 2D 2^{-(j-l_{0})\kappa} = C_{H}2^{-j\kappa}D,
\]
where we denoted $C_{H}:=2^{l_{0}\kappa+1}$.
Hence, we arrive at the following oscillation estimate
\eqq{
\eosc_{ Q(2^{-j})}{u} \leq C_{H}2^{-j\kappa}D \hd \m{ for \hd }  j \geq\tl.
}{osc2}
We will make this estimate continuous by standard argument. We define the cylinders
\[
\tQ(\rho r) = (t_{1}-\theta\overline{\Phi}(\rho r),t_{1}) \times B(x_{1},\rho r) \hd \m{ for } \hd  \rho \in (0,\rho_{0}), \hd \rho_{0} = 2^{-\gamma(\theta,\eta)},
\]
where  $\gamma(\theta,\eta)$ comes from (\ref{ltildees}). Then there exists $j_{*} \geq \tilde{l}$ such that $2^{-(j_{*}+1)} < \rho \leq 2^{-j_{*}}$. Then from (\ref{osc2}) we get
\[
\eosc_{\tQ(\rho r)} u \leq \eosc_{Q(2^{-j_{*}} r)} u  \leq C_{H}2^{-j_{*}\kappa}D \leq C_{H}2^{\kappa}\rho^{\kappa}D=:\tilde{C}\rho^{\kappa}D,
\]
where $\tilde{C}=\tilde{C}(\mu, \Lambda, \nu, N, \eta, \theta)$. Let $\tilde{\rho} = \rho r$ and recall that $D$ were defined in (\ref{osc3}).  Then for every $(t_{1},x_{1}) \in (\eta\bvr,2\eta\bvr) \times B(x_{0},r)$ we have
\[
\eosc_{\tQ(\tilde{\rho})} u\leq \tilde{C}\left(\frac{\tilde{\rho}}{r}\right)^{\kappa}\left(\|u\|_{L_{\infty}((0,2\eta\bvr)\times B_{2r}(x_{0}))} + 4 r^{2}\|F\|_{L_{\infty}((\frac{\eta}{2}\bvr,2\eta\bvr)\times B_{2r}(x_{0}))}\right).
\]
We note that
\[
r^{2}\|F\|_{L_{\infty}((\frac{\eta}{2}\bvr,2\eta\bvr)\times B_{2r}(x_{0}))} \leq r^{2}\|f\|_{L_{\infty}((\frac{\eta}{2}\bvr,2\eta\bvr)\times B_{2r}(x_{0}))} + r^{2}\|u_{0}\|_{L_{\infty}(B(x_{0},2r))}k\left(\frac{\eta}{2}\bvr \right)
\]
\[
\leq r^{2}\|f\|_{L_{\infty}((\frac{\eta}{2}\bvr,2\eta\bvr)\times B_{2r}(x_{0}))} + c(\eta)\|u_{0}\|_{L_{\infty}(B(x_{0},2r))},
\]
where we used the estimate
\[
r^{2}k\left(\frac{\eta}{2}\bvr \right)  \leq c(\eta )r^{2}\ki\left(\vr \right) =c(\eta).
\]
To sum-up, we have proved that there exists $r^{*}=r^{*}(\mu )>0$ such that  for each  $r\in (0,r^{*})$, $\eta>0$ and $\theta>0$ there exists $\gamma=\gamma(\eta, \theta)>0$,   $\tilde{C}=\tilde{C}(\mu, \theta, \eta, \Lambda, \nu, N)$ and $\kappa=\kappa (\mu, \theta, \eta, \Lambda, \nu, N)$ such that for every weak solution of (\ref{osc4}), for every $(t_{1},x_{1})\in (\eta\bvr, 2\eta\bvr)\times B(x_{0},r)$
\eqnsl{\eosc_{(t_{1}- \theta \Pkro, t_{1} )\times B(x_{1},\rok)}{u} & \\
\leq \tilde{C}\left(\frac{\tilde{\rho}}{r}\right)^{\kappa}(\|u\|_{L_{\infty}((0,2\eta\bvr)\times B_{2r}(x_{0}))} +& r^{2}\|f\|_{L_{\infty}((\frac{\eta}{2}\bvr,2\eta\bvr)\times B_{2r}(x_{0}))} + \|u_{0}\|_{L_{\infty}(B(x_{0},2r))}),}{osc5}
where $\rok\in (0,\rho_{0})$, $\rho_{0}= 2^{-\gamma(\eta, \theta)}$.

In order to deduce from the oscillation estimate the H\"older continuity of weak solution we establish the following proposition.
\begin{prop}\label{estiprop}
Let $\gmb\in (0,1)$ be such that (\ref{intmugk}) is satisfied. If we denote
\[
c_{\mu}:=\min\left\{\left(\int_{\gmb}^{1} \dd \right)^{1/\gmb}, 1  \right\},
\]
then $c_{\mu} r^{\frac{2}{\gmb}}\leq \vr$ for $r\in (0,1)$.
\end{prop}
\begin{proof}
It is enough to show that $r^{-2}=\ki(\vr)\leq \ki(c_{\mu} r^{\frac{2}{\gmb}})$. To obtain this estimate we note that
\[
\ki(c_{\mu} r^{\frac{2}{\gmb}}) = \izj  c_{\mu}^{-\al}r^{-\frac{2\al}{\gmb}} \dd \geq \int_{\gmb}^{1}
 c_{\mu}^{-\al}r^{-\frac{2\al}{\gmb}} \dd \geq c_{\mu}^{-\gmb}r^{-2} \int_{\gmb}^{1}
\dd \geq r^{-2}.
\]
\end{proof}
Finally, using the above proposition we get
\eqq{\eosc_{(t_{1}- \theta c_{\mu} 2^{\frac{2}{\gmb}} {\rok}^{\frac{2}{\gmb}}, t_{1} )\times B(x_{1},\rok)}{u} \leq \eosc_{(t_{1}- \theta \Pkro, t_{1} )\times B(x_{1},\rok)}{u},}{osc6}
hence (\ref{osc5}) implies
\[
\sup_{\begin{array}{l} x_{1}\in B(x_{0},r) \\ t_{1} \in (\eta \bvr, 2 \eta \bvr)  \end{array}} \esssup_{\begin{array}{l} |t_{1}-t_{2}|< \theta c_{\mu} (2^{1-\gamma(\eta, \theta)})^{\frac{2}{\gmb}}  \\ |x_{1}-x_{2}|<2^{-\gamma(\eta, \theta)}\end{array}} \frac{|u(t_{1},x_{1})-u(t_{2},x_{2}) |}{\left(|t_{1}-t_{2}|^{\frac{\gmb}{2}}+|x_{1}-x_{2}| \right)^{\kappa} }
\]
\[
\leq \tilde{C}\left(\frac{1}{r}\right)^{\kappa}(\|u\|_{L_{\infty}((0,2\eta\bvr)\times B_{2r}(x_{0}))} + r^{2}\|f\|_{L_{\infty}((\frac{\eta}{2}\bvr,2\eta\bvr)\times B_{2r}(x_{0}))} + c(\eta)\|u_{0}\|_{L_{\infty}(B(x_{0},2r))}),
\]
where $\tilde{C}=\tilde{C}(\mu, \theta, \eta, \Lambda, \nu, N)$. Then, by a standard argument we deduce the H\"older continuity of weak solutions on the set $(\eta\bvr, 2\eta\bvr)\times B(x_{0},r)$.

To finish the proof of Theorem~\ref{holder}, for given subset $V\subset \Om_{T}$ separated from the parabolic boundary of $\Om_{T}$ it is enough to choose a finite  covering of $V$ consisting of sets of the form $(t_m+\eta\bvr, t_m+2\eta\bvr)\times B(x_{n},r)=:Q_{n,m}$, where $r\in (0,r^{*})$ and $\eta$ are sufficiently small. Then for $(t,x) \in Q_{n,m}$ we introduce the shifted time $s=t-t_m$ and set $\tilde{g}(s) = g(s+t_m)$ for $s \in (\overline{\Phi}(r)\eta,2\overline{\Phi}(r)\eta)$. Then $\tilde{u}(s,x)  = u(s+t_m,x)$ is a weak solution to
\eqq{
\partial_s (k*(\tilde{u}-u_0))(s,x) + \divv (\tilde{A}(s,x)D\tilde{u}(s,x)) = \tilde{f}(s,x) + \int_{0}^{t_m}\dot{k}(s+t_m-\tau)(u(\tau,x)-u_0(x))d\tau,
}{osc7}
even though $u_0$ is not the natural initial data for $\tilde{u}$ at $s=0$.
Since $u$ is bounded, the right-hand-side of (\ref{osc7}) is bounded on $Q_{n,m}$ and
\[
\norm{\int_{0}^{t_m}\dot{k}(\cdot+t_m-\tau)(u(\tau,\cdot) - u_0(x))d\tau }_{L_{\infty}((\frac{\eta}{2}\overline{\Phi}(r), 2\eta \overline{\Phi}(r))\times B(x_{n}, 2r) )}
\]
\[
\leq (\norm{u}_{L_{\infty}(\Omega_{T})} + \norm{u_0}_{L_{\infty}(\Omega)})k\left(\overline{\Phi}(r)\frac{\eta}{2}\right) \leq c(\eta)(\norm{u}_{L_{\infty}(\Omega_{T})} + \norm{u_0}_{L_{\infty}(\Omega)}) r^{-2}.
\]
Thus, $u$ is H\"older continuous on each  $Q_{n,m}$, and hence $u$ is H\"older continuous on $V$ and the estimate (\ref{holderkoniec}) holds.



\begin{thebibliography}{99}
{\footnotesize
\bibitem{Caf} Allen, M., Caffarelli, L., Vasseur, A.: A parabolic problem with a fractional time derivative. Archive for Rational Mechanics and Analysis {\bf 221} (2016), 603--630.
\bibitem{Laplace} Bobylev, A.V, Cercignani, C.: The Inverse Laplace Transform of Some Analytic Functions with an Application to the Eternal Solutions of the Boltzmann Equation.  Appl. Math. Lett. {\bf 15} (2002), 807–-813.
\bibitem{BomGiu} Bombieri, E., Giusti, E.: Harnack's inequality
for elliptic differential equations on minimal surfaces. Invent.
Math. {\bf 15} (1972), 24--46.
\bibitem{Phil1} Cl\'{e}ment, Ph.: On abstract Volterra equations in Banach spaces
with completely positive kernels. Infinite-dimensional systems
(Retzhof, 1983), 32--40, Lecture Notes in Math., {\bf 1076},
Springer, Berlin, 1984.
\bibitem{CN} Cl\'{e}ment, Ph.; Nohel, J.\ A.: Asymptotic behavior of
solutions of nonlinear Volterra equations with completely positive
kernels. SIAM J. Math. Anal. {\bf 12} (1981), 514--534.
\bibitem{CP} Cl\'{e}ment, Ph.; Pr\"uss, J.: Completely positive measures and
Feller semigroups. Math. Ann. {\bf 287} (1990), 73--105.
\bibitem{CZ} Cl\'{e}ment, Ph.; Zacher, R.: A priori estimates for
weak solutions of elliptic equations. Technical Report (2004),
Martin-Luther University Halle-Wittenberg, Germany.
\bibitem{DKSZ}
Dier, D.; Kemppainen, J.; Siljander, J.; Zacher, R.: On the parabolic Harnack inequality for non-local diffusion equations. Math. Z. {\bf 295} (2020), 1751--1769.


\bibitem{Grip1} Gripenberg, G.: Volterra integro-differential
equations with accretive nonlinearity. J. Differ. Eq. {\bf 60}
(1985), 57--79.
\bibitem{GLS} Gripenberg, G.; Londen, S.-O.; Staffans, O.:
{\em Volterra integral and functional equations.} Encyclopedia of
Mathematics and its Applications, {\bf 34}. Cambridge University
Press, Cambridge, 1990.
\bibitem{JPY}
Jia, J.; Peng, J.; Yang, J.:
Harnack's inequality for a space-time fractional diffusion equation and applications to an inverse source problem.
J. Differential Equations {\bf 262} (2017), 4415--4450.
\bibitem{decay} Kubica, A., Ryszewska K.: Decay of solutions to parabolic-type problem with distributed order  Caputo time derivative. J. Math. Anal. Appl. 465 (2018) 75-99.
\bibitem{nasza} Kubica, A., Ryszewska K.: Fractional diffusion equation with the generalized Caputo derivative. Journal of Integral Equations and Applications,Volume 31, Number 2, 2019.

\bibitem{KochDO} Kochubei, A.\ N.: Distributed order calculus and equations of
ultraslow diffusion. J. Math. Anal. Appl.\ {\bf 340} (2008),
252--281.
\bibitem{LSU}
Ladyzenskaja, O.\ A.; Solonnikov, V.\ A.; Uralceva, N.\ N.: {\em
Linear and quasilinear equations of parabolic type.} Translations of
Mathematical Monographs {\bf 23}, American Mathematical Society,
Providence, R.I. 1968.
\bibitem{Lm} Lieberman, G. M.: {\em Second order parabolic
differential equations.} World Scientific, London, 1996.
\bibitem{Metz} Metzler, R.; Klafter, J.: The random walk's guide to
anomalous diffusion: a fractional dynamics approach. Phys. Rep. {\bf
339} (2000), 1--77.
\bibitem{Moser64} Moser, J.: A Harnack inequality for parabolic
differential equations. Comm. Pure Appl. Math. {\bf 17} (1964),
101--134. Correction in Comm. Pure Appl. Math. {\bf 20} (1967),
231--236.
\bibitem{JanI} Pr\"uss, J.: {\em Evolutionary Integral Equations and
Applications}. Monographs in Mathematics {\bf 87}, Birkh\"auser,
Basel, 1993.
\bibitem{SalCoste} Saloff-Coste, L.: {\em Aspects of Sobolev-type
inequalities}. London Mathematical Society Lecture Note Series {\bf
289}, Cambridge University Press, 2002.
\bibitem{SCK}
Sokolov, I.\ M.; Chechkin, A. V., Klafter, J.: Distributed-order fractional kinetics. Acta Phys. Polon. B {\bf 35}
(2004), 1323–-1341.
\bibitem{Trud} Trudinger, N.\ S.: Pointwise estimates and quasilinear parabolic
equations. Comm. Pure Appl. Math. {\bf 21} (1968), 205--226.
\bibitem{VZd} Vergara, V.; Zacher, R.: A priori bounds for degenerate and singular evolutionary partial integro-differential equations. Nonlinear Analysis: Theory, Methods and Applications {\bf 73} (2010), 3572--3585.
\bibitem{VZ} Vergara, V.; Zacher, R.: Lyapunov functions and
convergence to steady state for differential equations of fractional
order. Math. Z. {\bf 259} (2008), 287--309.
\bibitem{Voller} Voller V. R., Falcini F., Garra R.: Fractional Stefan problems exhibiting lumped
and distributed latent-heat memory effects. Physical Review E, 87(4):042401, 2013.
\bibitem{Zhol} Zacher, R.: A De Giorgi–Nash type theorem for time fractional diffusion equations. Mathematische Annalen {\bf 356} (2013), 99--146.
\bibitem{base} Zacher, R.: A weak Harnack inequality for fractional evolution equations with discontinuous coefficients. Annali della Scuola Normale Superiore di Pisa-Classe di Scienze {\bf 12} (2013), 903--940.
\bibitem{Za2} Zacher, R.: Boundedness of weak solutions to evolutionary partial
integro-differential equations with discontinuous coefficients. J.
Math. Anal. Appl. {\bf 348} (2008), 137--149.
\bibitem{ZWH} Zacher, R.: Weak solutions of abstract evolutionary
integro-differential equations in Hilbert spaces. Funkcialaj
Ekvacioj {\bf 52} (2009), 1--18.



}
\end{thebibliography}
\end{document}